\documentclass[12pt]{dalthesis}

\usepackage{amsmath, amssymb}
\usepackage{enumerate}
\usepackage{enumitem}
\usepackage{amsthm}
\usepackage{thmtools}
\usepackage{mathtools}
\usepackage{bbm}
\usepackage{color}
\usepackage[all,cmtip,2cell]{xy}
\usepackage{fancyhdr}
\usepackage[noadjust]{cite}
\usepackage[linkbordercolor={0 1 1}]{hyperref}
\usepackage{mdwlist}
\usepackage{hyperref}
\usepackage{mathrsfs}
\usepackage{cleveref}
\usepackage{scalerel}
\usepackage{graphicx,amssymb}
\usepackage{tikz}
\usepackage{verbatim}
\usepackage{tikz}
\usepackage{titlesec}
\usepackage{indentfirst}

\input{diagxy}
\hypersetup{citebordercolor={0 1 1}}

\UseAllTwocells
\newtheorem{definition}{Definition}[section]
\newtheorem{proposition}[definition]{Proposition}
\newtheorem{theorem}[definition]{Theorem}
\newtheorem{corollary}[definition]{Corollary}
\newtheorem{conjecture}[definition]{Conjecture}
\newtheorem{lemma}[definition]{Lemma}
\declaretheoremstyle[bodyfont=\normalfont]{normalfont}
\theoremstyle{definition}
\newtheorem{example}[definition]{Example}
\theoremstyle{definition}
\newtheorem{remark}[definition]{Remark}

\newcommand{\haru}[1]{\ar[rr]|-{\vstretch{0.60}{|}} \ar[rr]^-{#1} }
\newcommand{\hard}[1]{\ar[rr]|-{\vstretch{0.60}{|}} \ar[rr]_-{#1} }
\newcommand{\srarrow}[2]{\xymatrixcolsep{0.35cm} \xymatrix{{#1} \ar[r] \ar[r]|{\vstretch{0.50}{|}} & {#2}}}

\makeatletter
\newcommand*{\@old@slash}{}\let\@old@slash\slash
\def\slash{\relax\ifmmode\delimiter"502F30E\mathopen{}\else\@old@slash\fi}
\makeatother

\begin{document}
\title{Cartesian Double Categories with an Emphasis on Characterizing Spans}
\author{Evangelia Aleiferi}

\phd

\degree{Doctor of Philosophy}
\degreeinitial{PhD}
\faculty{Science}
\dept{Department of Mathematics and Statistics}

\defencemonth{August}\defenceyear{2018}

\dedicate{For my parents, Pavlos and Stamatiki.}

\nolistoftables
\nolistoffigures

\frontmatter

\begin{abstract}
In this thesis, we introduce Cartesian double categories, motivated by the work of Carboni, Kelly, Walters, and Wood on Cartesian bicategories. Moving from bicategories to the slightly more generalized notion of double categories allows us to set the whole theory inside the welcoming 2-category of double categories, and to overcome technical problems that were caused by working with left adjoints inside a general bicategory. Cartesian double categories that are also fibrant are of particular interest to us. After describing some important properties of Cartesian and fibrant double categories, we give a characterization of the double category of Spans as a Cartesian double category. Lastly, we talk about profunctors and give a potential framework for their characterization as Cartesian double categories.
\end{abstract}

\begin{acknowledgements}
Firstly, I would like to thank Richard Wood and Geoff Cruttwell for their patience, motivation, and encouragement throughout this thesis. I couldn't have asked for better and kinder supervisors, thank you both!

I would also like to thank Bob Par\'{e}, for all of our discussions in the hallway outside our offices, as well as Dorette Pronk for her support throughout my studies. Many thanks to my fellow graduate students for five amazing years in the Department of Mathematics, to my undergraduate advisor Panagis Karazeris for inspiring me to get into Category Theory, and to Christina Vasilakopoulou for motivating me and helping me with this thesis from the other side of the continent.

Special thanks to my parents for their unlimited support throughout the years. I am really grateful for everything you have done for me! Also, many thanks to my brother Spyros, who has been an excellent example for me.

Finally, I would like to thank my dear friends in Halifax, Eliza, Shauna, Rowan, Emily, Chris, Mitch, Kirsti, Sarah, Emily, Shuv, PB, Izzy, Lew, and everybody else who made this place a beautiful home for me. Last but not least, special thanks to Noel, for his love and support this last year of my thesis. Thank you all!
\end{acknowledgements}

\mainmatter

\chapter{Introduction} \label{Introduction}

The first generalization of ordinary Cartesian categories came from Carboni and Walters in \cite{Carboni1987}. In that paper they introduced Cartesian bicategories in the case where the bicategories were locally ordered. The motivating example was that of the bicategory of relations, in which we can define an essentially unique, unital, associative, and symmetric up to coherent isomorphism tensor product. In the same paper they gave a characterization of the locally ordered bicategory $\mathbf{Ord}\mathcal{E}$ of ordered objects and ordered ideals as a Cartesian bicategory.\par

It was about twenty years later that a definition of a Cartesian bicategory was given in general. This was in the paper \cite{Carboni2008}, where Carboni, Kelly, Walters, and Wood define a Cartesian bicategory $\mathcal{B}$ as one that has finite products locally, its full sub-bicategory of left adjoints has finite bicategorical products, and certain derived lax functors $\otimes: \mathcal{B}\times \mathcal{B} \rightarrow \mathcal{B}$ and $I:\mathbf{1}\rightarrow \mathcal{B}$ have invertible constraints so that they are in fact pseudofunctors. Soon after that, in \cite{SpansLWW}, Lack, Walters, and Wood gave a characterization of the bicategory of spans as a Cartesian bicategory.\par

So it seemed that Cartesianness can be used as a base for characterizing important examples of bicategories. This was the idea behind the start of this thesis. The first example we considered was the one of profunctors internal to a finitely complete category $\mathcal{E}$ with reflexive coequalizers preserved by pullback functors, as in \cite{JohnstoneT}. The second was the bicategory of $V$-matrices over a Cartesian monoidal category $V$ with finite coproducts, such that the tensor distributes over them, as in \cite{Kelly2001}. Motivated by the fact that profunctors can be seen as modules over spans (\cite{Benabou}), we considered the bicategory of modules on a locally ordered bicategory first, and then on a general bicategory. In the first case, we showed that the bicategory of modules on a locally ordered Cartesian bicategory is Cartesian as well, by first showing that its Karoubi envelope (\cite{Borceux1986}) is Cartesian.

In the general case, a similar theorem would require showing that the full sub-bicategory of left adjoints in the bicategory of modules has finite products. However, describing this sub-bicategory is not an easy task. In particular, as far as we know, there is no characterization in the literature of the left adjoints in the bicategory of internal profunctors. Coming back to the example of $V$-matrices, one encounters the same problem. Having said that, it is not difficult to see that both internal profunctors and $V$-matrices share a similar property regarding this matter: there is a nice mere category with finite products embedded - not necessarily fully - in their sub-bicategory of left adjoints. For internal profunctors that was the category of internal categories and functors, and for $V$-matrices that was the category of sets and functions.\par

For this reason, we changed our setting from bicategories to double categories. Double categories were introduced in 1965 by Ehresmann (\cite{ehresmann65}), and have been further studied by various authors. Classic references are \cite{LimitsGP}, \cite{AdjointsGP}, or \cite{Shulman2008}. Double categories allow us to consider an extra type of arrows, whose composition is strictly associative and unital, i.e. they form a category. We will call that the vertical category of the double category. For our examples above, this consists of the internal functors and the functions respectively. Moreover, double categories together with double functors and vertical natural transformations live inside a strict 2-category. This means that we can consider adjunctions between double functors, which will lead to a very concrete definition of Cartesian double categories compared to the one for bicategories.\par

We prove many basic properties of Cartesian double categories, with a particular interest in the ones that are additionally fibrant. Fibrant double categories were discussed by Shulman in \cite{Shulman2010}, \cite{Shulman2008}, and prior to that in \cite{AdjointsGP}, as double categories with companions and conjoints. We then proceed to a characterization theorem of the double category of spans as a Cartesian double category. For that, we use ideas that appeared in \cite{SpansLWW} for the characterization of the bicategory of spans, together with some results on the double category of spans  in \cite{SpanCospan} by Niefield, and in \cite{Grandis2017} by Grandis and Par\'e.

In contrast to bicategories, it is fairly easy to prove that the double category of internal profunctors is a Cartesian double category. As in the case of locally ordered bicategories, we apply the construction of modules on the double category of spans. The construction of modules has been discussed briefly in \cite{Shulman2008} for fibrant double categories. This last result of the thesis can set a fundamental base for a characterization of the double category of internal profunctors in the future.

This thesis indicates that double categories might be a better enviroment for the generalization of Cartesian categories than bicategories. For that reason, we give a more detailed comparison between the two at the end of the thesis.\par

\vspace{1.3cm}

We assume that the reader is familiar with the basic theory of categories, as in \cite{MacLane}, as well as with the basic definitions and properties of 2-categories and bicategories, as in \cite{Benabou}, \cite{Gray1974}, \cite{Handbook}, or \cite{Companion}. The outline of the thesis is as follows:\par

In Chapter 2 we introduce the reader to the definition of a locally ordered Cartesian bicategory from \cite{Carboni1987}, and we prove that the Karoubi envelope (\cite{Borceux1986}) and the bicategory of modules (\cite{CKW}) retain the property of being Cartesian. We then give a brief overview of the paper \cite{Carboni2008}, and their definition of a general Cartesian bicategory. In the last two sections of this chapter we talk about modules, profunctors, $V$-matrices, and the problems we encountered when we tried to prove that they are Cartesian.\par

The first section of Chapter 3 is an introduction to double categories. The second and third sections include generalizations of the ideas that were introduced in \cite{Street1972} to double categories. The definitions for monads and comonads in double categories are the ones that Fiore, Gambino, and Kock gave in \cite{Fiore2011b}. However, we consider a simpler version of their definition for Eilenberg-Moore objects. The last section is devoted to fibrant double categories. The first part of this section consists of results that were proven by Shulman in \cite{Shulman2008}. However, we give full proofs for most of them, since we believe it will be helpful in understanding the rest of the thesis. In the second part of this section we study fibrant double categories for which the vertical category has finite products, and the others have finite products locally.

The main definitions and properties of Cartesian double categories are in Chapter 4. We start this chapter with the definition of precartesian double categories, on which we build the definition of Cartesian double categories in the following section. We give a characterization of Cartesian double categories that are also fibrant, and we study some examples. In the last section of this chapter we study some properties of Cartesian double categories with particular emphasis on fibrant Cartesian double categories. At this point, we also give the definition of a unit-pure double category, which will come in handy later in Chapter 5. 

In Chapter 5 we focus on the double category of spans. An important feature of this double category is that it has tabulators, which will prove useful for our characterization theorem. This is why we dedicate the first section of this chapter to study conditions that a Cartesian double category needs to satisfy in order to have tabulators. In the second section we show that the vertical category of a unit-pure double category with tabulators has pullbacks. We also review the construction of a double functor with domain the double category of spans on the vertical category of a given double category, which was introduced in \cite{SpanCospan}. In the last section we give conditions that a Cartesian double category needs to satisfy in order for this functor to be an equivalence.

In Chapter 6 we apply the structure of modules on what we have proved in the previous chapter for spans. This leads to the double category of profunctors, which we prove is Cartesian. We propose a possible characterization of it, based also on the characterization of the locally ordered bicategory of ordered objects and ordered ideals that Carboni and Walters gave in \cite{Carboni1987}.

Finally, in the conclusion we compare the two definitions, that is, the one for Cartesian double categories and the one for Cartesian bicategories.

\chapter{Cartesian Bicategories} \label{Cartesian Bicategories}

\section{Locally Ordered Cartesian Bicategories} \label{Locally Ordered Cartesian Bicategories}

In this section we review the first paper that was written on Cartesian Bicategories, \cite{Carboni1987}. In that paper, the authors focus on the example of sets together with relations or additive relations, and that of ordered objects and ideals. They notice that all of the above are equipped with a symmetric tensor product that satisfies specific properties and every object is a cocommutative comonoid object. They proceed to call this tensor product a Cartesian structure. This might be misleading at first, but later they show that we are not talking about an extra structure on the bicategory, but rather a property of it.  We present their definition here, but first we define what a tensor product on a bicategory is. For the following, consider a locally ordered bicategory $\mathcal{B}$, that is, a bicategory where every hom-category is a partially ordered set. Such a bicategory is actually a 2-category, since we assume that the order is antisymmetric. However, we would like to follow the terminology that was used in \cite{Carboni2008}.

\begin{definition}
A \textbf{tensor product} on $\mathcal{B}$ is a pseudofunctor $\otimes :\mathcal{B}\times \mathcal{B} \rightarrow \mathcal{B}$, together with an object $I$, called the identity object, and natural isomorphisms $$\rho:X\rightarrow X\otimes I,$$ $$ \lambda:X \rightarrow I\otimes X,$$ $$\gamma : X\otimes Y \rightarrow Y \otimes X \text{ and}$$ $$\alpha :(X \otimes Y)\otimes Z \rightarrow X\otimes (Y\otimes Z),$$ satisfying the following conditions.
\begin{enumerate}
\item $$\xymatrixcolsep{0.2cm} \xymatrix{  & (X \otimes Y) \otimes (Z \otimes W) \ar[dr]^{\alpha} \\
((X \otimes Y) \otimes Z) \otimes W \ar[ur]^{\alpha} \ar[d]_{\alpha \otimes W} &  & X \otimes (Y \otimes (Z \otimes W)) \\
(X \otimes (Y \otimes Z)) \otimes W \ar[rr]_{\alpha} & & X \otimes ((Y \otimes Z) \otimes W), \ar[u]_{X\otimes \alpha} } $$
\item $$ \xymatrixcolsep{0.2cm} \xymatrix{ & X \otimes Y \ar[dl]_{\rho \otimes Y} \ar[dr]^{X \otimes \lambda} \\
(X \otimes I) \otimes Y \ar[rr]_{\alpha} & & X\otimes (I \otimes Y)  }$$
\item $$ \xymatrixcolsep{1cm} \xymatrix{ (X \otimes Y) \otimes Z \ar[r]^{\alpha}\ar[d]_{\gamma \otimes Z} & X \otimes (Y \otimes Z) \ar[r]^{\gamma} & (Y \otimes Z ) \otimes X \ar[d]^{\alpha} \\
(Y \otimes X ) \otimes Z \ar[r]_{\alpha} & Y \otimes (X \otimes Z) \ar[r]_{Y \otimes \gamma} & Y \otimes (Z \otimes X) } $$
\item $$ \xymatrixcolsep{0.2cm} \xymatrix{ & Y \otimes X \ar[dr]^{\gamma} \\
X \otimes Y \ar[rr]_1 \ar[ur]^{\gamma} & & X \otimes Y.}$$
\end{enumerate}
\end{definition}

The definition that Carboni and Walters gave is the following.

\begin{definition}\label{CW} \cite{Carboni1987}
A \textbf{Cartesian structure} on $\mathcal{B}$ consists of:
\begin{enumerate}[label=\roman*.]
\item A tensor product $\otimes$ on $\mathcal{B}$.
\item On every $X$ in $\mathcal{B}$, a cocommutative comonoid structure
$$d_X : X \rightarrow X \otimes X \text{ and}$$ $$ t_X:X \rightarrow I,$$ meaning that the following diagrams commute:
\begin{enumerate}
\item $$\xymatrixcolsep{1cm} \xymatrix{ X \ar[rr]^{d_X} \ar[d]_{d_X} & & X\otimes X \ar[d]^{X \otimes d_X} \\
X \otimes X \ar[r]_{d_X \otimes X \quad \quad} & (X \otimes X) \otimes X \ar[r]_{\alpha} & X\otimes (X \otimes X) } $$
\item $$\xymatrixcolsep{1cm} \xymatrix{ & X \ar[dl]_{\rho} \ar[dr]^{\lambda} \ar[d]^{d_X}\\
X\otimes I & X\otimes X \ar[l]^{X \otimes t_X} \ar[r]_{t_X \otimes X }& I \otimes X } $$
\item $$ \xymatrixcolsep{0.4cm} \xymatrix{ & X \ar[dl]_{d_X} \ar[rd]^{d_X} \\
X \otimes X \ar[rr]_{\gamma} & & X\otimes X }$$
\end{enumerate}
\end{enumerate}
We ask for the following axioms to be satisfied:
\begin{enumerate}
\item For each arrow $F: X\rightarrow Y$,
$$\xymatrixcolsep{0.2cm} \xymatrixrowsep{0.2cm} \xymatrix{ X \ar[rr]^{d_X \quad} \ar[dd]_F & & X \otimes X \ar[dd]^{F \otimes F}\\
& \quad \le \\
Y \ar[rr]_{d_Y \quad} & & Y \otimes Y}$$
and
$$\xymatrixcolsep{0.2cm} \xymatrixrowsep{0.2cm} \xymatrix{ X \ar[ddd]_F \ar[dddrrr]^{t_X} \\
\\
& \le \quad \\
Y \ar[rrr]_{t_Y} & & & I, } $$
i.e. each $F$ is a colax comonoid homomorphism.
\item The arrows $d_X$ and $t_X$ have right adjoints $d_X^*$ and $t_X^*$. That is, there are inequalities $1_X \le d_X^* d_X$,  $d_X d_X^* \le 1_{ X\otimes X}$, $1_X \le t_X^* t_X$ and $t_X t_X^* \le 1_I $.
\end{enumerate}
\end{definition}

In the above definition and throughout this chapter, $\le$ represents the partial order on the hom-categories of $\mathcal{B}$. One of the main examples in \cite{Carboni1987} is the bicategory of relations.

\begin{example}\label{relations}
Consider a regular category $\mathcal{E}$. That is a category such that:
\begin{enumerate}[label=\roman*.]
\item All finite limits exist.
\item For every morphism $f:d\to c$ in $\mathcal{E}$ and its pullback,
\begin{displaymath} \xymatrix{
d \times_c d \ar[r]^{p_1} \ar[d]_{p_2} & d \ar[d]^{f} \\ d \ar[r]_f & c } \end{displaymath}
the coequalizer of $p_1$ and $p_2$, exists in $\mathcal{E}$. The pair
$\xy \morphism(0,0)|a|/@{>}@<3pt>/<400,0>[d\times_c d`d;p_1]
\morphism(0,0)|b|/@{>}@<-3pt>/<400,0>[d\times_c d`d;p_2]\endxy$
is called the \textbf{kernel pair} of $f$.
\item For every regular epimorphism $f : d\to c$, i.e. for every $f : d\to c$ which can be expressed as the coequalizer of some parallel pair, its pullback along any morphism is a regular epimorphism.
\end{enumerate}
 Then the bicategory of relations $\mathbf{Rel}\mathcal{E}$ in a regular category $\mathcal{E}$ consists of the objects $A,B, \dots$ of $\mathcal{E}$, together with relations in $\mathcal{E}$ and morphism between them: A relation $r:A \rightarrow B$ is a span
\begin{displaymath} \xymatrix{
A & R \ar[l]_{r_0} \ar[r]^{r_1} & B } \end{displaymath}
such that the arrow $(r_0,r_1) : R \rightarrow A\times B$ is monic. For the composition we use again the pullbacks, with the restriction that we need monic arrows. For that we use the fact that in a regular category, every arrow can be factored uniquely up to isomorphism as a composite of a regular epi with a monic. So given the pullback of such two relations, we can factor it as a regular epi followed by a monic span, and the latter is the composite we want. The bicategory $\mathbf{Rel}\mathcal{E}$ is Cartesian.
\end{example}

In the following two lemmas is where we will see that the above structure looks more like a property.

\begin{lemma}\label{lemma_fox}\cite{Carboni1987}
If $\mathcal{B}$ is a locally ordered bicategory with tensor product then the tensor product is the bicategorical product in $\mathcal{B}$ (\Cref{biproducts}) if and only if every object has a cocomutative comonoid structure $(X,d_X,t_X)$ and every arrow $F:X \rightarrow Y$ is a comonoid homomorphism, meaning that the following diagrams commute.
$$\xymatrix{ X \ar[d]_F \ar[r]^{d_X \quad} & X\otimes X \ar[d]^{F\otimes F} \\ Y \ar[r]_{d_Y \quad} & Y\otimes Y}$$
$$\xymatrix{ X \ar[d]_F \ar[dr]^{t_X}\\ Y \ar[r]_{t_Y} & I}$$
\end{lemma}

For a locally ordered bicategory $\mathcal{B}$ we can consider the full sub-bicategory $\mathbf{Map(\mathcal{B})}$ consisting of the arrows that have a right adjoint. We call these arrows \textbf{\emph{maps}}.

\begin{lemma}\label{lemma_Map_products}\cite{Carboni1987}
If $\mathcal{B}$ is a locally ordered Cartesian bicategory then $\mathbf{Map(\mathcal{B})}$ has finite products.
\end{lemma}
\begin{proof}
By \Cref{lemma_fox} and since $d_X$ and $t_X$ are maps, the tensor product will be the product in $\mathbf{Map (\mathcal{B})}$ if every map is a comonoid homomorphism. Let $f:X\rightarrow A$ be a map. We know that $f$ and $f^*$ are lax comonoid homomorphisms. The latter means that $d_X f^* \le (f^* \otimes f^*) d_A$ and $t_X f^* \le t_A$. Since $\otimes$ is a pseudofunctor, $(f\otimes f)^* = f^* \otimes f^*$, so $d_X f^* \le (f\otimes f)^* d_A$. By adjunction we have $d_X f^* \le (f \otimes f)^* d_A \Leftrightarrow d_X \le (f \otimes f)^* d_A f \Leftrightarrow (f \otimes f) d_X \le d_A f$ and $t_X f^* \le t_A \Leftrightarrow t_X \le t_A f$. So $d_A f = (f \otimes f) d_X$ and $t_A f = t_X$, i.e. $f$ is a comonoid homomorphism.
\end{proof}

\begin{corollary}\cite{Carboni1987}
If $\mathcal{B}$ is a locally ordered Cartesian bicategory then the only comonoid structure on $X$ with structure arrows having right adjoints is $(X,d_X,t_X)$.
\end{corollary}
\begin{proof}
Suppose that $(X, \delta, \tau)$ is another comonoid structure on $X$ with right adjoints $\delta^*$, $\tau^*$. We showed in \Cref{lemma_Map_products} that $\mathbf{Map(\mathcal{B})}$ has finite products. Since $I$ is the terminal, we have $\tau=t_X$. Also the diagram
$$\xymatrix{ X \ar[rr]^{d_X} \ar[dd]_{\delta} \ar[dr]^{\rho} & & X\otimes X \ar[dl]^{X\otimes t_X} \ar[dd]^{p} \\
& X \otimes I \ar[dr]_p \\
X\otimes X \ar[ur]^{X \otimes t_X} \ar[rr]_p & & X }$$
commutes and similarly $r \delta = r d_X$. So $\delta = d_X$.
\end{proof}

The above lemmas lead to the theorem below. This theorem leads to the second paper on Cartesian bicategories since it ensures that a locally ordered bicategory has a Cartesian structure if and only if it is Cartesian as defined later in \cite{Carboni2008} by Carboni, Kelly, Walters and Wood. 

\begin{theorem} \label{CW_1.6}\cite{Carboni1987}
If $\mathcal{B}$ has a Cartesian structure then:
\begin{enumerate}[label=\roman*.]
\item $\mathbf{Map(\mathcal{B})}$ has finite products. In particular, the tensor product is the product on $\mathbf{Map(\mathcal{B})}$. We will denote this product by $\times$ and its projections by $p$ and $r$. The identity $I$ plays the role of the terminal.
\item Each hom-category $\mathcal{B}(X,A)$ has finite products. We will use $\wedge$ and $\top$ for this product and terminal respectively.
\item For any arrows $F$ and $G$ in $\mathcal{B}$, $$F \otimes G = (p^* F p)\wedge (r^* Gr).$$ Moreover $\top_{I,I}=1_I$.
\end{enumerate}
Conversely, if $\mathbf{Map(\mathcal{B})}$ has finite products, each $\mathcal{B}(X,A)$ has finite products and the formula in iii. defines a functorial tensor product on $\mathcal{B}$, then $\mathcal{B}$ has a Cartesian structure.
\end{theorem}

\section{The Karoubi Bicategory of  Locally Ordered Cartesian Bicategories} \label{The Karoubi Bicategory of  Locally Ordered Cartesian Bicategories}

The Karoubi envelope is a special case of the Cauchy completion of a category which was first introduced in \cite{Metric}. In this section we prove that the Karoubi envelope of a locally ordered Cartesian bicategory is Cartesian as well which, as far as we know, does not exist in the literature so far. This will help us later prove that the bicategory of modules over monads in a locally ordered Cartesian bicategory is Cartesian too.

\begin{definition}\cite{Borceux1986}
For a locally ordered bicategory $\mathcal{B}$, we define its \textbf{Karoubi bicategory} to be the bicategory $\mathbf{Kar(\mathcal{B})}$ that consists of the following data:
\begin{enumerate}[label=\roman*.]
\item An object in $\mathbf{Kar(\mathcal{B})}$ is a pair $(A,a)$, where $A$ is an object in $\mathcal{B}$ and $a$ an idempotent on $A$, i.e. an arrow $a:A \rightarrow A$ such that $aa=a$.
\item If $(X,x)$ and $(A,a)$ are idempotents then an arrow between them is an arrow $R:X \rightarrow A$ in $\mathcal{B}$ such that $a R x = R$. Such an arrow will be called a module between idempotents.
\item 2-cells in $\mathbf{Kar(\mathcal{B})}$ are inequalities, as in $\mathcal{B}$.
\item For an idempotent $(X,x)$, the identity on it is the arrow $x$ itself.
\item The composition of arrows is the composition of arrows in $\mathcal{B}$.
\end{enumerate}
\end{definition}

\begin{remark}
\begin{enumerate}
\item Note that if $a R x = R$ then $a R = aaRx = aRx = R = aRx = aRxx=R x$ and conversely, if $aR=R=Rx$ then $aRx = R$.
\item If $(X,x)$ and $(A,a)$ are idempotents, then every arrow $F:X \rightarrow A$ gives an arrow $aFx:(X,x) \rightarrow (A,a)$ between idempotents, since $aaFxx=aFx$.
\end{enumerate}
\end{remark}

\begin{example}
Consider the locally ordered bicategory $\mathbf{Rel}$. An idempotent relation on a set $X$ is a transitive relation $< : \srarrow{X}{X}$ for which $x < y \Rightarrow (\exists z)(x <z <y)$. An arrow $R:\srarrow{(X,<_X)}{(A,<_A)}$ is a relation $R:\srarrow{X}{A}$ with $R <_X \; = R \; =\; <_A R$. That is, $(\exists y)(aRy \text{ and } y<x) \; \text{ iff }\; aRx \; \text{ iff }\; (\exists b)(a<b \text{ and } bRx)$.
\end{example}

\begin{lemma}\label{lemma_sharp_adj}
If $(X,x)$ and $(A,a)$ are idempotents and $F \dashv F^*$, $F:X \rightarrow A$ is an adjunction in $\mathcal{B}$ such that
$$\xymatrixcolsep{0.4cm} \xymatrixrowsep{0.4cm} \xymatrix{ X \ar[rr]^F \ar[dd]_x & & A \ar[dd]^a\\
& \le \\
X \ar[rr]_F & & A, }$$
then we have an adjunction $aFx \dashv xF^*a$ in $\mathbf{Kar(\mathcal{B})}$.
\end{lemma}
\begin{proof}
By the adjunction $F \dashv F^*$ we get $1_X \le F^* F $ and $FF^* \le 1_A$. The unit and counit of the adjunction $aFx \dashv xF^* a$ will be given by the following inequalities:
$$1_{(X,x)} = x=x1_X x \le x F^* F x \le x F^* a Fx = (xF^*a)(aFx) \text{ and}$$
$$(aFx)(xF^*a) \le aaFxF^*a = aFx F^* a \le aaFF^* a \le a1_A a =a = 1_{(A,a)}.$$
\end{proof}

\begin{lemma}\label{lemma_sharp_iso}
Consider idempotents $(X,x)$ and $(A,a)$ and an isomorphism $R:X \rightarrow A$ in a locally ordered bicategory $\mathcal{B}$. If $Rx =aR$, then $aRx :(X,x) \rightarrow (A,a)$ is an isomorphism in $\mathbf{Kar(\mathcal{B})}$.
\end{lemma}
\begin{proof}
The inverse of $aRx$ is the module $x R^{-1} a$ : $$(aRx)(x R^{-1} a) = a R x R^{-1} a =a a R R^{-1} a =a 1_A a =a=1_{(A,a)} $$
$$(x R^{-1} a)(aRx) = x R^{-1} a R x =x R^{-1} R x x = x1_X x = x = 1_{(X,x)}.$$
\end{proof}

\begin{definition}
Define $\mathbf{LoBicat}$ to be the 2-category with objects the locally ordered bicategories, arrows the pseudofunctors between them, and 2-cells the pseudonatural transformations.
\end{definition}

To show that the Karoubi envelope of a locally ordered Cartesian bicategory is Cartesian, we will build a 2-functor from $\mathbf{LoBicat}$ to itself. To start, we prove the following lemma:

\begin{lemma}\label{lemma_sharp_functor}
For locally ordered bicategories $\mathcal{B}$ and $\mathcal{D}$, consider a pseudofunctor $F:\mathcal{B} \rightarrow \mathcal{D}$. Then the following mapping produces a functor $F_\# : \mathbf{Kar(\mathcal{B})} \rightarrow \mathbf{Kar(\mathcal{D})}$:
$$(A,a) \mapsto (FA,Fa) \text{, for an idempotent } (A,a) \text{ and}$$
$$ R : (X,x) \rightarrow (A,a) \mapsto FR : (FX, Fx) \rightarrow (FA,Fa) \text{, for a module } R.$$
\end{lemma}
\begin{proof}
$F_\#$ is well defined since $Fa Fa = F aa =Fa$, which means that $(FA,Fa)$ is an idempotent and $Fa FR Fx = FaRx=FR$, which shows that $FR$ is a module between idempotents. Also the following relations for a module $R:(X,x) \rightarrow (A,a)$ imply that $F_\#$ is indeed a functor:
$$ R \le S \Rightarrow FR \le FS,$$
$$F_\# S \: F_\# R = FS \: FR = FSR = F_\# SR \text{ and}$$
$$1_{F_\# (A,a)} = 1_{(FA,Fa)} = Fa = F_\# a = F_\# 1_{(A,a)}.$$
\end{proof}

\begin{lemma}\label{lemma_sharp_transf}
For two locally ordered bicategories $\mathcal{B}$ and $\mathcal{D}$, consider a pseudonatural transformation $\phi :F \Rightarrow G : \mathcal{B} \rightarrow \mathcal{D}$. Then we can define a natural transformation $\phi_\# :F_\# \Rightarrow G_\# :  \mathbf{Kar(\mathcal{B})} \rightarrow \mathbf{Kar(\mathcal{D})}$ with components $\phi_\# (A,a) : Ga \: \phi A\: Fa$. Moreover, $\phi_\#$ is an isomorphism if $\phi$ is an isomorphism.
\end{lemma}
\begin{proof}
The components $\phi_\#(A,a):(FA,Fa) \rightarrow (GA,Ga)$ are modules since
$$Ga\:Ga \:\phi A\: Fa\:Fa = Gaa\: \phi A \:F aa = Ga\: \phi A \:F a.$$
Also, $\phi_\#$ becomes a pseudonatural transformation by considering the inequalities $$ G_\# R \: \phi_\# (X,x) = GR \: Gx \: \phi X \: Fx = GRx \: \phi X \: Fx = $$
$$ GaR \: \phi X \: Fx = Ga \: GR \: \phi X \: Fx \le Ga \: \phi A \: FR \: Fx = Ga \: \phi A \: F Rx = $$
$$Ga \: \phi A \: FaR = Ga \: \phi A \: Fa \: FR = \phi_\# (A,a) \: F_\# R.$$
If now $\phi$ is an isomorphism, then by naturality we have $\phi A \: Fa \le Ga \: \phi A$ and $Ga \: \phi A = \phi A \: \phi^{-1} A \: Ga \: \phi A \le \phi A \: F a \: \phi^{-1} A \: \phi A = \phi A \: F a$, so by \Cref{lemma_sharp_iso}, $\phi^{-1}_\#(A,a)$ is the inverse of $\phi_\# (A,a)$.
\end{proof}

\begin{proposition}
There is a (strict) 2-functor $\--_{\#} : \mathbf{LoBicat} \rightarrow \mathbf{LoBicat}$ that maps each locally ordered bicategory $\mathcal{B}$ to its Karoubi bicategory $\mathcal{B}_{\#} := \mathbf{Kar(\mathcal{B})}$, each pseudofunctor $F$ to $F_\#$, and each pseudonatural transformation $\phi$ to $\phi_{\#}$.
\end{proposition}
\begin{proof}
Consider two locally ordered bicategories $\mathcal{B}$ and $\mathcal{D}$. Then the mapping $F \mapsto F_{\#}$ and $\phi \mapsto \phi_{\#}$ will give a functor $\-- _{\#}:\mathbf{LoBicat}(\mathcal{B},\mathcal{D}) \rightarrow \mathbf{LoBicat}(\mathcal{B},\mathcal{D})$. Indeed, if $F:\mathcal{B} \rightarrow \mathcal{D}$ is a pseudofunctor, then the component of $(1_F)_{\#}$ on the idempotent $(A,a)$ of $\mathcal{B}$ is going to be $Fa\: 1_F(A) \: Fa=Fa \: 1_{FA} \:Fa =Fa \:Fa = Fa$ and the component of $1_{F_{\#}}$ is $1_{F_{\#}(A,a)} = 1_{(FA,Fa)} =Fa$. So $(1_F)_{\#} = 1_{F_{\#}}$. \\
Also, for pseudonatural transformations as in the diagram
$$\xymatrixcolsep{2cm}
\xymatrix{
  \mathcal{B} \ruppertwocell^F{\phi}
    \rlowertwocell_H{\psi}
    \ar[r]|{D}
  &\mathcal{D}\\ }$$
and for an idempotent $(A,a)$ in $\mathcal{B}$ we have $$\psi_{\#} \phi_{\#}(A,a) = \psi_{\#}(A,a) \phi_{\#}(A,a)= Ha\: \psi A \: Ga \: Ga \: \phi A \: Fa = Ha \: \psi A \: Ga \: \phi A \: Fa$$and
$$(\psi \phi)_{\#} (A,a) = Ha \: \psi \phi A \: Fa = Ha \: \psi A \: \phi A \: Fa.$$
However, by naturality, $Ga \: \phi A = \phi A \: Fa$, so
$$\psi_{\#} \phi_{\#}(A,a) = Ha \: \psi A \: \phi A \: Fa \: Fa = Ha \: \psi A \: \phi A \: Fa = (\psi \phi )_{\#} (A,a).$$
By the definition of $F_{\#}$ for any pseudofunctor $F$ it's clear that $(1_{\mathcal{B}})_{\#} = 1_{\mathcal{B}_{\#}}$ and $G_{\#} F_{\#} = (GF)_{\#}$. So $\-- _{\#}: \mathbf{LoBicat} \rightarrow \mathbf{LoBicat}$ is a 2-functor.
\end{proof}

\begin{lemma}\label{lemma_Kar_prod}
For two locally ordered bicategories $\mathcal{B}$ and $\mathcal{D}$, $\mathbf{Kar(\mathcal{B})} \times \mathbf{Kar(\mathcal{D})} \simeq \mathbf{Kar(\mathcal{B\times D})}$
\end{lemma}
\begin{proof}
A pair $((A,B), (a,b))$ in $\mathcal{B\times D}$ is an idempotent if and only if $(A,a)$ and $(B,b)$ are idempotents in $\mathcal{B}$ and $\mathcal{D}$ respectively: $(a,b)(a,b)=(a,b)$ if and only if $(aa, bb) = (a,b)$ if and only if $aa=a$ and $bb=b$.

Also an arrow $(S,T):(X,Y) \rightarrow (A,B)$ is an arrow between the idempotents $((X,Y), (x,y))$ and $((A,B),(a,b))$ in $\mathcal{B\times D}$ if and only if $S:(X,x) \rightarrow (A,a)$ and $T:(Y,y) \rightarrow (B,b)$ are arrows between idempotents: $(a,b)(S,T)(x,y) =(S,T)$ if and only if (aSx, bTy)=(S,T) if and only if $aSx = S$ and $bTy=T$.

Since everything is defined componentwise, it's easy to see that the functor $\langle p_1, p_2 \rangle_{\#}:\mathbf{Kar(\mathcal{B\times D})} \rightarrow \mathbf{Kar(\mathcal{B})} \times \mathbf{Kar(\mathcal{D})}$, where $p_1$ and $p_2$ are the projections of $\mathcal{B}$ and $\mathcal{D}$, is an equivalence.
\end{proof}

In the following consider a locally ordered Cartesian bicategory $\mathcal{B}$. We will denote the tensor product by $\otimes$, the product, the projections and terminal in $\mathbf{Map(\mathcal{B})}$ by $\times$, $p$, $r$ and $I$, and the local product and local terminal in $\mathcal{B}$ by $\wedge$ and $\top$. Remember that the product $\times$ in $\mathbf{Map(\mathcal{B})}$ is the tensor product and $I$ is the identity. Also, for the comonoid structure on an object $X$, we will write $d_X :X \rightarrow X \times X$ and $t_X :X \rightarrow I$ and we will use capital letters F,G,S,T,... for general arrows in $\mathcal{B}$ and small letters $a,b,x,y,\dots$ for idempotents.

\begin{lemma} \label{Kar_loc_prod}
If each $\mathcal{B}(X,A)$ has finite products, then each $\mathbf{Kar(\mathcal{B})}((X,x),(A,a))$ has finite products too.
\end{lemma}
\begin{proof}
Let $(X,x)$ and $(A,a)$ be idempotents and $\top_{X,A} :X \rightarrow A$ the terminal in $\mathcal{B}(X,A)$. Then the arrow $a\top_{X,A} x :(X,x) \rightarrow (A,a)$ is the terminal in the category $\mathbf{Kar(\mathcal{B})}((X,x),(A,a))$. Indeed, if we consider an arrow $F:(X,x) \rightarrow (A,a)$ between idempotents, then $F = aFx \le a \top_{X,A} x$. \par
Suppose now that we have arrows $F,G:(X,x) \rightarrow (A,a)$. Then $a (F \wedge G) x$ is the product of $F,G$ in $\mathbf{Kar(\mathcal{B})}$ : For the projections we have $a (F\wedge G) x \le aFx = F$ and $a (F \wedge G) x \le aGx = G$. If $H:(X,x) \rightarrow (A,a)$ is an arrow in  $\mathbf{Kar(\mathcal{B})}$ such that $H \le F$ and $H\le G$, then $H \le F \wedge G$ in $\mathcal{B}$ and so $H = aHx \le a (F\wedge G) x$.
\end{proof}

Henceforth we will denote the product of $F$ and $G$ in  $\mathbf{Kar(\mathcal{B})}((X,x),(A,a))$by $F \sqcap G$.

\begin{lemma}
If $\mathcal{B}$ is a locally ordered Cartesian bicategory, then $\mathbf{Map(Kar(\mathcal{B}))}$ has a terminal object.
\end{lemma}
\begin{proof}
Let $I$ be the terminal object in $\mathbf{Map(\mathcal{B})}$. Then $1_I$ is the local terminal object in $\mathcal{B}(I,I)$. We will show that $(I,1_I)$ is the terminal object in $\mathbf{Map(Kar(\mathcal{B}))}$. Let $(A,a)$ be an object in $\mathbf{Kar(\mathcal{B})}$. Since $\mathcal{B}$ is Cartesian, we have a unique map $t:A \rightarrow I$. If $t^*$ is its right adjoint,then $1_A \le t^*t$ and $tt^* \le 1_I$. The arrows $ta:A \rightarrow I$ and $at^*:I\rightarrow A$ are modules between the idempotents $(A,a)$ and $(I,1_I)$ and $ta \dashv at^*$ in $Kar(\mathcal{B})$ :

For the unit we have $a=aa=a1_Aa \le at^*ta = (at^*)(ta)$ and for the counit we have $(ta)(at^*)=taat^*=tat^* \le \top_{I,I} =1_I$, because $\mathcal{B}$ is Cartesian.

Suppose now that $S:(A,a)\rightarrow (I,1_I)$ is a map in $\mathbf{Kar(\mathcal{B})}$, with right adjoint $R: (I,1_I) \rightarrow (A,a)$. By definition, $Sa=S$ and $aR=R$. By the adjunction we have $A\le RS$ and $SR \le 1_I$. We will show that $S=ta$. First, $S\le \top_{A,I} =t$, because $\mathcal{B}$ is Cartesian. So $S=Sa\le ta$. On the other hand, $ta \le tRS$. But $tR \le \top_{I,I} =1$ implies $R\le r$ by adjunction. So $ta \le tRS \le trS \le 1_I S =S$.
\end{proof}

\begin{proposition}
If $\mathcal{B}$ is a locally ordered Cartesian bicategory, the tensor product $\otimes$ on $\mathcal{B}$ produces a functor $\otimes_\# : \mathbf{Kar(\mathcal{B})} \times \mathbf{Kar(\mathcal{B})} \rightarrow \mathbf{Kar(\mathcal{B})}$ and the natural isomorphisms $\rho, \lambda, \gamma$ and $\alpha$ on $\mathcal{B}$ produce natural isomorphisms
$$\rho_\# : (X,x) \rightarrow (X,x) \otimes_\# (I, 1_I), $$
$$\lambda_\# : (X,x) \rightarrow (I,1_I) \otimes_\# (X,x), $$
$$\gamma_\# : (X,x) \otimes_\# (Y,y) \rightarrow (Y,y) \otimes_\# (X,x) \text{ and} $$
$$\alpha_\# : ( (X,x) \otimes_\# (Y,y) ) \otimes_\# (Z,z) \rightarrow (X,x) \otimes_\# ((Y,y) \otimes_\# (Z,z) )$$
on $\mathbf{Kar(\mathcal{B})}$. Moreover, the above data satisfies the coherence conditions in order for $\otimes_\#$ to be a tensor product on $\mathbf{Kar(\mathcal{B})}$.
\end{proposition}
\begin{proof}
Considering \Cref{lemma_sharp_functor}, \Cref{lemma_sharp_transf}, and \Cref{lemma_Kar_prod} in the previous section, it suffices to show that $\otimes_\#$ satisfies the coherence laws. These will follow by naturality, by the definition of idempotents and by the respective laws in $\mathcal{B}$. For instance, in the following diagram, the triangles 1 and 4 commute because $(x \otimes 1_I) \otimes y$ and $x \otimes (1_I \otimes y)$ are idempotents, the squares 2 and 3 commute by naturality of $\rho$ and $\lambda$, and 5 holds in $\mathcal{B}$.
$$\xymatrixcolsep{0.2cm} \xymatrix{ &  (X,x) \otimes_\# (Y,y) \ar[dl]_{ \rho_\# \otimes_\# (Y,y) } \ar[dr]^{(X,x) \otimes_\# \lambda_\# } \\
((X,x) \otimes_\# (I, 1_I) ) \otimes_\# (Y,y) \ar[rr]_{\alpha_\#} & & (X,x) \otimes_\# ( (I,1_I) \otimes_\# (Y,y) ) } $$
$$ = $$
$$\xymatrixcolsep{0.1cm} \xymatrixrowsep{0.2cm} \xymatrix{
& & X \otimes Y \ar[ddll]_{\rho \otimes Y} \ar[ddrr]^{x\otimes y} & & X\otimes Y \ar[ll]_{x\otimes y} \ar[rr]^{x\otimes y} & & X \otimes Y \ar[ddll]_{x\otimes y} \ar[ddrr]^{X\otimes \lambda} \\
\\
(X \otimes I) \otimes Y \ar[dd]_{(x \otimes 1_I) \otimes y} \ar[ddrr]^{x\otimes (1_I \otimes y)} & & 2 & & X\otimes Y \ar[ddll]_{\rho \otimes X} \ar[ddrr]^{X\otimes \lambda} & & 3 & & X\otimes (I \otimes Y) \ar[ddll]_{x\otimes (1_I \otimes y)} \ar[dd]^{x \otimes (1_I \otimes y )} \\
 \quad \quad \quad 1& & & & 5 & & & & 4 \quad \quad \quad \\
(X \otimes I) \otimes Y \ar[rr]_{(x \otimes 1_I) \otimes y} & & (X \otimes I) \otimes Y \ar[rrrr]_{\alpha} & & & & X\otimes (I \otimes Y) \ar[rr]_{x \otimes (1_I \otimes y)} & & X \otimes (I\otimes Y) } $$
Similarly we have the triangle diagram for the symmetry
$$\xymatrixcolsep{0.2cm}  \xymatrix{ & (Y,y) \otimes_\# (X,x) \ar[dl]_{ \gamma_\#} \ar[dr]^{\gamma_\# }\\
(X,x) \otimes_\# (Y,y) \ar[rr]_{1_{(X,x) \otimes_\# (Y,y)} } & & (X,x) \otimes_\# (Y,y) } $$
$$ = $$
$$\xymatrixcolsep{0.3cm} \xymatrixrowsep{0.2cm} \xymatrix{
& &Y \otimes X \ar[ddll]_{\gamma} \ar[ddrr]^{y \otimes x} & & Y \otimes X \ar[ll]_{y \otimes x} \ar[rr]^{y \otimes x} & & Y \otimes X \ar[ddll]_{y \otimes x} \ar[ddrr]^{\gamma} \\
\\
X \otimes Y \ar[dd]_{x \otimes y} \ar[ddrr]^{x \otimes y} & & & & Y \otimes X \ar[ddll]_{\gamma} \ar[ddrr]^{\gamma} & & & & X\otimes Y \ar[ddll]_{x \otimes y} \ar[dd]^{x \otimes y} \\
\\
X \otimes Y \ar[rr]_{x \otimes y} & & X \otimes Y \ar[rrrr]_{1} & & & & X\otimes Y \ar[rr]_{x \otimes y} & & X \otimes Y, } $$
or the pentagon and the hexagon diagrams.
\end{proof}

\begin{proposition} \label{Map_Kar_prod}
If $\mathbf{Map\mathcal{B}}$ has finite products, $\mathbf{Map(Kar\mathcal{B})}$ has finite products too.
\end{proposition}
\begin{proof}
It suffices to show that the restriction of the pseudofunctor $\otimes_\#$ on $\mathbf{Map(Kar\mathcal{B})}$ is right adjoint to the pseudofunctor $\Delta: \mathbf{Map(Kar\mathcal{B})} \rightarrow \mathbf{Map(Kar\mathcal{B})} \times \mathbf{Map(Kar\mathcal{B})}$, $(X,x) \mapsto ((X,x),(X,x))$. For each idempotent $(X,x)$ we have the arrow between idempotents $(d_X)_\# = (x \otimes x ) d_X x : (X,x) \rightarrow (X\times X, x \otimes x)$ which we will prove to be the unit of the adjunction. For the counit, first notice that for each idempotent of the form $(A\times B, a\otimes b)$ we have the arrows between idempotents $p_\# =a p (a \otimes b)$ and $r_\# = b r (a\otimes b)$, where $p$ and $r$ are the projection maps of the product $A\times B$ in $\mathbf{Map\mathcal{B}}$. We will show that $(p_\#, r_\#)$ is the counit of the adjunction.

To show that $\otimes_\#$ is right adjoint to $\Delta$ it is to show that $(p_\# \otimes_\# r_\#) (d_{A\times B})_\# = 1_{(A,a) \otimes_\# (B,b)}= 1_{(A\times B , a \otimes b)}=a\otimes b$ and $(p_\#,r_\#)((d_X)_\#,(d_X)_\#)=1_{(X,x),(X,x))}=(x,x)$. For the former we see in the following diagram that $a\otimes b \le (p_\# \otimes_\# r_\#) (d_{A\times B})_\#$:
$$\xymatrixcolsep{1cm} \xymatrixrowsep{0.4cm} \xymatrix{A\times B \ar[r]^{a\times b} \ar@/_1.5pc/[rrrddddd]_{a\otimes b} & A\times B \ar[r]^-{d_{A\times B}} \ar[ddr]_{a\otimes b} & (A\times B) \times (A\times B) \ar[r]^{(a\otimes b)\otimes (a\otimes b)} \ar[ddr]_{(a\otimes b)\otimes (a\otimes b)} & (A\times B)\times (A\times B) \ar[dd]^{(a\otimes b)\otimes (a\otimes b)} \\
& & \le & = \quad \quad \quad \quad \\
& & A\times B \ar[r]^-{d_{A\times B}} \ar[ddr]_1 & (A\times B) \times (A\times B) \ar[dd]^{p\times r} \\
& & & = \quad \quad \quad \quad \\
& & & A \times B \ar[d]^{a\otimes b} \\
& & & A\times B }$$
For the opposite inequality we have $$d_\# = ((a\otimes b)\otimes (a\otimes b)) \: d \: (a \otimes b) \le (a\otimes b)\otimes (a\otimes b) (p^* \otimes r^*) (a\otimes b) = $$ $$((a\otimes b) p^* a) \otimes ((a\otimes b) r^* b )= (p_\#)^* \otimes(r_\#)^* =(p_\# \otimes_\# r_\#)^*,$$so $d_\# \le (p_\# \otimes_\# r_\#)^* = (p_\# \otimes_\# r_\#)^* (a\otimes b)$, which is equivalent to $(p_\# \otimes_\# r_\#) d_\# \le a \otimes b$.\\
For the second identity the diagram
$$\xymatrixcolsep{0.8cm} \xymatrixrowsep{0.4cm} \xymatrix{X \ar[r]^x \ar@/_1.5pc/[rrrddddd]_x & X \ar[r]^-{d_X} \ar[ddr]_x &X \times X \ar[r]^{x\otimes x} \ar[ddr]_{x\otimes x} & X\times X \ar[dd]^{x \otimes x} \\
& & \le & = \quad \quad \quad \quad \\
& & X \ar[r]^-{d_X} \ar[ddr]_1 & X \times X \ar[dd]^p \\
& & & = \quad \quad \quad \quad \\
& & & X \ar[d]^x \\
& & & X }$$
will give $x \le p_\# (d_X)_\#$. But we also have $$d_X \le p^* \Rightarrow (x \otimes x ) d_X x \le (x \otimes x) p^* x \Rightarrow$$ $$ (d_X)_\# \le (p_\#)^* = (p_\#)^* x \Rightarrow p_\# (d_X)_\# \le x.$$So $p_\# (d_X)_\# = x$ and similarly $r_\# d_X = x $.
\end{proof}

We can now prove the following theorem:

\begin{theorem}
If $\mathcal{B}$ is a locally ordered Cartesian bicategory then $\mathbf{Kar\mathcal{B}}$ is a locally ordered Cartesian bicategory too.
\end{theorem}
\begin{proof}
By \Cref{CW_1.6}, \Cref{Kar_loc_prod}  and \Cref{Map_Kar_prod}, it suffices to show that for any arrows between idempotents $F:(X,x) \rightarrow (A,a)$ and $G:(Y,y) \rightarrow (B,b)$, $$F \otimes_\# G = F \otimes G = (p_\#^* F p_\#) \sqcap (r_\#^* G r_\#).$$ 
Since $\mathcal{B}$ is a locally ordered Cartesian bicategory, we have $F\otimes G = (p^* F p) \wedge (r^* G r)$. So $F \otimes G \le p^* F p$ and $F \otimes G \le r^* G r$. It follows that $$F \otimes G = (a\otimes b)(F \otimes G) (x\otimes y) \le (a \otimes b ) p^*F p (x\otimes y) =$$ $$ (a\otimes b) p^* (a F x) p (x\otimes y) = p_\#^* F p_\#$$
and similarly $F \otimes G \le r_\#^* G r_\#$. So $$F \otimes G \le (p_\#^* F p_\#) \sqcap (r_\#^* G r_\#).$$
On the other hand, notice that $$a \otimes b = (p^* a p)\wedge (r^* b r) \Rightarrow a\otimes b \le p^* a p \Rightarrow (a\otimes b) p^* \le p^* a$$ and similarly $p (x \otimes y) \le x p$. So we have
$$ p_\#^* F p_\# = ((a\otimes b) p^* a) F (x p (x\otimes y)) \le p^* a a F x x p = p^* F p.$$ Also, $r_\#^* G r_\# \le r^* G r$. So $$(p_\#^* F p_\#) \wedge (r_\#^* G r_\#) \le F \otimes G,$$ which implies that $$(p_\#^* F p_\#) \sqcap (r_\#^* G r_\#) =(a\otimes b) ((p_\#^* F p_\#) \wedge (r_\#^* G r_\#)) (x\otimes y)$$ $$ \le (a\otimes b) (F\otimes G) ( x\otimes y) = F\otimes G.$$
\end{proof}

\section{The Bicategory $\mathbf{Mod\mathcal{B}}$ for Locally Ordered Cartesian Bicategories} \label{The Bicategory ModB for Locally Ordered Cartesian Bicategories}

In a locally ordered bicategory $\mathcal{B}$, a monad $(A,a)$ is an arrow $a:A\rightarrow A$ together with inequalities $1_A \le a$ and $aa \le a$. Also, a module $m:(A,a) \rightarrow (B,b)$ is an arrow $m:A\rightarrow B$, equipped with two inequalities $ma \le m$ and $bm \le m$. \par
Note that for any monad $a:A\rightarrow A$ we also have $a =a 1_A \le aa \le a$, so $aa=a$. That is, an arrow $a:A \rightarrow A$ is a monad if and only if it is an idempotent with $1_A \le a$. Moreover, if $m:(A,a) \rightarrow (B,b)$ is a module, then $m = m1_A \le ma$ and $m =1_B m \le bm$, so $m = ma = bm$. That means that $m$ is a module if and only if it is a module between idempotents. We can regard then $\mathbf{Mod\mathcal{B}}$ as the full subcategory of $\mathbf{Kar\mathcal{B}}$, consisting of the idempotents $a:A\rightarrow A$ with $1_A \le a$.\par
Suppose again that $\mathcal{B}$ is a locally ordered Cartesian bicategory. In the previous section we proved that $\mathbf{Kar\mathcal{B}}$ is a locally ordered Cartesian bicategory as well. We will use this fact to show that the same is true for $\mathbf{Mod\mathcal{B}}$:

\begin{theorem}\label{mod_lo_Cartesian}
If $\mathcal{B}$ is a locally ordered Cartesian bicategory then $\mathbf{Mod\mathcal{B}}$ is a locally ordered Cartesian bicategory too.
\end{theorem}
\begin{proof}
By looking at the \Cref{CW}, it suffices to show that for every pair of monads $(A,a)$, $(B,b)$, their tensor product $(A,a) \otimes_\# (B,b) = (A\times B, a\otimes b)$ is not only an idempotent, but it also satisfies $1_{A\times B} \le a \otimes b$. By \Cref{CW_1.6}, $a\otimes b$ is given by the product $(p^* a p) \wedge (r^* b r)$ and we have $$1_{A\times B} \le p^*p = p^* 1_A p \le p^* a p \text{ and}$$ $$ 1_{A\times B} \le r^* r = r^* 1_B r \le r^* b r.$$ So $1_{A\times B} \le a \otimes b$ and the proposition is true.
\end{proof}

Even though Carboni and Walters used that the locally ordered bicategory of ordered objects and ordered ideals is Cartesian, a proof of it was not found in \cite{Carboni1987}. The theorem above and the following example make that clear.

\begin{example}
Consider an exact category $\mathcal{E}$, that is, a regular category where every equivalence relation is the kernel pair of some morphism. We ask for this extra condition because it will allow us to have coequalizers for an equivalence relation of the form $\xy \morphism(0,0)|a|/@{>}@<3pt>/<300,0>[a`b;f]
\morphism(0,0)|b|/@{>}@<-3pt>/<300,0>[a`b;g]\endxy$, which we need in defining the composition for the bicategory $\mathbf{Mod(\mathbf{Rel}\mathcal{E})}$ below.
If $r: A \rightarrow B$ is a relation in $\mathcal{E}$ and $a,b$ are arrows from some object $U$ to $A$ and $B$ respectively, then we write $a(r)b$ if there exists an arrow $U\rightarrow R$ that makes the following diagram commute
$$\xymatrix{
& & U \ar[dl]_a \ar[d] \ar[dr]^b \\
& A & R \ar[l]_{r_0} \ar[r]^{r_1} & B &. } $$
An \textbf{ordered object} $A$ in $\mathcal{E}$ is an object $A$ together with a relation $r_A:A\rightarrow A$,
$$ \xymatrix{ A & R_A \ar[l]_{r_{A,0}} \ar[r]^{r_{A,1}} & A } $$
with the properties
\begin{enumerate}
\item if $a:U\to A$ then $a(r_A)a$, and
\item if $a,b,c:U \to A$ with $a (r_A) b$ and $b (r_A) c$, then $ a (r_A) c$.
\end{enumerate}
Let $A$ and $B$ be two ordered objects in $\mathcal{E}$. An \textbf{ordered ideal} from $A$ to $B$ is a relation $s:A\rightarrow B$ such that for arrows $a,a':U \rightarrow A, b,b':U \rightarrow B$ with $a'(r_A) a, a(s) b, b(r_A)b'$, then $a'(s)b'$
$$ \xymatrixrowsep{0.3cm} \xymatrixcolsep{0.5cm} \xymatrix{
& R_A \ar[dl]_{r_{A,0}} \ar[dr]^{r_{A,1}} & & S \ar[dl]_{s_0} \ar[dr]^{s_1} & & R_B \ar[dl]_{r_{B,0}} \ar[dr]^{r_{B,1}} \\
A & & A & & B & & B \\
\\
& & & U \ar[uulll]^{a'} \ar[uul]^a \ar[uur]_b \ar[uurrr]_{b'} & & .
} $$
For an ordered object $A$, the identity ideal is the relation $r_A :A\rightarrow A$. This is an ideal since for $a,a',b,b':U\rightarrow A$, if $a' (r_A) a$ and $a (r_A) b$, then $a' (r_A) b$. If in addition $b (r_A) b'$, then $a' (r_A) b'$. The morphisms between ideals are just morphisms between relations. \par
We have formed a bicategory with objects the ordered objects in an exact category $\mathcal{E}$, 1-cells the order ideals, and 2-cells the morphisms between them. We denote this bicategory by $\mathbf{Idl}\mathcal{E}$, as in \cite{Order}.
We can observe now that a monad in the bicategory of relations in $\mathcal{E}$ is exactly the same as an ordered object and a module is the same as an ordered ideal. It follows that
$$\mathbf{Mod(\mathbf{Rel}\mathcal{E})} \simeq \mathbf{Idl}\mathcal{E}$$
and by \Cref{mod_lo_Cartesian} and \Cref{relations}, $\mathbf{Idl}\mathcal{E}$ is Cartesian.
\end{example}

\section{Cartesian Bicategories} \label{Cartesian Bicategories}

In \cite{Carboni2008}, prior to Cartesian bicategories, the authors define precartesian bicategories. They do not ask for finite bicategorical products throughout the whole bicategory, only throughout the full sub-bicategory of left adjoints. As in the previous section, we denote the latter with $\mathbf{Map}(\mathcal{B})$, and we call the left adjoints \emph{\textbf{maps}}. They prove that in any precartesian bicategory we can build a canonical lax tensor product. They then proceed to call a precartesian bicategory Cartesian in the case that this tensor product is pseudo. We also see that in this paper, in constrast to its prequel \cite{Carboni1987}, they introduce Cartesian bicategories as ones that have various properties rather than ones with some extra structure satisfying coherence conditions. We start with a review of bicategorical products:

\begin{definition}\cite{Benabou}\label{biproducts}
Consider a bicategory $\mathcal{B}$ and objects $A,B$ in it. An object $A\times B$, together with 1-cells $p_{A,B}:A\times B \rightarrow A$ and $r_{A,B}:A\times B \rightarrow B$ is the \textbf{bicategorical product} of $A$ and $B$ if the following functor is an equivalence $\forall C \in \mathcal{B}$:
$$ \begin{array}{cr}
\Gamma : \mathcal{B}(C,A\times B) \rightarrow \mathcal{B}(C,A)\times \mathcal{B}(C,B) \\
(f:C\rightarrow A\times B) \mapsto \langle p_{A,B} \circ f, r_{A,B} \circ f \rangle \\
(\alpha :f \rightarrow g ) \mapsto \langle p_{A,B} \alpha , r_{A,B} \alpha \rangle \end{array} $$
Or equivalently, if the functor $\Gamma$ is fully faithful and essentially surjective.
A \textbf{terminal object} in $\mathcal{B}$ is an object $1$ such that  $\mathcal{B}(A,1)$ is equivalent to the category $\mathbbm{1}$. In other words, for each $A$ in $\mathcal{B}$, there is a 1-cell $A \rightarrow 1$, and for two 1-cells $f,g :A\rightarrow 1$ there is a unique isomorphism between them. 
\end{definition}

To explain a little more what a bicategorical product is, consider first 1-cells \\ $f:C\rightarrow A$ and $g:C\rightarrow B$. Then there exists a 1-cell $\langle f,g\rangle:C \rightarrow A\times B $ such that $$p_{A,B} \circ \langle f,g \rangle \cong f \text{ and } r_{A,B} \circ \langle f,g \rangle \cong g.$$ We will denote these isomorphisms by $\mu_{f,g}$ and $\nu_{f,g}$ respectively. If now $f,g: C \rightarrow A\times B$ are 1-cells in $\mathcal{B}(C,A \times B)$ and $$\langle \beta, \gamma\rangle : \langle p_{A,B} \circ f, r_{A,B} \circ f \rangle \rightarrow \langle p_{A,B} \circ g, r_{A,B} \circ g \rangle$$ is a 2-cell in $\mathcal{B}(C,A)\times \mathcal{B}(C,B)$, then there exists a unique $\alpha : f\rightarrow g$ such that $$\Gamma (\alpha)=\langle \beta,\gamma \rangle \text{, i.e. } p_{A,B} \alpha = \beta \text{ and } r_{A,B} \alpha = \gamma.$$

We will show that the isomorphisms $\mu_{f,g}$ and $\nu_{f,g}$ are natural. Consider 2-cells $\phi : f \rightarrow f'$ and $\psi :g \rightarrow g'$, for $f,f':C\rightarrow A$ and $g,g':C \rightarrow B$. Then for the 2-cell
$$\langle \mu_{f',g'}^{-1} \circ \phi \circ \mu_{f,g},\nu_{f',g'}^{-1} \circ \psi \circ \nu_{f,g}\rangle : \big\langle p_{A,B} \circ \langle f,g\rangle , r_{A,B} \circ \langle f,g\rangle \big\rangle \rightarrow \big\langle p_{A,B} \circ \langle f',g'\rangle , r_{A,B} \circ \langle f',g'\rangle \big\rangle$$
there will be a unique 2-cell $\langle \phi, \psi \rangle : \langle f,g \rangle \rightarrow \langle f',g' \rangle $ which makes the following diagrams commute.
$$\xymatrix{ p_{A,B} \circ \langle f,g \rangle \ar[d]_{ p_{A,B} \circ \langle \phi, \psi \rangle} \ar[rr]^{\: \mu_{f,g}} &&f \ar[d]^{\phi} \\  p_{A,B} \circ \langle f',g' \rangle \ar[rr]_{\: \mu_{f',g'}} && f'}$$
$$\xymatrix{ r_{A,B} \circ \langle f,g \rangle \ar[d]_{ r_{A,B} \circ \langle \phi, \psi \rangle} \ar[rr]^{\: \nu_{f,g}} &&g \ar[d]^{\psi} \\  r_{A,B} \circ \langle f',g' \rangle \ar[rr]_{\: \nu_{f',g'}} && g'}$$

\begin{proposition}\cite{Carboni2008} A bicategory with bicategorical products and a terminal object, has finite bicategorical products. That is, for every natural number $n$ and objects $A_1,\dots A_n$ in $\mathcal{B}$, there is an object $P$, together with 1-cells $p_i:P \rightarrow A_i, \: i=1,\dots,n$, such that for every $C \in \mathcal{B}$, the functor
$$ \begin{array}{cr}
\Gamma : \mathcal{B}(C,P) \rightarrow \prod_i \mathcal{B}(C,A_i), \\
(f:C\rightarrow P) \mapsto \langle p_i \circ f\rangle_i \\
(\alpha :f \rightarrow g ) \mapsto \langle p_i \alpha \rangle_i \end{array} $$
is an equivalence.

Moreover, a bicategory with the above assumptions admits a symmetric monoidal bicategorical structure.
\end{proposition}

\begin{definition} \cite{Carboni2008}
A bicategory $\mathcal{B}$ is said to be \textbf{precartesian} if
\begin{enumerate}[label=\roman*.]
\item the bicategory $\mathbf{Map}(\mathcal{B})$ has finite products
\item each category $\mathcal{B}(A,B)$ has finite products.
\end{enumerate}
\end{definition}

One of the examples mentioned in \cite{Carboni2008} was the bicategory of spans in a finitely complete category. We would like to provide a proof  that this is indeed precartesian, since one was not found in either \cite{Carboni1987} or \cite{SpansLWW}.

\begin{proposition}
The bicategory $\mathbf{Span}(\mathcal{E})$, for $\mathcal{E}$ finitely complete category, is precartesian.
\end{proposition}
\begin{proof}
A span $r=(r_0,R,r_1):A \xleftarrow{r_0} R \xrightarrow{r_1} B$ in $\mathcal{E}$ has a right adjoint if and only if $r_0:R\rightarrow A$ is invertible in $\mathcal{E}$. Let $1$ be the terminal object in $\mathcal{E}$. Then for every $A$ in $\mathbf{MapSpan}(\mathcal{E})$, the span $A \xleftarrow{1_A} A \xrightarrow{t} 1$, where $t$ is the unique arrow to the terminal $1$, is a map since $1_A$ is invertible. If $A \xleftarrow{r_0} R \xrightarrow{r_1} 1$ is another map from $A$ to $1$ then the diagram
$$
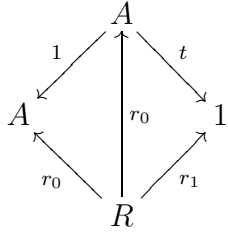
\xymatrix{&A \ar[dl]_1 \ar[dr]^t \\ A & & 1 \\ & R \ar[ul]^{r_0} \ar[ur]_{r_1} \ar[uu]_{r_0} }$$
commutes by the universal property of $1$. Also $r_0$ is invertible, so the two spans are isomorphic and then the terminal object $1$ becomes a terminal object in $\mathbf{MapSpan}(\mathcal{E})$.

For objects $A,B$ we can form their product $A\times B$ with projections $p,r$ in $\mathcal{E}$. We claim that $A\times B$ together with the maps $p_{A,B} =(1,A\times B,p) $ and $r_{A,B}=(1,A\times B,r)$ is the product of $A,B$.

Consider maps $C\xleftarrow{f_0} R \xrightarrow{f_1}A$ and $C \xleftarrow{g_0} S \xrightarrow{g_1}B$ and let $f=f_1f_0^{-1},\: g=g_1g_0^{-1}$. Then, we have the unique arrow $\langle f,g \rangle :C \rightarrow A\times B$ with $p\langle f,g \rangle=f$ and $r\langle f,g \rangle=g$. In the following diagrams we see the compositions of what we claim to be the projections of $A\times B$ with the span $\langle f,g \rangle := C \xleftarrow{1} C \xrightarrow{\langle f,g \rangle} A\times B$:
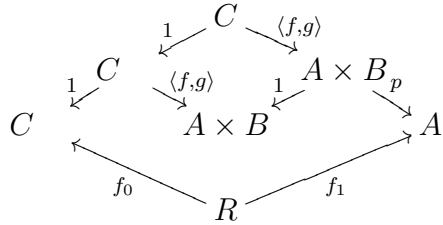
$$\xymatrixcolsep{0.2cm} \xymatrixrowsep{0.2cm}  \xymatrix{& & C \ar[dl]_1 \ar[dr]^{\langle f,g \rangle}\\
& C \quad \ar[dl]_1 \ar[dr]^{\langle f,g \rangle} & & A\times B \ar[dl]_1 \ar[dr]^p\\
C\quad & & A\times B & & A\\
\\
& & R \ar[uull]^{f_0} \ar[uurr]_{f_1} }$$

$$\xymatrixcolsep{0.2cm} \xymatrixrowsep{0.2cm} 
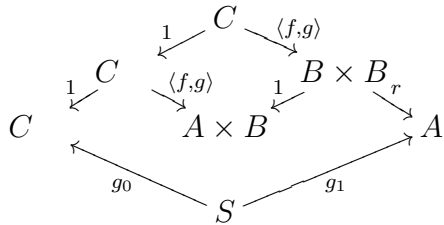
\xymatrix{& & C \ar[dl]_1 \ar[dr]^{\langle f,g \rangle}\\
& C \quad \ar[dl]_1 \ar[dr]^{\langle f,g \rangle} & & B\times B \ar[dl]_1 \ar[dr]^r\\
C\quad & & A\times B & & A\\
\\
& & S \ar[uull]^{g_0} \ar[uurr]_{g_1} }$$
In the first case the arrow $f_0^{-1}:C\rightarrow B$ is an isomorphism between the maps $p_{A,B} \circ \langle f,g \rangle$ and $(f_0,R,f_1)$ and in the second, the arrow $g_0^{-1} :C\rightarrow S$ is an isomorphism between $r_{A,B} \circ \langle f,g \rangle$ and $(g_0,S,g_1)$. We proved that the functor $\Gamma$ in the \Cref{biproducts} is essentially surjective. It is also fully faithful, as we show below.

For two maps
$$\xymatrix{ & R\ar[dl]_{f_0} \ar[dr]^{f_1} \\ C \quad & & A\times B & (1) \\ & S\ar[ul]^{g_0} \ar[ur]_{g_1} }$$
and arrows $\beta, \gamma :R\rightarrow S$ that make the following diagrams commute
$$\xymatrix{ & R\ar[dl]_{f_0} \ar[dr]^{p \circ f_1} \ar[dd]^{\beta} & & & R \ar[dl]_{f_0} \ar[dd]^{\gamma} \ar[dr]^{r\circ f_1} \\
C & & A, & C & & B & (2) \\
& S\ar[ul]^{g_0} \ar[ur]_{p\circ g_1} & & & S \ar[ul]^{g_0} \ar[ur]_{r\circ g_1} }$$
we have $\beta=g^{-1}_0 f_0 =\gamma$. Also $p\circ f_1 = p \circ g_1 \circ \beta $ and $r\circ f_1 =r\circ g_1 \circ \beta$, so $f_1 = g_1 \circ \beta$, i.e. $\beta$ is an arrow between the two maps in (1).\\
Also the diagrams (2) imply that $\Gamma (\beta)=\langle \beta,\beta \rangle = \langle \beta,\gamma \rangle$. If $\alpha$ is another arrow with $\Gamma (\alpha)=\langle \beta,\gamma \rangle$ then $\alpha =g_0^{-1} f_0=\beta$, so $\beta$ is unique with this property. This proves that $A \times B$ with maps $p_{A,B}, r_{A,B}$ is indeed the product of $A$ and $B$ in $\mathbf{MapSpan}(\mathcal{E})$. Hence, by the proposition 2.3.4. $\mathbf{MapSpan}(\mathcal{E})$ has finite products.\par It remains to show that for every two objects $A,B$, the category $\mathbf{Span}(\mathcal{E})(A,B)$ has finite products. The span $(p,A\times B, r)$ will play the role of the terminal object: If $f=(f_0,R,f_1):A\rightarrow B$ then we have the unique arrow $\langle f_0,f_1 \rangle :R\rightarrow A\times B$ for which the following diagram commutes.
$$\xymatrix{ & R\ar[dl]_{f_0} \ar[dr]^{f_1} \ar[dd]^{\langle f_0,f_1 \rangle} \\ A & & B \\ &A\times B \ar[ul]^p \ar[ur]_r }$$
Consider now two spans $f=(f_0,R,f_1)$ and $g=(g_0,S,g_1)$. The following pullback exists in $\mathcal{E}$:
$$\bfig
\square[R\times_{A\times B} S`S`R`A\times B;\rho`\pi`\langle g_0,g_1 \rangle`\langle f_0,f_1\rangle]
\efig$$
and the product of $f,g$ is the span $(p\circ \langle f_0,f_1\rangle \circ \pi, R\times_{A\times B} S, r\circ \langle f_0,f_1\rangle \circ \pi )$ with projections $\pi$ and $\rho$. Indeed, the diagrams
$$\xymatrix{ & R\times_{A\times B} S \ar[dl]_{p\circ \langle f_0,f_1 \rangle \circ \pi} \ar[dr]^{r\circ \langle f_0,f_1 \rangle \circ \pi} \ar[dd]^{\pi} \\ A & & B \\ & R \ar[ul]^{f_0} \ar[ur]_{f_1} }$$
$$\xymatrix{ & R\times_{A\times B} S \ar[dl]_{p\circ \langle f_0,f_1 \rangle \circ \rho} \ar[dr]^{r\circ \langle f_0,f_1 \rangle \circ \rho} \ar[dd]^{\rho} \\ A & & B \\ & S \ar[ul]^{g_0} \ar[ur]_{g_1} }$$
commute and the universal property of the above pullback will give the universal property of the product.

It follows that the category $\mathbf{Span}(\mathcal{E})(A,B)$ fas finite products and that $\mathbf{Span}(\mathcal{E})$ is indeed a precartesian bicategory.
\end{proof}

For any bicategory $\mathcal{B}$ we can consider the Grothendieck construction $\mathbf{G}$ for the pseudofunctor
$$\mathbf{Map}(\mathcal{B})^{op}\times \mathbf{Map}(\mathcal{B}) \xrightarrow{i^{op} \times i} \mathcal{B}^{op} \times \mathcal{B} \xrightarrow{\mathcal{B}(\--,\--)} \mathcal{C}at$$
where $i$ is the inclusion, $\mathcal{B}(\--,\--)$ the usual hom-pseudofunctor and $\mathcal{C}at$ the 2-category of categories. All details for this construction can be found in Street's paper \cite{Fibrations} or in Verity's PhD thesis \cite{Verity11}. An object of $\mathbf{G}$ is a triple of the form $(X, R:X\rightarrow A, A)$ in $\mathcal{B}$, an arrow from $(X, R:X\rightarrow A, A)$ to $(Y, S:Y\rightarrow B, B)$ of $\mathbf{G}$ is a triple $(f,\alpha, u)$ in $\mathcal{B}$ as in the diagram:
$$\xymatrixcolsep{0.2cm} \xymatrixrowsep{0.2cm} \xymatrix{ X \ar[rr]^R \ar[dd]_f & & A \ar[dd]^u\\
& \Downarrow \alpha\\
Y \ar[rr]_S & & B,}$$
where $f$ and $u$ are maps. A 2-cell from $(f,\alpha,u)$ to $(f',\alpha',u')$ is a pair of 2-cells $(\phi, \psi)$ in $\mathcal{B}$ such that the following diagram commutes:
$$\xymatrix{
uR \ar[r]^{\alpha} \ar[d]_{\psi R} & Sf \ar[d]^{S\phi}\\
u'R \ar[r]_{\alpha'} & u'R.}$$

\begin{proposition}\cite{Carboni2008}\label{Grothendieck_has_products}
If $\mathcal{B}$ is precartesian then the above bicategory $\mathbf{G}$ has finite bicategorical producs.
\end{proposition}

We use $\otimes$ for the products in $\mathbf{G}$. It is true that for two objects $(X, R:X\rightarrow A, A)$ and $(Y, S:Y\rightarrow B, B)$, their product $R \otimes S$ is given by $p^* R p \wedge r^* S r$, where  $p^*, r^*$ are the right adjoints of the projections in $\mathbf{Map}(\mathcal{B})$, and $\wedge$ denotes the products in the hom-categories of $\mathcal{B}$. Also, the terminal object in $\mathbf{G}$ is given by the terminal object $\top_{I,I}$ in the category $\mathcal{B}(I,I)$, where $I$ is the terminal in $\mathbf{Map}(\mathcal{B})$. This tensor extends to give the following proposition.

\begin{proposition}\cite{Carboni2008}\label{Carboni_lax_functors}
For a precartesian bicategory $\mathcal{B}$, there are lax functors $\otimes: \mathcal{B} \times \mathcal{B} \rightarrow \mathcal{B}$ and $I:\mathbf{1} \rightarrow \mathcal{B}$.
\end{proposition}

\begin{definition}\cite{Carboni2008}
A precartesian bicategory $\mathcal{B}$ is said to be \textbf{Cartesian} when the lax functors $\otimes: \mathcal{B} \times \mathcal{B} \rightarrow \mathcal{B}$ and $I:\mathbf{1} \rightarrow \mathcal{B}$ of \Cref{Carboni_lax_functors} are pseudofunctors.
\end{definition}

 In the same paper, the authors proceed to show that for a Cartesian bicategory, if we restrict $\otimes$ and $I$ to $\mathbf{Map}(\mathcal{B})$, then they provide right adjoints in the bicategorical sense to the diagonal and unique to the terminal bicategory $\mathbf{1}$ pseudofunctors: $$\Delta:\mathbf{Map}(\mathcal{B})\rightarrow\mathbf{Map}(\mathcal{B})\times \mathbf{Map}(\mathcal{B}) \text{ and } !:\mathbf{Map}(\mathcal{B})\rightarrow \mathbf{1}.$$ However, adjunctions between bicategories are rather complicated notions. We refer the reader to Gray's book \cite{Gray1974}, where he uses the name ``transendental quasi-adjunctions'', or to Fiore's paper \cite{Fiore2006}, where he describes the weaker ``biadjunctions''. We will see later that in the case of double categories, this is going to be our starting point. That is, we by start with defining a Cartesian double category as one for which the corresponding diagonal and unique functors admit right adjoints in the usual sense, and we will discuss why we are actually able to do that.

\section{Modules over Monads} \label{Modules over Monads}

In this section we study the bicategory of modules over monads in a bicategory. Our main result here is that if a bicategory $\mathcal{B}$ has finite products locally, then the bicategory of monads and modules in $\mathcal{B}$ also has finite products locally. This shows that if $\mathcal{B}$ is Cartesian, then its bicategory of monads and modules satisfies the first condition in order to be Cartesian as well. We believe though that $\mathcal{B}$ being Cartesian is not enough for showing that the bicategory of modules that have right-adjoints has finite products. We do not know what extra conditions we might need, so we leave it as a question for future investigation.

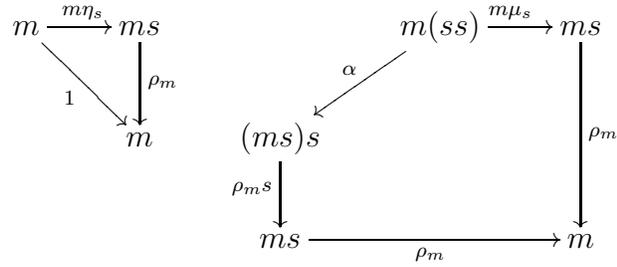
\begin{definition}
Let $\mathcal{B}$ be a bicategory. A \textbf{monad} on an object $A$ is an arrow $t:A\rightarrow A $ together with 2-cells $\eta_t:1_A\Rightarrow t$ and $\mu_t:t t \Rightarrow t$ called unit and multiplication respectively such that the following diagrams commute:
$$\xymatrix{ t \ar[r]^{\lambda^{-1}} \ar[drr]_1 & 1_A t \ar[r]^{\eta_t t} & tt \ar[d]^{\mu_t} & t1_A \ar[l]_{t\eta_t} & t \ar[l]_{\rho^{-1}} \ar[dll]^1 \\ & & t }$$
$$\xymatrix{ & (tt)t \ar[r]^{\mu_t t} & tt \ar[dd]^{\mu_t} \\
t(tt) \ar[ur]^{\alpha} \ar[d]_{t\mu_t}\\
tt \ar[rr]_{\mu_t} & & t }$$
\end{definition}

\begin{definition}
Let $A,B$ be objects in $\mathcal{B}$ and $(s:A\rightarrow A, \eta_s,\mu_s),(t:B\rightarrow B,\eta_t, \mu_t)$ monads. An $s,t$-bimodule is an arrow $m:A\rightarrow B$ equipped with two 2-cells $\rho_m:ms\Rightarrow m$ and $\lambda_m:tm\Rightarrow m$ so that the following diagrams commute:
\begin{enumerate}
\item
$$\xymatrix{ m \ar[r]^{m\eta_s } \ar[dr]_1 & ms \ar[d]^{\rho_m} & & m(ss) \ar[dl]_{\alpha} \ar[r]^{m\mu_s} & ms \ar[dd]^{\rho_m} \\
& m & (ms)s \ar[d]_{\rho_ms}\\
& & ms \ar[rr]_{\rho_m} & & m }$$
($m$ is a right $s$-module)
\item
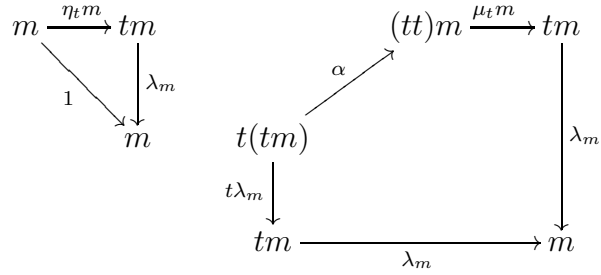
$$\xymatrix{ m \ar[r]^{\eta_tm } \ar[dr]_1 & tm \ar[d]^{\lambda_m} & & (tt)m \ar[r]^{\mu_t m} & tm \ar[dd]^{\lambda_m} \\
& m & t(tm) \ar[ur]^{\alpha} \ar[d]_{t\lambda_m}\\
& & tm \ar[rr]_{\lambda_m} & & m }$$
($m$ is a left $t$-module)
\item
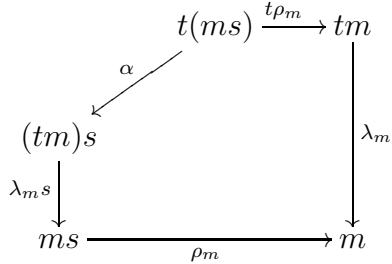
$$\xymatrix{& t(ms)\ar[dl]_{\alpha} \ar[r]^{t\rho_m} & tm \ar[dd]^{\lambda_m}\\
(tm)s \ar[d]_{\lambda_m s}\\
ms \ar[rr]_{\rho_m}& & m}$$
(left and right actions are compatible).
\end{enumerate}
\end{definition}

\begin{definition}
A \textbf{morphism} between two $s,t$-bimodules $(m,\rho_m,\lambda_m)$ and $(n,\rho_n,\lambda_n)$ is a 2-cell $\alpha :m\Rightarrow n$ such that the diagrams
$$\xymatrix{ tm \ar[r]^{t\alpha}\ar[d]_{\lambda_m} & tn \ar[d]^{\lambda_n} & ms \ar[d]_{\rho_m} \ar[r]^{\alpha s} & ns \ar[d]^{\rho_n} \\
m \ar[r]_{\alpha} & n & m \ar[r]_{\alpha} & n}$$
commute.
\end{definition}

\begin{proposition}
Consider a bicategory $\mathcal{B}$ such that every $\mathcal{B}(A,B)$ has reflexive coequalizers that are preserved by composition. Then the monads in $\mathcal{B}$ together with the bimodules and the morphisms between them form a bicategory, which will be denoted by $\mathbf{Mod}\mathcal{B}$.
\end{proposition}

\begin{proof}
1. Let $(s:A\rightarrow A,\eta_s,\mu_s)$,$(t:B\rightarrow B,\eta_t,\mu_t)$ be monads in $\mathcal{B}$. Then the collection $\mathbf{Mod}\mathcal{B}(s,t)$ of $s,t$-bimodules and the morphisms between them is a category:\\
If $l,m,n$ are $s,t$-bimodules and $\alpha, \beta$ morphisms as in the diagram,
$$\xymatrixcolsep{0.3cm} \xymatrixrowsep{0.02cm} \xymatrix{ & \quad \Downarrow \alpha \\
s\ar@/^2pc/[rr]^l \ar@/_2pc/[rr]_n \ar[rr]^m & & t\\
& \quad \Downarrow \beta}$$
then the composite $\beta \circ \alpha $ in $\mathcal{B}(A,B)$ is a morphism of $s,t$-bimodules because the diagrams
$$\xymatrix{ tl \ar[d]_{\lambda_l} \ar[r]^{t\alpha} & tm \ar[d]_{\lambda_m} \ar[r]^{t\beta} & tn\ar[d]^{\lambda_n}\\
l \ar[r]_{\alpha} & m \ar[r]_{\beta} & n}$$
$$\xymatrix{ ls \ar[d]_{\rho_l} \ar[r]^{\alpha s} & ms \ar[d]_{\rho_m} \ar[r]^{\beta s} & ns \ar[d]^{\rho_n}\\
l \ar[r]_{\alpha} & m \ar[r]_{\beta} & n}$$
commute and $t(\beta \circ \alpha) =t\beta \circ t\alpha, (\beta \circ \alpha ) s = \beta s \circ \alpha s $ by the interchange law in $\mathcal{B}$. The identiy $1_m$ in $\mathcal{B}(A,B)$ is the identity in $\mathbf{Mod}\mathcal{B}(s,t)$.
\vspace{4mm} \\
2. Consider monads $s,t$ and $u$ on the objects $A,B$ and $C$ respectively. If $m:t\rightarrow u$, $ n:s\rightarrow t$ are bimodules then we define their composite to be the coequalizer $m \otimes n$ of the following reflexive pair:
$$\xymatrix{m*t*n \ar@<1ex>[r]^{\rho_m * n} \ar@<-1ex>[r]_{m* \lambda_n} & m* n \ar[r] & m \otimes n.}$$
One can show that $m\otimes n$ is a right $s$-module and a left $u$-module.

Then we can extend the definition to morphisms between bimodules. That is, if $\beta_1$ and $\beta_2$ are morphisms between bimodules as in the diagram
$$\xymatrixcolsep{0.2cm} \xymatrix{ s \ar@/^1pc/[rr]^{m_1}  \ar@/_1pc/[rr]_{n_1} & \quad \Downarrow \beta_1 & t \ar@/^1pc/[rr]^{m_2}  \ar@/_1pc/[rr]_{n_2} & \quad \Downarrow \beta_2 & u ,}$$
then we can consider $\beta_2 \otimes \beta_1$ to be the unique arrow we can get from the universal property of the coequalizer $m_2\otimes m_1$, as in the following diagram:
	$$\xymatrixcolsep{0.4cm} \xymatrixrowsep{0.4cm} \xymatrix{m_2*t*m_1 \ar@<1ex>[rr]^{\rho * m_1} \ar@<-1ex>[rr]_{m_2* \lambda} & & m_2* m_1 \ar[dr]_{\beta_2 *\beta1} \ar[rr] & & m_2 \otimes m_1 \ar@{-->}[dd]^{\beta_2 \times \beta_1}\\
& & & n_2 * n_1 \ar[dr] \\
& & & & n_2 \otimes n_1.}$$
One can show here that $\beta_2 \otimes \beta_1$ preserves the right and the left action.

3. For a monad $(s:A\rightarrow A, \eta_s ,\mu_s)$ the identity $s,s$-bimodule is the same $s:A\rightarrow A$ with the 2-cell $\mu_s :ss \Rightarrow s $ in the role of both left and right action. The associative and the unit laws for the monad $s$ imply that $s$ is indeed an $s,s$-bimodule.
\end{proof}

\begin{example} For a finitely complete category $\mathcal{E}$, with reflexive coequalizers preserved by pullback functors, the bicategory $\mathbf{Prof(\mathcal{E})}$ can be also defined as the bicategory $\mathbf{Mod}(\mathbf{Span(\mathcal{E})})$.
\end{example}

\begin{proposition}
If $\mathcal{B}(A,B)$ has finite products for all objects $A,B$ in $\mathcal{B}$, then for every pair of monads $(s:A\rightarrow A , \eta_s, \mu_s), (t:B \rightarrow B, \eta_t ,\mu_t)$, the category $\mathbf{Mod}\mathcal{B}(s,t)$ has finite products.
\end{proposition}
\begin{proof}
1) Let $m,n$ be two $s,t$-bimodules. Then we have the product $m \wedge n$  in $\mathcal{B}(A,B)$. We claim that $m\wedge n$ is an $s,t$-bimodule.

We have the 2-cells $$(m\wedge n) s \xRightarrow{\pi s} ms \xRightarrow{\rho_m} m \text{ and }$$
$$ (m\wedge n) s \xRightarrow{\rho s} ns \xRightarrow{\rho_n} n.$$
Define $\rho_{m\wedge n}$ to be the 2-cell $< \rho_m \circ \pi s, \rho_n \circ \rho s >$. Similarly, define $\lambda_{m\wedge n}$ to be $<\lambda_m \circ t \pi, \lambda_n \circ t\rho >$. As we see in the diagram
$$\xymatrixcolsep{0.2cm} \xymatrixrowsep{0.35cm} \xymatrix{ & \Downarrow \eta_s & & \Downarrow 1\\
A \ar@/^4pc/[rr]^{1_A} \ar@/_1pc/[rr]_s \ar@/^1pc/[rr]^s & \Downarrow 1  & A \ar@/^4pc/[rr]^{m\wedge n} \ar@/^1pc/[rr]^{m\wedge n} \ar@/_1pc/[rr]_m \ar@/_4pc/[rr]_m & \Downarrow \pi & B,\\
& & & \Downarrow 1 }$$
we have $\pi s \circ (m\wedge n) \eta_s =m \eta_s \circ \pi$. Since also $m$ is an $s,t$-bimodule, the following diagram commutes.
$$\xymatrix{ m\wedge n \ar[r]^{(m\wedge n) \eta_s \:\: } \ar[d]_{\pi} & (m\wedge n) s \ar[d]^{\pi s} \\
m \ar[r]^{m\eta_s} \ar[dr]_1 & ms \ar[d]^{\rho_m} \\
& m} $$
Similarly the diagram
$$\xymatrix{ m\wedge n \ar[r]^{(m\wedge n) \eta_s \:\: } \ar[d]_{\rho} & (m\wedge n) s \ar[d]^{\rho s} \\
n \ar[r]^{n\eta_s} \ar[dr]_1 & ns \ar[d]^{\rho_n} \\
& n} $$
commutes, and so does the triangle
$$\xymatrix{ m\wedge n \ar[r]^{(m \wedge n) \eta_s\:\:} \ar[dr]_1 & (m\wedge n)s \ar[d]^{\rho_{m\wedge n}} \\ & m\wedge n.}$$
The diagram below commutes because of the naturality  of $\alpha$, the interchange law, the product properties and the fact that $m$ is a right $s$-module.
$$\xymatrix{ & & (m\wedge n)(ss) \ar[rr]^{(m\wedge n)\mu_s} \ar[dd]^{\pi(ss)} \ar[ddll]_{\alpha} & & (m\wedge n) s \ar[dd]^{\pi s} \\
\\
((m\wedge n)s)s \ar[dd]_{<\rho_m\circ \pi s, \rho_n \circ \rho s>s} \ar[dr]^{(\pi s)s} & & m(ss) \ar[dl]_{\alpha} \ar[rr]^{m \mu_s} & & ms \ar[dd]^{\rho_m} \\
& (ms)s \ar[d]^{\rho_m s} \\
(m\wedge n) s \ar[r]_{\pi s} & ms \ar[rrr]_{\rho_m} & & & m}$$
Similarly we have the commutativity of the diagram
$$\xymatrix{& (m\wedge n)(ss) \ar[dl]^{\alpha} \ar[rr]^{(m\wedge n)\mu_s} & & (m\wedge n)s \ar[d]^{\rho s} \\
((m\wedge n)s)s \ar[d]_{\rho_{m\wedge n}s} & & & ns \ar[d]^{\rho_n} \\
(m\wedge n)s \ar[r]_{\rho s} & ns \ar[rr]_{\rho_n} & & n, }$$
and then of the diagram
$$\xymatrix{& (m\wedge n)(ss) \ar[dl]^{\alpha} \ar[rr]^{(m\wedge n)\mu_s} & & (m\wedge n)s \ar[dd]^{\rho_{m\wedge n}} \\
((m\wedge n)s)s \ar[d]_{\rho_{m\wedge n}s} \\
(m\wedge n)s \ar[rrr]_{\rho_{m\wedge n}} & & & n. }$$
So $m\wedge n$ is a right $s$-module. We can also show that it is a left $t$-module. Finally the diagrams 
$$\xymatrix{ & & t((m\wedge n)s) \ar[dd] \ar[rr]^{t \rho_{m\wedge n}} \ar[ddll]_{\alpha} & & t(m\wedge n) \ar[dd]^{t\pi} \\
\\
(t(m\wedge n))s \ar[dd]_{\lambda_{m\wedge n}s} \ar[dr] & & t(ms) \ar[dl]_{\alpha} \ar[rr]^{t\rho_m} & & tm \ar[dd]^{\lambda_m} \\
& (tm)s \ar[d]^{\lambda_m s} \\
(m\wedge n) s \ar[r]_{\pi s} & ms \ar[rrr]_{\rho_m} & & & m}$$
and
$$\xymatrix{ & & t((m\wedge n)s) \ar[dd]_{t(\rho s)} \ar[rr]^{t \rho_{m\wedge n}} \ar[ddll]_{\alpha} & & t(m\wedge n) \ar[dd]^{t\rho} \\
\\
(t(m\wedge n))s \ar[dd]_{\lambda_{m\wedge n}s} \ar[dr]^{(t\rho)s} & & t(ns) \ar[dl]_{\alpha} \ar[rr]^{t\rho_n} & & tn \ar[dd]^{\lambda_n} \\
& (tn)s \ar[d]^{\lambda_n s} \\
(m\wedge n) s \ar[r]_{\rho s} & ns \ar[rrr]_{\rho_n} & & & n}$$
commute because of the naturality of $\alpha$, the definition of $\rho_{m\wedge n}$ and $\lambda_{m\wedge n}$, the fact that $m$ and $n$ are $s,t$-bimodules. So the two actions are compatible and $m\wedge n$ is an $s,t$-bimodule.
\vspace{0.1cm}\\
2) We will show now that $\mathbf{Mod}\mathcal{B}(s,t)$ has a terminal object.

Let $\top:A \rightarrow B$ be the terminal of $\mathcal{B}(A,B)$. Define $\rho_{\top}$ to be the unique 2-cell from $\top s$ to $\top$ and $\lambda_{\top}$ the unique 2-cell from $t\top$ to $\top$. Then $\top$ becomes an $s,t$-bimodule with these actions, since all the diagrams in the definition commute by the universal property of the terminal object.
\end{proof}

\begin{lemma}
If $F:\mathcal{B}\rightarrow \mathcal{D}$ is a lax functor and $(s:A\rightarrow A, \eta_s, \mu_s)$ a monad in $\mathcal{B}$, then $(Fs:FA\rightarrow FA,F\eta_s \circ \iota, F\mu_s \circ \sigma)$ is a monad in $\mathcal{B}$.
\end{lemma}
\begin{proof}
It follows by the commutativity of the diagrams:
$$\xymatrix{ Fs \ar[r]^{\lambda^{-1}} \ar[dr]_1 & 1_{FA}Fs \ar[r]^{\iota Fs} & F1_A Fs \ar[r]^{F\eta_s Fs} \ar[d]_{\sigma} & FsFs \ar[d]^{\sigma} \\
& Fs \ar[r]^{F\lambda} \ar[rrdd]_1 & F1_A s \ar[r]^{F\eta_s s} & Fss \ar[dd]^{F\mu_s} \\
\\ & & & Fs }$$
and
$$\xymatrix{ & (FsFs)Fs \ar[r]^{\sigma Fs} & FssFs \ar[d]^{\sigma} \ar[r]^{F\mu_s Fs} & FsFs \ar[d]^{\sigma} \\
Fs(FsFs) \ar[ru]^{\alpha} \ar[d]_{Fs\: \sigma} & & FssFs \ar[r]^{F\mu_s s} & Fss \ar[dd]^{F\mu_s} \\
FsFss \ar[d]_{FsF\mu_s} \ar[r]^{\sigma} & Fs(ss) \ar[ru]^{\alpha} \ar[d]^{Fs \mu_s} \\
FsFs \ar[r]_{\sigma} & Fss \ar[rr]_{F\mu_s} & & Fs. }$$
\end{proof}

\section{Matrices} \label{Matrices}

The bicategory of $V$-matrices for a monoidal category $V$ was studied in \cite{Benabou73}. Later, in \cite{Variation}, the authors considered the more general case were instead of just a one-object bicategory, aka a monoidal category, we have matrices enriched in a general bicategory.

Consider a monoidal category $V= (V,\otimes, I, a,r,l)$ with coproducts such that the tensor distributes over coproducts. We can construct the bicategory $V\mbox{-}\mathbf{Mat}$ as follows.
\begin{enumerate}
\item The objects of $V\mbox{-}\mathbf{Mat}$ are (small) sets.
\item If $S$ and $T$ are sets, then an arrow $M:S \nrightarrow T$ is a functor $M:T\times S \rightarrow V$, i.e. a family of objects of $V$. We call the arrows of $V\mbox{-}\mathbf{Mat}$ \textbf{matrices}.
\item A 2-cell from $M:S \nrightarrow T$ to $N:S \nrightarrow T$ is a natural transformation from the functor $M$ to $N$.
\item The identity matrix $I_S$ on $S$ is the family $I_S(s,s')=\begin{cases} I & \text{ if } s=s' \\ 0 & \text{ otherwise }\\ \end{cases}$, where $0$ is the initial object.
\item If $M:S\nrightarrow T$ and $N:T\nrightarrow U$ are matrices then their composite $NM:S\nrightarrow U$ is given by $NM(u,s)= \bigoplus_{t\in T} N(u,t)\otimes M(t,s)$.
\end{enumerate}

\begin{proposition}
$V\mbox{-}\mathbf{Mat}$ is indeed a bicategory.
\end{proposition}
\begin{proof}
The vertical composition of the 2-cells is the usual vertical composition of the natural transformations. The horizontal composition is as follows: for matrices
$$\xymatrixcolsep{0.2cm} \xymatrix{ S \ar@/^1pc/[rr]^M  \ar@/_1pc/[rr]_{M'} & \quad \Downarrow \phi & T \ar@/^1pc/[rr]^ N \ar@/_1pc/[rr]_{N'} & \quad \Downarrow \psi & U , }$$
we have $\forall t \in T$ the arrows $\phi (t,s):M(t,s) \rightarrow M'(t,s)$ and $\psi (u,t):N(u,t)\rightarrow N'(u,t)$. So we define $\psi \circ \phi (u,s)$ to be the canonical arrow $\bigoplus_{t\in T} \psi (u,t) \otimes \phi(t,s)$.

We define now the associator: Consider matrices as in $$\xymatrix{ S \ar[r]^M & T \ar[r]^N & U \ar[r]^P & V.}$$
Then, $\forall(v,s)\in V\times S$, we have $$P(NM)(v,s)=\bigoplus_{u\in U} P(v,u)\otimes (\bigoplus_{t\in T}N(u,t)\otimes M(t,s))$$ and $$(PN)M(v,s)=\bigoplus_{t\in T} (\bigoplus_{u\in U} P(v,u)\otimes N(u,t))\otimes M(t,s).$$
But by the distributivity of the tensor product we have $$P(NM)(v,s) \xrightarrow{\cong} \bigoplus_{u\in U} \bigoplus_{t\in T} P(v,u) \otimes (N(u,t)\otimes M(t,s)) \xrightarrow{\cong}$$ $$ \bigoplus_{t\in T} \bigoplus_{u\in U} P(v,u) \otimes (N(u,t)\otimes M(t,s)) \xrightarrow{\alpha} \bigoplus_{t\in T} \bigoplus_{u\in U}( P(v,u) \otimes N(u,t))\otimes M(t,s)$$ $$ \xleftarrow{\cong} \bigoplus_{t\in T} (\bigoplus_{u\in U} P(v,u) \otimes N(u,t))\otimes M(t,s) = (PN)M(v,s).$$
So there is an isomorphism $P(NM)(v,s)\cong(PN)M(v,s)$ and this will be our associator $\alpha$. For the right and the left unitors we have $$ MI_S(t,s) = \bigoplus_{s'\in S} M(t,s') \otimes I(s',s) = (M(t,s)\otimes I)\bigoplus_{s'\in S, s'\neq s} (M(t,s)\otimes 0)$$ $$ \cong (M(t,s) \oplus 0 \cong M(t,s)\oplus I \cong M(t,s),$$
and similarly $I_TM(t,s) \cong M(t,s)$.

In the following diagram we write :
$$ A \text{ for } \bigoplus_{v\in V, u\in U, t\in T} Q(w,v)\otimes[ P(v,u) \otimes[ N(u,t) \otimes M(t,s)]],$$
$$ B \text{ for } \bigoplus_{v\in V, u\in U, t\in T}[Q(w,v)\otimes P(v,u)] \otimes [N(u,t) \otimes M(t,s)],$$
$$ C \text{ for } \bigoplus_{v\in V, u\in U, t\in T}Q(w,v)\otimes [[P(v,u) \otimes N(u,t)] \otimes M(t,s)],$$
$$ D \text{ for } \bigoplus_{v\in V, u\in U, t\in T} [Q(w,v)\otimes[ P(v,u) \otimes N(u,t)]] \otimes M(t,s), \text{ and}$$
$$ E \text{ for } \bigoplus_{v\in V, u\in U, t\in T} [[Q(w,v)\otimes P(v,u)] \otimes N(u,t)] \otimes M(t,s).$$
$$\xymatrixcolsep{0.2cm} \xymatrix{ & & & (QP)(NM) \ar[d]^{\cong} \ar[rrrdd]^{\alpha} \\
& & & B \ar[drr]^{\oplus \alpha}\\
Q(P(NM)) \ar[uurrr]^{\alpha} \ar[r]^{\cong} \ar[dddr]_{Q\alpha} & A \ar[urr]^{\oplus \alpha} \ar[ddr]_{\oplus Q(w,v)\otimes \alpha} & & & & E & ((QP)N)M \ar[l]_{\cong} \\
\\
& & C \ar[rr]_{\oplus \alpha} & & D \ar[uur]_{\oplus \alpha \otimes M(t,s)} \\
& Q((PN) M) \ar[ur]_{\cong} \ar[rrrr]_{\alpha} & & & & (Q(PN))M \ar[ul]_{\cong} \ar[ruuu]_{\alpha M} }$$
The inner pentagon commutes by the pentagon identity of $\alpha$ in $V$ and the functoriality of the coproduct. Each of the outer squares commutes as we see in the construction of $\alpha$ above. So the pentagon identity holds. Similarly, the triangle identity holds and $V\mbox{-}\mathbf{Mat}$ is indeed a bicategory.
\end{proof}

If $V$ is a category with finite products then it becomes a monoidal category with tensor the product and unit object the terminal. In that case we call $V$ a Cartesian monoidal category.

The next lemma shows that the discrete bicategory of sets and functions is embedded in the bicategory of maps in $V\mbox{-}\mathbf{Mat}$:

\begin{lemma}\label{matrices_maps}
For $V$ a Cartesian monoidal category with coproducts such that the tensor distributes over them, there is a pseudofunctor $M:\mathbf{Set} \rightarrow \mathbf{Map}(V\mbox{-}\mathbf{Mat})$ which is locally fully faithful.
\end{lemma}
\begin{proof}
We define a pseudofunctor $M:\mathbf{Set} \rightarrow \mathbf{Map}(V\mbox{-}\mathbf{Mat})$.

If $S$ is a set, $MS=S$.

If $f:S\rightarrow T$ is a function, define $Mf:S \nrightarrow T$ such that $$Mf (t,s) = \begin{cases}1, & \text{if } f(s)=t \\ 0, & \text{otherwise} \\ \end{cases}\quad .$$ Since $\mathbf{Set}(S,T)$ is discrete we only have $M1_f (t,s) = \begin{cases}1_1, & \text{if } f(s)=t \\ 1_0, & \text{otherwise} \\ \end{cases}.$

We will show that $Mf$ is a map. Define the matrix $Mf^* :T \nrightarrow S$ with $Mf^* (s,t)=Mf(t,s)$. The components of the unit $\eta$ for $(s,s')\in S\times S$, $\eta(s,s'):I_S (s,s') \rightarrow Mf^* Mf (s,s')$ are defined as follows.

If $s=s'$, then $I_S (s,s')=1$ and $Mf^* Mf (s,s)= \bigoplus_{t\in T} Mf^* (s,t) \otimes Mf(t,s) \cong Mf^* (s, f(s)) \otimes Mf(f(s),s) \cong 1 \otimes 1 \cong 1$.\\
If $s\neq s'$, then $I_S (s,s') =0$ and $Mf^* Mf (s,s') = \bigoplus_{t\in T} Mf^* (s,t) \otimes Mf(t,s') \cong [Mf^*(s,f(s)) \otimes 0 ] \oplus [0 \otimes Mf(f(s'),s')] \cong 0$.\\
So define $$\eta (s,s') = \begin{cases}1_1, & \text{if } s=s' \\ 1_0, & \text{otherwise} \\ \end{cases}\quad .$$
For the counit $\epsilon : Mf\:Mf^* \rightarrow I_T$, let $(t,t') \in T\times T$.

Then if $t=t'$, $I_T (t,t) = 1$. If there exists an $s\in S$ with $f(s)=t$, then $Mf \: Mf^* (t,t) \cong 1$, while, if there is no such $s$, $Mf \: Mf^* (t,t) \cong 0$. \\
If $t\neq t'$, $Mf \: Mf^* (t,t') = \bigoplus _{s \in S} Mf (t,s) \otimes Mf^* (s,t')$ and $\forall s$, $ Mf (t,s) \otimes Mf^* (s,t') \cong 0$.\\
So we define $$ \epsilon (t,t') = \begin{cases} 1_1, & \text{if } t =t' =f(s) \text{ for some } s\in S \\
!:0\rightarrow 1, &\text{if } t=t' \text{ and } \forall s, f(s)\neq t\\
1_0, & \text{otherwise} \\ \end{cases} \quad.$$
For the triangle identities we have $$Mf \eta (t,s) = \bigoplus_{s' \in S} 1_{Mf} (t,s') \otimes \eta (s',s) \cong 1_{Mf}(t,s) \otimes 1_1 \cong 1_{Mf(t,s)} $$ and $$\epsilon Mf (t,s) = \bigoplus_{t'\in T} \epsilon (t,t') \otimes 1_{Mf}(t',s) \cong \epsilon(t,t) \otimes 1_{Mf} (t,s)$$ $$ \cong  \begin{cases} 1_1\otimes 1_1, & \text{if } f(s)=t \\
1_1 \otimes 1_0, &\text{if } \exists s'\in S, s'\neq s, \text{ with } f(s')=t \\
! \otimes 1_0, & \text{otherwise} \\ \end{cases} $$ $$ \cong \begin{cases} 1_1, & \text{if } f(s)=t \\ 1_0, & \text{otherwise} \\ \end{cases} \cong 1_{Mf(t,s)} .$$ So $(\epsilon Mf ) (Mf \eta ) = 1_{Mf}$ and similarly $(Mf^* \epsilon)(\eta Mf^*)=1_{Mf^*}$.

We now show that $M:\mathbf{Set}(S,T) \rightarrow \mathbf{Map}(V\mbox{-}\mathbf{Mat})(S,T)$ is fully faithfull:

For a function $f:S\rightarrow T$ and the matrix $Mf (t,s) = \begin{cases}1, & \text{if } f(s)=t \\ 0, & \text{otherwise} \\ \end{cases}$, if $\beta: Mf \rightarrow Mf$, then $\beta(t,s) = \begin{cases} 1_1 , & \text{if } f(s)=t \\ 1_0, & \text{otherwise} \\ \end{cases} = M1_f$.

\end{proof}

\begin{proposition}
If $V=(V,\times, 1,q,p)$ is a cartesian monoidal category then the bicategory $V\mbox{-}\mathbf{Mat}$ has finite products locally.
\end{proposition}
\begin{proof}
1) To prove that $V\mbox{-}\mathbf{Mat}(S,T)$ has finite products, it suffices to prove that it has terminal and products. Consider sets $S$ and $T$. Then define $\top_{S,T}:S\nrightarrow T$ with $\top_{S,T} (t,s)= 1, \: \forall (t,s)\in T\times S$.

If $M:S\nrightarrow T$ is a matrix, then $\forall (t,s)$ there exists a unique arrow $\tau (t,s):M(t,s) \rightarrow 1$ which will work as a component for a unique 2-cell $\tau :M \Rightarrow \top_{S,T}$.\\
For matrices $M:S\nrightarrow T, \: N:S\nrightarrow T$ define $M\otimes N:S\nrightarrow T$ with $M\otimes N(t,s)=M(t,s)\times N(t,s), \: \forall (t,s)\in T\times S $. The projections $\pi :M\otimes N \Rightarrow M$ and $\rho:M\otimes N \Rightarrow N$ are given by the projections in $V$, $\pi(t,s):M(t,s)\times N(t,s) \rightarrow M(t,s)$ and $ \rho(t,s):M(t,s)\times N(t,s) \rightarrow N(t,s), \: \forall(t,s)\in T\times S$.
\end{proof}

Similarly to the case of modules and monads in a bicategory, we can see now that every hom-category in $V\mbox{-}\mathbf{Mat}$ has finite products. Again though, we don't know if $\mathbf{Map}(V\mbox{-}\mathbf{Mat})$ has finite bicategorical products, and for this to be true we may need more conditions than $V$ being Cartesian. If for example the pseudofunctor $M$ of \Cref{matrices_maps} is moreover locally essential surjective, then this would imply that $\mathbf{Map}(V\mbox{-}\mathbf{Mat}) \simeq \mathbf{Set}$, and so $\mathbf{Map}(V\mbox{-}\mathbf{Mat})$ would have finite products.

\chapter{Double Categories} \label{Double Categories}

\section{Introduction to Double Categories} \label{Introduction to Double Categories}

In this chapter we study double categories, we extend some of the concepts presented in the classic ``Formal Theory of Monads'' paper to double categories, and lastly, we talk about fibrant double categories.

The purpose of this first section is to introduce the reader to the theory of double categories. Double categories first appeared in Ehresmann's \cite{ehresmann65}. In 1989, Par\'e introduced double limits in his CT meeting presentation \cite{Pare1989} in Bangor. Later, together with Grandis, they wrote a series of four papers on the subject: \cite{LimitsGP}, \cite{AdjointsGP}, \cite{LaxKanGP}, and \cite{KanGP}. Other standard references are Shulman's papers \cite{Shulman2008} or \cite{Shulman2010}, as well as Niefield's \cite{SusanGlue} and \cite{SusanExp}. In the earlier literature, double categories were defined as internal categories in the 2-category of categories $\mathcal{C}at$. This leads to a stronger notion of double categories than the one we are using in this thesis. Essentially, according to this definition, a double category has two types of arrows $\--$called vertical and horizontal$\--$ and in both directions we have strict associativity and strict identities. In our case, we allow one of the directions, in particular the horizontal one, to be associative and unitary up to isomorphism. These are often called pseudo or weak double categories, but for us here they are just called double categories.

\begin{definition}
A \textbf{(pseudo) double category} $\mathbb{D}$ consists of:
\begin{enumerate}[label=\roman*.]
\item A category $D_0$, whose objects are called \textbf{objects} of the double category and its arrows are called \textbf{vertical arrows}. We will be using uppercase letters towards the beginning and the end of the Latin alphabet $A, B, X, Y, \dots$ for the objects of $\mathbb{D}$, and lowercase letters $f,g,h, \dots$ for the vertical arrows. We call $D_0$ the \textbf{vertical category} of $\mathbb{D}$.
\item A category $D_1$, whose objects are called \textbf{horizontal arrows} and its arrows are called \textbf{cells}. We will be using uppercase letters in the middle of the Latin alphabet $F, G, M, N, \dots$ for the horizontal arrows, and Greek letters $\alpha, \beta, \gamma, \dots$ for the cells.
\item Functors $$U:D_0 \rightarrow D_1: A \mapsto U_A, f \mapsto U_f$$ $$S,T:D_1 \rightarrow D_0 \text{ and}$$ $$\odot:D_1 \times_{D_0} D_1 \rightarrow D_1,$$ where $D_1 \times_{D_0}D_1$ is the pullback of $(S,T)$, such that $S(U_A) = A = T(U_A)$, $S(M \odot N) = SN$ and $T(M\odot N)=TM$.
\item Natural isomorphisms with components the cells $$\alpha: (M \odot N) \odot P \rightarrow M \odot (N \odot P), $$ $$\lambda: U_{TM} \odot M \rightarrow M \text{ and}$$ $$ \rho: M \odot U_{SM} \rightarrow M,$$ for horizontal arrows $M$, $N$ and $P$, such that $S(\alpha), T(\alpha), S(\lambda), T(\lambda), S(\rho) \text{ and } T(\rho)$ are the identity vertical arrows and the following diagrams commute:
$$\xymatrixcolsep{0.3cm} \xymatrix{
& &(M \odot N) \odot (P \odot Q) \ar[rd]^{\alpha} \\
& ((M \odot N) \odot P) \odot Q \ar[ru]^{\alpha} \ar[d]_{\alpha \odot Q} & & M \odot (N  \odot( P  \odot Q))\\
& (M\odot (N \odot P)) \odot Q \ar[rr]_{\alpha} & & M \odot ((N \odot P )  \odot Q) \ar[u]_{M \odot \alpha} &,
} $$
$$\xymatrixcolsep{0.3cm} \xymatrix{
(M \odot U) \odot N \ar[rr]^{\alpha} \ar[rd]_{\rho \odot N} & & M \odot (U \odot N) \ar[dl]^{M \odot \lambda}\\
& M \odot N &. } $$
\end{enumerate}
\end{definition}

\begin{remark}
Throughout this thesis, we will be using juxtaposition and $1$ for the composition and identity in $D_0$, and the following diagrammatic notation for vertical arrows, vertical identity, horizontal arrows, and cells respectively:
$$\xymatrixcolsep{0.2cm} \xymatrixrowsep{0.2cm} \xymatrix{
{} \ar[dd] & & {}\ar@{=}[dd] & & & & & & {} \ar[dd] \haru{} & & {} \ar[dd]\\
&, & &, & {} \haru{} & & {} & \text{, and}\\
{}& & {} & & & & & & {} \hard{} & & {}.}$$
 Also, when the source and target arrows are understood from context, we simply write $\frac{\alpha}{\beta}$ for the vertical composite of composable in $D_0$ cells $\alpha$ and $\beta$, and $\alpha|\beta$ for the horizontal composite of composable in $D_1$ cells.
\end{remark}

\begin{remark} The definition above shows that a strict double category, i.e. a double category where instead of isomorphisms $a$, $l$ and $r$, we have equalities, can be seen as an internal category in the category $\mathbf{Cat}$.
\end{remark}

\begin{definition}
A cell $\phi$ such that $S(\phi)=1$ and $T(\phi)=1$ will be called \textbf{globular}.
\end{definition}

\begin{definition}
Consider a double category $\mathbb{D}$. Then we can form a bicategory with objects the objects of $\mathbb{D}$, arrows the horizontal arrows and 2-cells all the globular cells. This will be called the \textbf{horizontal bicategory} of $\mathbb{D}$ and will be denoted by $\mathcal{H(\mathbb{D})}$.
\end{definition}

\begin{example}
The double category $\mathbb{S}\mathrm{et}$, with objects the sets, vertical arrows the functions and horizontal arrows the relations. We say that there is a cell of the form
$$\xymatrixcolsep{0.2cm} \xymatrixrowsep{0.2cm} \xymatrix{ A \haru{R} \ar[dd]_f & & B \ar[dd]^g \\
\\ X \hard{S} & & Y ,}$$ where $f$ and $g$ are functions and $R$ and $S$ are relations, if for every $(a,b) \in A \times B$, $aRb \Rightarrow (fa)S(gb)$. The horizontal composite of relations $\xymatrixcolsep{0.2cm} \xymatrixrowsep{0.2cm} \xymatrix{ A \haru{R} & & B \haru{S} & & C}$ is the relation $a(S \odot R)c \Leftrightarrow (\exists b)(a R b \wedge b S c)$, and the horizontal unit $U_A : \srarrow{A}{A}$ is the diagonal relation.
\end{example}

\begin{example}
The double category $\mathbb{S}\mathrm{pan}(\mathcal{E})$, for $\mathcal{E}$ a category with pullbacks, with objects the objects of $\mathcal{E}$, vertical arrows the arrows of $\mathcal{E}$ and horizontal arrows spans between objects of $\mathcal{E}$, i.e. a horizontal arrow $R:\srarrow{A}{B}$ is a pair of arrows of the form $\xymatrixcolsep{0.4cm} \xymatrix{ A & R \ar[l]_{r_0} \ar[r]^{r_1} & B}$. A cell
$$\xymatrixcolsep{0.2cm} \xymatrixrowsep{0.2cm} \xymatrix{
& & R \ar[dll]_{r_0} \ar[drr]^{r_1} \\
A \ar[dd]_f & & & & B \ar[dd]^g \\
\\ X & & & & Y \\
& & S \ar[ull]^{s_0} \ar[urr]_{s_1} }$$
is an arrow $\alpha: R \rightarrow S$ of $\mathcal{E}$, such that $s_0 \alpha = f r_0$ and $s_1 \alpha = g r_1$. The horizontal composition is given by pullbacks and the identity span $U_A$ is the pair $\xymatrixcolsep{0.4cm} \xymatrix{ A & A \ar[l]_{1} \ar[r]^{1} & A}$.
\end{example}

\begin{example}
The double category $V\mbox{-}\mathbb{M}\mathrm{at}$, for $V$ a monoidal category with coproducts such that the tensor distributes over coproducts. The objects of $V\mbox{-}\mathbb{M}\mathrm{at}$ are sets, vertical arrows are functions and horizontal arrows are $V$-matrices between sets, i.e. functions from $T\times S$ to the objects of $V$, for $T$ and $S$ sets. If $M$ and $N$ are matrices and $f, \: g$ functions as in the square
$$\xymatrixcolsep{0.2cm} \xymatrixrowsep{0.2cm} \xymatrix{ S \haru{M} \ar[dd]_f & & T \ar[dd]^g \\
\\ X \hard{N} & & Y ,}$$ then a cell $\phi:M \rightarrow N$ is a family of arrows $\phi(t,s):M(t,s) \rightarrow N(g\: t,fs)$ in $V$.
\end{example} 

\begin{definition}
Let $\mathbb{B}$ and $\mathbb{D}$ be double categories. A \textbf{lax double functor} $F:\mathbb{B}\rightarrow \mathbb{D}$ consists of:
\begin{enumerate}[label=\roman*.]
\item two functors $F_0 : B_0 \rightarrow D_0$ and $F_1 :B_1 \rightarrow D_1$, such that $S \circ F_1 = F_0 \circ S$ and $T\circ F_1 = F_0 \circ T$, and
\item natural transformations $$F_{\odot}: F_1 M \odot F_1 N \rightarrow F_1 (M \odot N) \text{ and}$$ $$F_U : U_{F_0 A} \rightarrow F_1 (U_A),$$ with globular components, i.e. $S(F_{\odot}), T(F_{\odot}), S(F_U) \text{ and } T(F_U)$ are all identities, such that the following diagrams commute:
\end{enumerate}
$$\xymatrixcolsep{0.4cm} \xymatrix{
& & & F_1M \odot U_{F_0 SM} \ar[d]^{F_1 M \odot F_U} \ar[ddlll]_{\rho} & & U_{F_0 TM} \odot F_1 M \ar[d]_{F_U \odot F_1 M} \ar[ddrrr]^{\lambda} \\
& & & F_1M \odot F_1 U_{SM} \ar[d]^{F_{\odot}} & & F_1 U_{TM} \odot F_1 M \ar[d]_{F_{\odot}} \\
 F_1 M & & & F_1(M \odot U_{SM}) \ar[lll]^{F_1 r \qquad} &, &  F_1(U_{TM} \odot M) \ar[rrr]_{F l} & & & F_1 M }
$$
and
$$ \xymatrix{ \left( F_1 M \odot F_1 N \right) \odot F_1 P \ar[r]^{\alpha} \ar[d]_{F_{\odot} \odot F_1 P } & F_1 M \odot \left( F_1 N \odot F_1 P \right) \ar[d]^{F_1 M \odot F_{\odot}} \\
F_1(M \odot N) \odot F_1 P \ar[d]_{F_{\odot}} & F_1 M \odot F_1 (N\odot P) \ar[d]^{F_{\odot}} \\
F_1 \big( (M\odot N) \odot P \big) \ar[r]_{F_1\alpha} & F \big( M \odot (N \odot P ) \big) }
$$
\end{definition}

\begin{definition}\label{double_functor}
A lax double functor $F:\mathbb{B}\rightarrow \mathbb{D}$ between double categories is said to be \textbf{normal} if the natural transformation $F_U$ is an isomorphism. If in addition $F_{\odot}$ is an isomorphism we say that $F$ is a \textbf{(pseudo) double functor}.
\end{definition}

We can dually define the notions of \textbf{\textit{oplax double functor}} and \textbf{\textit{opnormal double functor}}.

\begin{definition} \label{vertical_transf}
Consider two double categories $\mathbb{B}$ and $\mathbb{D}$ and two lax double functors between them, $F, G : \mathbb{B} \rightarrow \mathbb{D}$. A \textbf{vertical natural transformation} $\phi : F \rightarrow G$ consists of natural transformations $\phi_0 :F_0 \rightarrow G_0$ and $\phi_1 :F_1 \rightarrow G_1$ such that $S(\phi_M) = \phi_{SM}$ and $T(\phi_M)= \phi_{TM}$ and the following hold for horizontal arrows $M:\srarrow{A}{B}, N: \srarrow{B}{C}$:
$$\xymatrixcolsep{0.2cm} \xymatrixrowsep{0.2cm} \xymatrix{ FA \haru{FM} \ar@{=}[dd] & & FB \haru{FN} & & FC \ar@{=}[dd] & & FA \ar[dd]_{{\phi_0}_A} \haru{FM} & & FB \ar[dd]_{{\phi_0}_B} \haru{FN} & & FC \ar[dd]^{{\phi_0}_C}\\
& & F_{\odot} & & & & & {\phi_1}_M & & {\phi_1}_N\\
FA \ar[dd]_{{\phi_0}_A}\ar[rrrr]|{\vstretch{0.60}{|}} \ar[rrrr]^{F(N\odot M)} & & & & FC \ar[dd]^{{\phi_0}_C} & = & GA \ar@{=}[dd] \hard{GM} & & GB \hard{GN} & & GC \ar@{=}[dd]\\
& & {\phi_1}_{N \odot M} & & & & & & G_{\odot}\\
GA \ar[rrrr]|{\vstretch{0.60}{|}} \ar[rrrr]_{G(N\odot M)} & & & & GC & & GA \ar[rrrr]|{\vstretch{0.60}{|}} \ar[rrrr]_{G(N \odot M)} & & & & GC }$$
and
$$\xymatrixcolsep{0.2cm} \xymatrixrowsep{0.2cm} \xymatrix{ FA \haru{U_{FA}} \ar@{=}[dd] & & FA \ar@{=}[dd] & & FA \ar[dd]_{{\phi_0}_A} \haru{U_{FA}} & & FA \ar[dd]^{{\phi_0}_A} \\
& F_U & & & & U_{{\phi_0}_A}\\
FA \ar[dd]_{{\phi_0}_A}\haru{F(U_A)} & & FA \ar[dd]^{{\phi_0}_A} & = & GA \ar@{=}[dd] \hard{U_{GA}} & & GA \ar@{=}[dd]\\
& {\phi_1}_{U_A} & & & & G_U\\
GA \hard{G(U_A)} & & GA & & GA \hard{G(U_A)} & & GA }$$
\end{definition}

\begin{remark}
Essentially, a vertical natural transformation assigns to each object a vertical arrow and to each horizontal arrow a cell, whose source and target are the vertical components on the source and target of the horizontal arrow. We will often write just $\phi$ for bor $\phi_0$ and $\phi_1$.
\end{remark}

We can form a 2-category with objects the double categories, arrows the double functors and 2-cells the vertical transformations (\cite{AdjointsGP}). It will be denoted by $\mathbf{DblCat}$. The composition of double functors is the usual composition of the two functors in the definition of a double functor. Similarly, the vertical and the horizontal composition of vertical natural transformations comes from the horizontal and vertical composition of natural transformations. Moreover we can define a 2-category with objects the double categories, arrows the lax double functors and 2-cells the vertical transformation, which will be denoted by $\mathbf{DblCat}_{\mathcal{L}}$.

If $\mathbb{B}$ and $\mathbb{D}$ are double categories, their \textbf{\textit{product double category}} $\mathbb{B} \times \mathbb{D}$ consists of the categories $B_0 \times D_0$  and $B_1 \times D_1$, the functors $U_{\mathbb{B}} \times U_{\mathbb{D}}, \: S_{\mathbb{B}} \times S_{\mathbb{D}},\: T_{\mathbb{B}} \times T_{\mathbb{D}}, \: \odot_{\mathbb{B}} \times \odot_{\mathbb{D}}$ and the natural transformations $\alpha_{\mathbb{B}} \times \alpha_{\mathbb{D}}, \: \lambda_{\mathbb{B}} \times \lambda_{\mathbb{D}}, \: \rho_{\mathbb{B}} \times \rho_{\mathbb{D}}$. This is exactly the product of $\mathbb{B}$ and $\mathbb{D}$ in the 2-category $\mathbf{DblCat}$.

Note now that, for any double category $\mathbb{D}$, we can build a double functor $\Delta: \mathbb{D} \rightarrow \mathbb{D} \times \mathbb{D}$ such that $\Delta_0(A)=(A,A), \: \Delta_0(f)=(f,f), \: \Delta_1(M)=(M,M) \text{ and } \Delta_1 (\phi)=(\phi,\phi)$. We will call $\Delta$ the \textbf{\textit{diagonal functor}} on $\mathbb{D}$.\par
Also, we can construct the double category $\mathbbm{1}$ which has only one object, one vertical arrow, one horizontal arrow and one cell, and for every double category $\mathbb{D}$ we have a unique functor $!: \mathbb{D} \rightarrow \mathbbm{1}$.

\begin{definition}
Let $F:\mathbb{B}\rightarrow \mathbb{D}$ be a double functor. We say that $F$ has a \textbf{right adjoint} $G:\mathbb{D}\rightarrow \mathbb{B}$, or that we have a \textbf{(double) adjunction} $F \dashv G$, if there is an adjunction $(F \dashv G :\mathbb{D}\rightarrow \mathbb{B})$ in the 2-category $\mathbf{DblCat}$. If $F$ and $G$ are lax double functors instead and we have an adjunction $(F \dashv G :\mathbb{D}\rightarrow \mathbb{B})$ in $\mathbf{DblCat}_{\mathcal{L}}$, then we say that $F$ has a \textbf{lax right adjoint} $G$, or that $F \dashv G$ is a \textbf{lax (double) adjunction}.
\end{definition}

\begin{definition}
Let $F:\mathbb{B}\rightarrow \mathbb{D}$ be a double functor and $M: \srarrow{C}{D} $ a horizontal arrow in $\mathbb{D}$. A \textbf{universal cell from $F$ to $M$} is a cell
$$\xymatrixcolsep{0.2cm} \xymatrixrowsep{0.2cm} \xymatrix{ FA \haru{FN} \ar[dd]_f & & FB \ar[dd]^g \\ & \alpha \\ C \hard{M} & & D }$$
for some horizontal arrow $N:\srarrow{A}{B}$ in $\mathbb{B}$, so that the vertical arrows $f$ and $g$ are universal arrows from $F_0$ to $C$ and $D$, respectively, in $D_0$ and every other cell of the form
$$\xymatrixcolsep{0.2cm} \xymatrixrowsep{0.2cm} \xymatrix{ FA' \haru{FN'} \ar[dd]_{f'} & & FB' \ar[dd]^{g'} \\ & \alpha' \\ C \haru{M} & & D }$$ factors through $\alpha$ via a unique cell $\beta : N' \rightarrow N$. That is, if $f'$ factors through $f$ via the vertical arrow $h$ and $g'$ factors through $g$ via the vertical arrow $k$ then the following holds:
$$\xymatrixcolsep{0.2cm} \xymatrixrowsep{0.2cm} \xymatrix{ & & & & FA' \haru{FN'} \ar[dd]_{Fh} & & FB' \ar[dd]^{Fk}\\
FA' \haru{FN'} \ar[dd]_{f'} & & FB' \ar[dd]^{g'} & & & F\beta \\
 & \alpha' & & = & FA \haru{FN} \ar[dd]_f & & FB \ar[dd]^g \\
C \hard{M} & & D & & & \alpha \\
& & & & C \hard{M} & & D
}$$
\end{definition}

We can also define the dual, i.e. a \textbf{universal cell from $M$ to $F$}.

\begin{remark}
Consider an adjunction $F \dashv G :\mathbb{D}\rightarrow \mathbb{B}$ between double categories. That means that we have vertical natural transformations $\eta: 1_{\mathbb{B}} \rightarrow GF$ and $\epsilon : FG \rightarrow 1_{\mathbb{D}}$ such that $(G\epsilon)(\eta G)=1_G$ and $(\epsilon F)(F \eta)=1_F$. In other words, the functors $F_0:B_0\rightarrow D_0$ and $G_0:D_0\rightarrow B_0$ are equipped with the natural transformations $\eta_0: 1_{B_0} \rightarrow G_0F_0$ and $\epsilon_0 : F_0G_0 \rightarrow 1_{D_0}$, for which the triangle identities $(G_0\epsilon_0)(\eta_0 G_0)=1_{G_0}$ and $(\epsilon_0 F_0)(F_0 \eta_0)=1_{F_0}$ hold. So the functor $G_0$ is the right adjoint of $F_0$. Similarly, we see that the functor $G_1$ is the right adjoint of the functor $F_1$. Also, the following hold:
\begin{enumerate}
\item Each vertical arrow $\epsilon_D$ is universal from $F_0$ to $D$.
\item Each vertical arrow $\eta_B$ is universal from $B$ to $G_0$.
\item Each cell $\epsilon_M$ is universal from $F$ to $M$.
\item Each cell $\eta_N$ is universal from $N$ to $G$.
\end{enumerate}
\end{remark}

\begin{theorem}\cite{LimitsGP}\label{adjoint_grandis_pare}
An oplax double functor $F:\mathbb{B} \rightarrow \mathbb{D}$ has a lax double right adjoint $G$ if and only if the following hold:
\begin{enumerate}[label=\roman*.]
\item For every object $D$ in $\mathbb{D}$ there is a universal vertical arrow $(G_0D, \epsilon_D: F(G_0D) \rightarrow D)$ from the functor $F_0$ to $D$.
\item For every horizontal arrow $M:\srarrow{C}{D}$ in $\mathbb{D}$ there is a universal cell $(G_1M, \epsilon_M:F(G_1M) \rightarrow M)$ from the functor $F_1$ to $M$, such that $S(\epsilon_M)=\epsilon_C$ and $T(\epsilon_M)=\epsilon_D$.
\end{enumerate}
\end{theorem}

\section{Monads, Comonads and Co-pointed Arrows in Double Categories} \label{Monads and Comonads in Double Categories}

In \cite{Fiore2011b}, Fiore, Gambino and Kock extend the basic concepts of Street's formal theory of monads \cite{Street1972} to double categories. Here we review these notions and we also consider, instead of comonads, arrows equipped only with a counit. We will call these co-pointed arrows, and they will play a central role in Chapter 5 for our characterization theorem. In the following consider a double category $\mathbb{D}$. All the double categories that we discuss in this section are special cases of the following:

\begin{definition}\cite{Fiore2011b}
The double category $\mathbf{End(\mathbb{D})}$ consists of:
\begin{enumerate}
\item Objects are horizontal endomorphisms $(A, S:\srarrow{A}{A})$ in $\mathbb{D}$.
\item A vertical arrow from $(A,S)$ to $(B,T)$ is a vertical morphism $(f,\phi)$, i.e. a vertical arrow $f:A\rightarrow B$ in $\mathbb{D}$, equipped with a cell
$$\xymatrixcolsep{0.2cm} \xymatrixrowsep{0.2cm} \xymatrix{ A \haru{S} \ar[dd]_f & & A \ar[dd]^f \\ & \phi \\ B \hard{T} & & B. }$$
\item Horizontal arrows are horizontal maps, i.e. pairs of the form $(F, \alpha):(A,S) \rightarrow (A',S')$, where $F$ is a horizontal arrow $\srarrow{A}{A'}$ and $\alpha$ is a cell
$$\xymatrixcolsep{0.2cm} \xymatrixrowsep{0.2cm} \xymatrix{ A \ar@{=}[dd] \haru{S} & & A \haru{F} & & A' \ar@{=}[dd] \\
& & \alpha \\
A \hard{F} & & A' \hard{S'} & & A'. }$$
\item A cell
$$ \xymatrixcolsep{0.2cm} \xymatrixrowsep{0.2cm} \xymatrix{
(A,S) \haru{(F,\alpha_F)} \ar[dd]_{(f,\phi)} & & (A',S') \ar[dd]^{(f',\phi')} \\
& \theta \\
(B,T) \hard{(G,\alpha_G)} & & (B',T') }$$
is a cell
$$ \xymatrixcolsep{0.2cm} \xymatrixrowsep{0.2cm} \xymatrix{
A \haru{F} \ar[dd]_{f} & &A' \ar[dd]^{f'} \\
& \theta \\
B \hard{G} & & B' }$$
such that
$$ \xymatrixcolsep{0.2cm} \xymatrixrowsep{0.2cm} \xymatrix{
A \ar@{=}[dd] \haru{S} & & A \haru{F} & & A' \ar@{=}[dd] & & A \ar[dd]_f \haru{S} & & A \ar[dd]_f \haru{F} & & A' \ar[dd]^{f'} \\
& & \alpha_F & & & & & \phi & & \theta \\
A \ar[dd]_f \haru{F} & & A' \ar[dd]_{f'} \haru{S'} & & A' \ar[dd]^{f'} & = & B \hard{T} \ar@{=}[dd] & & B \hard{G} & & B' \ar@{=}[dd] \\
& \theta & & \phi' & & & & & \alpha_G \\
B \hard{G} & & B' \hard{T'} & & B' & & B \hard{G} & & B' \hard{T'} & & B'. }$$
\end{enumerate}
\end{definition}

We now introduce the basic definitions for monads and the morphisms between them. We will not use them in the next 2 chapters, however, they will come in handy in the 6th one, when we will need to talk about modules in a double category.

\begin{definition}\cite{Fiore2011b}
A \textbf{monad} in $\mathbb{D}$ is an endomorphism $(A, S:\srarrow{A}{A})$ on an object $A$, equipped with globular cells $\eta:U_A \rightarrow S$ and $\mu:S\odot S \rightarrow S$ called unit and multiplication respectively, such that the following hold:
$$\xymatrixcolsep{0.2cm} \xymatrixrowsep{0.2cm} \xymatrix{ A \haru{U_A} \ar@{=}[dd] & & A \haru{S} \ar@{=}[dd] & & A \ar@{=}[dd] \\
& \eta & & 1_S & & & A \ar@{=}[dd] \haru{U_A} & & A \haru{S}  & & A \ar@{=}[dd] \\
A \ar@{=}[dd] \haru{S} & & A \haru{S} & & A \ar@{=}[dd] & = & & & r\\
& & \mu & & & & A \ar[rrrr]_S \ar[rrrr]|{\vstretch{0.60}{|}} & & & & A,\\
A \ar[rrrr]_S \ar[rrrr]|{\vstretch{0.60}{|}} & & & & A }$$
$$\xymatrixcolsep{0.2cm} \xymatrixrowsep{0.2cm} \xymatrix{ A \haru{S}  \ar@{=}[dd] & & A \haru{U_A} \ar@{=}[dd] & & A \ar@{=}[dd] \\
& 1_S & & \eta & & & A \ar@{=}[dd]\haru{S} & & A \haru{U_A} & & A \ar@{=}[dd]\\
A \ar@{=}[dd] \haru{S} & & A \haru{S} & & A \ar@{=}[dd] & = & & & l\\
& & \mu & & & & A \ar[rrrr]_S \ar[rrrr]|{\vstretch{0.60}{|}} & & & & A, \\
A \ar[rrrr]_S \ar[rrrr]|{\vstretch{0.60}{|}} & & & & A }$$
$$\xymatrixcolsep{0.2cm} \xymatrixrowsep{0.2cm} \xymatrix{A \ar@{=}[dd] \haru{S} & & A \ar@{=}[dd] \haru{S} & & A \haru{S} & & A \ar@{=}[dd] & & A \ar@{=}[dd] \haru{S} & & A \haru{S} & & A \ar@{=}[dd] \haru{S} & & A \ar@{=}[dd] \\
& 1_S & & & \mu & & & & & & \mu & & & 1_S \\
A \ar@{=}[dd] \hard{S} & & A \ar[rrrr]_S \ar[rrrr]|{\vstretch{0.60}{|}} & & & & A \ar@{=}[dd] & = & A \ar@{=}[dd] \ar[rrrr]_S \ar[rrrr]|{\vstretch{0.60}{|}}& & & & A \hard{S} & & A \ar@{=}[dd] \\
& & & \mu & & & & & &  & & \mu \\
A \ar[rrrrrr]_S \ar[rrrrrr]|{\vstretch{0.60}{|}} & & & & & & A & & A \ar[rrrrrr]_S \ar[rrrrrr]|{\vstretch{0.60}{|}} & & & & & & A }$$

A \textbf{monad morphism} from a monad $(A,S)$ to a monad $(B,T)$ is a vertical morphism $(f,\phi):(A,S) \rightarrow (B,T)$ such that:
$$\xymatrixcolsep{0.2cm} \xymatrixrowsep{0.2cm} \xymatrix{ A \haru{U_A} \ar@{=}[dd] & & A \ar@{=}[dd] & & A \ar[dd]_f \haru{U_A} & & A \ar[dd]^f \\
& \eta & & & & U_f\\
A \ar[dd]_f \haru{S} & & A \ar[dd]^f & = & B \ar@{=}[dd] \hard{U_B} & & B \ar@{=}[dd]\\
& \phi & & & & \eta \\
B \hard{T} & & B & & B \hard{T} & & B }$$
$$\xymatrixcolsep{0.2cm} \xymatrixrowsep{0.2cm} \xymatrix{ A \haru{S} \ar@{=}[dd] & & A \haru{S} & & A \ar@{=}[dd] & & A \ar[dd]_f \haru{S} & & A \ar[dd]_f \haru{S} & &A \ar[dd]^f\\
& & \mu & & & & & \phi & & \phi\\
A \ar[dd]_f \ar[rrrr]^S \ar[rrrr]|{\vstretch{0.60}{|}} & & & & A \ar[dd]^f & = & B \ar@{=}[dd] \hard{T} & & B \hard{T} & & B \ar@{=}[dd]\\
& & \phi & & & & & & \mu \\
B \ar[rrrr]_T \ar[rrrr]|{\vstretch{0.60}{|}} & & & & B & & B \ar[rrrr]_T \ar[rrrr]|{\vstretch{0.60}{|}} & & & & B. }$$

A \textbf{horizontal monad map} from a monad $(A,S)$ to a monad $(A',S')$ is a horizontal arrow $F:\srarrow{A}{A'}$, together with a cell
$$\xymatrixcolsep{0.2cm} \xymatrixrowsep{0.2cm} \xymatrix{ A \ar@{=}[dd] \haru{S} & & A \haru{F} & & A' \ar@{=}[dd] \\
& & \alpha \\
A \hard{F} & & A' \hard{S'} & & A', }$$
such that the following hold:
$$ \xymatrixcolsep{0.2cm} \xymatrixrowsep{0.2cm} \xymatrix{
& & & & & & & & A \ar@{=}[dd] \haru{S} & & A \ar@{=}[dd] \haru{S} & & A \haru{F} & & A' \ar@{=}[dd] \\
A \ar@{=}[dd] \haru{S} & & A \haru{S} & & A \ar@{=}[dd] \haru{F} & & A' \ar@{=}[dd] & & & & & & \alpha \\
& & \mu & & & & & & A \ar@{=}[dd] \haru{S} & & A \haru{F} & & A' \ar@{=}[dd] \haru{S'} & & A' \ar@{=}[dd] \\
A \ar@{=}[dd] \ar[rrrr]^{S} \ar[rrrr]|{\vstretch{0.60}{|}} & & & & A \haru{F} & & A' \ar@{=}[dd] & = & & & \alpha \\
& & & \alpha & & & & & A \ar@{=}[dd] \hard{F} & & A' \ar@{=}[dd] \hard{S'} & & A' \hard{S'} & & A' \ar@{=}[dd] \\
A \ar[rrrr]_{F} \ar[rrrr]|{\vstretch{0.60}{|}} & & & & A' \hard{S'} & & A & & & & & & \mu' \\
& & & & & & & & A \hard{F} & & A' \ar[rrrr]_{S'} \ar[rrrr]|{\vstretch{0.60}{|}} & & & & A' }$$
$$ \xymatrixcolsep{0.2cm} \xymatrixrowsep{0.2cm} \xymatrix{
A \ar@{=}[dd] \haru{U_A} & & A \ar@{=}[dd] \haru{F} & & A' \ar@{=}[dd] \\
& \eta & & & & & A \ar@{=}[dd] \haru{F} & & A' \haru{U_{A'}} \ar@{=}[dd] & & A' \ar@{=}[dd] \\
A \ar@{=}[dd] \haru{S} & & A \haru{F} & & A' \ar@{=}[dd] & = & & & & \eta' \\
& & \alpha & & & & A \hard{F} & & A' \hard{S'} & & A'. \\
A \hard{F} & & A' \hard{S'} & & A' }$$

We can now define the double category $\mathbf{Mnd(\mathbb{D})}$ with objects the monads, vertical arrows the monad morphisms, horizontal arrows the horizontal monad maps and cells as in $\mathbf{End(\mathbb{D})}$.
\end{definition}

\begin{definition}
Dually to the above, we define a \textbf{comonad} to be an endomorphism $(X, P: \srarrow{X}{X})$, equipped with globular cells $\epsilon: P \rightarrow U_X$ and $\delta: P \rightarrow P \odot P$, called counit and comultiplication respectively, that satisfy the suitable conditions. A \textbf{comonad morphism} from a comonad $(X,P)$ to a comonad $(Y,R)$ is a vertical morphism $(g,\psi):(X,P) \rightarrow (Y,R)$ which is compatible with the counit and comultiplication. A \textbf{horizontal comonad map} from a comonad $(X,P)$ to a comonad $(X',P')$ is a horizontal arrow $F:\srarrow{X}{X'}$, together with a cell
$$\xymatrixcolsep{0.2cm} \xymatrixrowsep{0.2cm} \xymatrix{ X \ar@{=}[dd] \haru{F} & & X' \haru{P'} & & X' \ar@{=}[dd] \\
& & \alpha \\
X \hard{P} & & X \hard{F} & & X', }$$
compatible with the counit and the comultiplication too and lastly, we define the double category $\mathbf{Com(\mathbb{D})}$ with objects the comonads, vertical arrows the comonad morphisms, horizontal arrows the horizontal comonad maps and cells as in $\mathbf{End(\mathbb{D})}$.
\end{definition}

\begin{definition}
A horizontal endomorphism $P:\srarrow{A}{A}$ is called \textbf{co-pointed} if it is equipped with a globular cell $\epsilon:P \rightarrow U_A$ called counit. The double category $\mathbf{Copt(\mathbb{D})}$ is the sub-double category of $\mathbf{End(\mathbb{D})}$ with objects the co-pointed endomorphisms, and the vertical morphisms and horizontal maps that are compatible with the counit in the usual sense.
\end{definition}

\begin{lemma}
For a double category $\mathbb{D}$ we have $(\mathbf{Mnd(\mathbb{D}^{\mathrm{vop}})}) \simeq (\mathbf{Com(\mathbb{D})})^{\mathrm{vop}}$, where $\mathbb{D}^{\mathrm{vop}}$ is the double category where we invert the vertical arrows.
\end{lemma}
\begin{proof}
An object in $(\mathbf{Mnd(\mathbb{D}^{\mathrm{vop}})})$ is a pair $(A,S:\srarrow{A}{A})$ together with globular cells $\eta:U_A \rightarrow S$ and $\mu:S \odot S \rightarrow S$ in $\mathbb{D}^{\mathrm{vop}}$, i.e. globular cells $\eta:S \rightarrow U_A$ and $\mu: S \rightarrow S\odot S $ in $\mathbb{D}$, satisfying the suitable conditions. This is exactly a commonad structure in $\mathbb{D}$.
\end{proof}

\section{Eilenberg-Moore Objects} \label{Eilenberg-Moore Objects}

In the paper \cite{Fiore2011b}, the authors concentrated on the construction of free monads, and they treated the case for Eilenberg-Moore objects in the last two pages of the sequel \cite{Fiore2011c}. Here we will present their definition and we will prove that the double category of spans admits Eilenberg-Moore objects for comonads. However, we will consider at the end a much simpler definition and this is what we will use later in Chapter 5. This is because we noticed that in our case we didn't need the generality of their definition.

Consider the double functor $I:\mathbb{D} \rightarrow \mathbf{Com(\mathbb{D})}$ such that
$$\xymatrixcolsep{0.2cm} \xymatrixrowsep{0.2cm} \xymatrix{
X & \mapsto & (X,U_X), }$$
$$\xymatrixcolsep{0.2cm} \xymatrixrowsep{0.2cm} \xymatrix{
X \ar[dd]_f & & X \ar[dd]_f & & X \ar[dd]_f \haru{U_X} & & X \ar[dd]^f \\
& \mapsto & & , & & U_f \\
Y& & Y & & Y \hard{U_Y} & & Y, }$$
$$\xymatrixcolsep{0.2cm} \xymatrixrowsep{0.2cm} \xymatrix{
& & & & & & & & X \ar@{=}[dd] \haru{U_X} & & X \haru{M} & & X' \ar@{=}[dd] \\
X \haru{M} & & X' & \mapsto & X \haru{M} & & X'& , & & & \cong & & & \text{and} \\
& & & & & & & & X \hard{M} & & X' \hard{U_{X'}} & & X' }$$
$$\xymatrixcolsep{0.2cm} \xymatrixrowsep{0.2cm} \xymatrix{
X \ar[dd]_f \haru{M} & & X' \ar[dd]^g & & X \ar[dd]_f \haru{M} & & X' \ar[dd]^g \\
& \alpha & & \mapsto & & \alpha \\
Y \hard{N} & & Y' & & Y \hard{N} & & Y'. }$$ 

The authors in \cite{Fiore2011c} say that a double category $\mathbb{D}$ has Eilenberg-Moore objects for comonads if the double functor $I:\mathbb{D} \rightarrow \mathbf{Com(\mathbb{D})}$ has a right adjoint in $\mathbf{DblCat}$. Dually a double category $\mathbb{D}$ has Eilenberg-Moore objects for monads if the double functor $I:\mathbb{D} \rightarrow \mathbf{Mon(\mathbb{D})}$ has a right adjoint. Then the following proposition is a consequence of \Cref{adjoint_grandis_pare}.

\begin{proposition}\label{Eilenberg_Moore}
A double category $\mathbb{D}$ has Eilenberg-Moore objects for comonads if and only if the following hold:
\begin{enumerate}[label=\roman*.]
\item For each comonad $(X,P)$ there is an object $EM(X,P)$ and a universal comonad morphism $\theta_{(X,P)}:(EM(X,P),U) \rightarrow (X,P)$ from $I$ to $(X,P)$.
\item For each horizontal comonad map $(F,\alpha): \srarrow{(X,P)}{(X',P')}$ there is a horizontal arrow $EM(F):\srarrow{EM(X,P)}{EM(X'P')}$ and a universal cell
$$\xymatrixcolsep{0.2cm} \xymatrixrowsep{0.2cm} \xymatrix{ (EM(X,P) ,U)\haru{EM(F)} \ar[dd]_{\theta_{(X,P)}} & & (EM(X',P'),U) \ar[dd]^{\theta_{(X',P')}} \\
& \theta_F \\
(X,P) \hard{F} & & (X',P') }$$
in $\mathbf{Com(\mathbb{D})}$ from $I$ to $F$, with source and target given by the universal vertical arrows in \textit{i}.
\item The induced lax double functor $EM:\mathbf{Com(\mathbb{D})} \rightarrow \mathbb{D}$ which is right adjoint to $I$, is pseudo.
\end{enumerate}
\end{proposition}

\begin{proposition}\label{Eilenberg_Moore_Spans}
The double category $\mathrm{\mathbb{S}pan}\mathcal{E}$, for $\mathcal{E}$ a category with pullbacks and terminal object, has Eilenberg-Moore objects for comonads, in the sense of Fiore, Gambino, and Kock.
\end{proposition}
\begin{proof}
Consider a comonad in $\mathrm{\mathbb{S}pan}\mathcal{E}$. Since we have a counit for it, this is exactly a span of the form
$$\xymatrixcolsep{0.2cm} \xymatrixrowsep{0.2cm} \xymatrix{
& P \ar[dl]_p \ar[dr]^p \\
X & & X.}$$
Then the comonad morphism 
$$\xymatrixcolsep{0.2cm} \xymatrixrowsep{0.2cm} \xymatrix{
& & & P \ar[dl]_1 \ar[dr]^1 \ar[dddd]^1 \\
P \ar[dd]_p & & P \ar[dd]_p & & P \ar[dd]^p \\
& , \\
X & & X & & X \\
& & & P \ar[ul]^p \ar[ur]_p }$$
is universal from the inclusion $I$ to $(X,P)$, since any comonad morphism
$$\xymatrixcolsep{0.2cm} \xymatrixrowsep{0.2cm} \xymatrix{
& & & R \ar[dl]_1 \ar[dr]^1 \ar[dddd]^{\psi} \\
R \ar[dd]_r & & R \ar[dd]_r & & R \ar[dd]^r \\
& , \\
X & & X & & X \\
& & & P \ar[ul]^p \ar[ur]_p }$$
can be written uniquely as
$$\xymatrixcolsep{0.2cm} \xymatrixrowsep{0.2cm} \xymatrix{
& & & R \ar[dl]_1 \ar[dr]^1 \ar[dd]^{\psi} \\
R \ar[dd]_{\psi} & & R \ar[dd]_{\psi} & & R \ar[dd]^{\psi} \\
& & & P \ar[dddd]^1 \ar[dl]_1 \ar[dr]^1 \\
P \ar[dd]_p & , & P \ar[dd]_p & & P \ar[dd]^p \\
\\
X & & X & & X \\
& & & P \ar[ul]^p \ar[ur]_p }$$
Consider now a horizontal comonad map $F:\srarrow{(X,P)}{(X,P)}$, which is a span
$$\xymatrixcolsep{0.2cm} \xymatrixrowsep{0.2cm} \xymatrix{
& F \ar[dl]_{f_1} \ar[dr]^{f_2} \\
X & & X' }$$
together with a cell
$$\xymatrixcolsep{0.2cm} \xymatrixrowsep{0.2cm} \xymatrix{
& & P\times_X F \ar[dl]_-{\pi_1} \ar[dr]^-{\pi_2} \\
& P \ar[dl]_p \ar[dr]^p & & F \ar[dl]_{f_1} \ar[dr]^{f_2} \\
X \ar@{=}[dd] & & X & & X' \ar@{=}[dd] \\
& & \alpha \\
X & & X' & & X' \\
& F \ar[ul]^{f_1} \ar[ur]_{f_2} & & P' \ar[ul]^{p'} \ar[ur]_{p'} \\
& & F \times_{X'} P' \ar[ul]^-{\pi_3} \ar[ur]_-{\pi_4} }$$
i.e. with an arrow $\alpha: P\times_X F \rightarrow F \times_X' P'$ such that $f_1 \pi_3 \alpha =p \pi_1$ and $p' \pi_4 \alpha = f_2 \pi_2$. Then the span
$$\xymatrixcolsep{0.2cm} \xymatrixrowsep{0.2cm} \xymatrix{
& P\times_X F \ar[dl]_-{\pi_1} \ar[dr]^-{\pi_4 \alpha} \\
P & & P' }$$
is a horizontal comonad map $\srarrow{(P,U_p)}{(P',U_{P'})}$, with structure cell the identity. We will show that the cell
$$\xymatrixcolsep{0.2cm} \xymatrixrowsep{0.2cm} \xymatrix{
& & P\times_X F \ar[dddd]_{\pi_2} \ar[dll]_-{\pi_1} \ar[drr]^-{\pi_4 \alpha}\\
P \ar[dd]_p & & & & P' \ar[dd]^{p'} \\
\\
X & & & & X' \\
& & F \ar[ull]^{f_1} \ar[urr]_{f_2} }$$
is universal from $I$ to $F$. Consider a cell
$$\xymatrixcolsep{0.2cm} \xymatrixrowsep{0.2cm} 
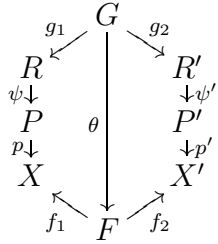
\xymatrix{
& G \ar[dl]_{g_1} \ar[dr]^{g_2} \ar[dddd]_{\theta} \\
R \ar[d]_-{\psi} & & R' \ar[d]^-{\psi'} \\
P \ar[d]_-p & & P' \ar[d]^-{p'} \\
X & \quad \quad & X' \\
& F \ar[ul]^{f_1} \ar[ur]_{f_2} }$$
and the unique arrow $\psi g_1 \times_X \theta$, taken from the universal property of the following pullback:
$$\xymatrixcolsep{0.4cm} \xymatrixrowsep{0.8cm} 
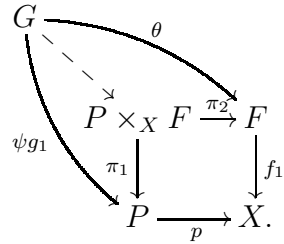
\xymatrix{ G \ar@{-->}[dr] \ar@/^2ex/[drr]^{\theta} \ar@/_2ex/[ddr]_{\psi g_1} \\
& P\times_X F \ar[d]_-{\pi_1} \ar[r]^-{\pi_2} & F \ar[d]^{f_1} \\
& P \ar[r]_p & X.}$$
Then we have
$$p' \pi_4 \alpha (\psi g_1 \times_X \theta) = f_2 \pi_2 (\psi g_1 \times_X \theta) = f_2 \theta = p' \psi' g_2$$
and by the universal property of $p'$ we get
$$\pi_4 \alpha (\psi g_1 \times_X \theta)=\psi' g_2.$$
So we can factor $\theta$ as follows:
$$\xymatrixcolsep{0.2cm} \xymatrixrowsep{0.2cm} 
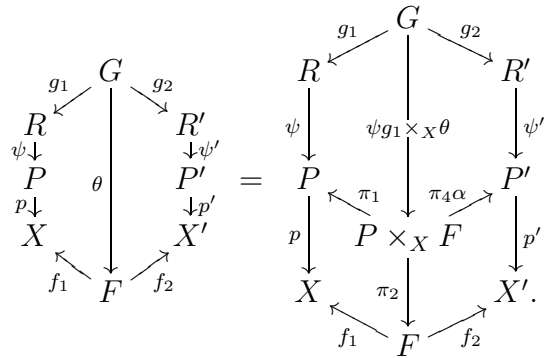
\xymatrix{
& & & & & G \ar[dl]_{g_1} \ar[dr]^{g_2} \ar[dddd]|{\psi g_1 \times_X \theta} \\
& G \ar[dl]_{g_1} \ar[dr]^{g_2} \ar[dddd]_{\theta} & & & R \ar[dd]_{\psi} & & R' \ar[dd]^{\psi'} \\
R \ar[d]_-{\psi} & & R' \ar[d]^-{\psi'} \\
P \ar[d]_-p & & P' \ar[d]^-{p'} & = & P \ar[dd]_p & & P' \ar[dd]^{p'} \\
X & \quad \quad & X' & & & P \times_X F \ar[ul]_{\pi_1} \ar[ur]^{\pi_4 \alpha} \ar[dd]_{\pi_2} \\
& F \ar[ul]^{f_1} \ar[ur]_{f_2} & & & X & & X'. \\
& & & & & F \ar[ul]^{f_1} \ar[ur]_{f_2} }$$
We have proved conditions \textit{i} and \textit{ii} of \Cref{Eilenberg_Moore}. It remains to show that the induced lax double functor $EM:\mathbf{Com}(\mathbb{S}\mathrm{pan}(\mathcal{E}))\rightarrow \mathbb{S}\mathrm{pan}(\mathcal{E})$ defined by
$$\xymatrixcolsep{0.4cm} \xymatrixrowsep{0.02cm} \xymatrix{ X & P \ar[l]_-p \ar[r]^-p & X & {\mapsto} & P}$$
$$\xymatrixcolsep{0.02cm} \xymatrixrowsep{0.02cm} \xymatrix{ & \quad F \quad \ar[ddl]_{f_1} \ar[ddr]^{f_2} & & & & P\times_X F \ar[ddl]_{\pi_1} \ar[rdd]^{\pi_4 \alpha}\\
& & & {\mapsto} \\
X & & X' & & P & & P' }$$
is pseudo. It is normal because the unit component
$$\xymatrixcolsep{0.2cm} \xymatrixrowsep{0.2cm} \xymatrix{
& P \ar[dl]_1 \ar[dr]^1 \ar[ddd]^{1\times_X 1} \\
P \ar@{=}[d] & & P \ar@{=}[d] \\
P & & P \\
& P\times_X P \ar[ul]^-{\pi} \ar[ur]_-{\pi} }$$
is invertible.
Consider now two composable horizontal comonad maps
$$\xymatrixcolsep{0.2cm} \xymatrixrowsep{0.2cm} \xymatrix{
& & & & & & P\times_X F \ar[dl]_-{\pi_1} \ar[dr]^-{\pi_2} \\
& & & & & P \ar[dl]_p \ar[dr]^p & & F \ar[dl]_{f_1} \ar[dr]^{f_2} \\
& F \ar[ddl]_{f_1} \ar[ddr]^{f_2} & & & X \ar@{=}[dd] & & X & & X' \ar@{=}[dd] \\
& & & & & & \alpha \\
X & & X' & , & X & & X' & & X' \\
& & & & & F \ar[ul]^{f_1} \ar[ur]_{f_2} & & P' \ar[ul]^{p'} \ar[ur]_{p'} \\
& & & & & & F \times_{X'} P' \ar[ul]^-{\pi_3} \ar[ur]_-{\pi_4} }$$
and
$$\xymatrixcolsep{0.2cm} \xymatrixrowsep{0.2cm} \xymatrix{
& & & & & & P'\times_{X'} G \ar[dl]_-{\rho_1} \ar[dr]^-{\rho_2} \\
& & & & & P' \ar[dl]_{p'} \ar[dr]^{p'} & & G \ar[dl]_{g_1} \ar[dr]^{g_2} \\
& G \ar[ddl]_{g_1} \ar[ddr]^{g_2} & & & X' \ar@{=}[dd] & & X' & & X'' \ar@{=}[dd] \\
& & & & & & \beta \\
X' & & X'' & , & X' & & X'' & & X'.' \\
& & & & & G \ar[ul]^{g_1} \ar[ur]_{g_2} & & P'' \ar[ul]^{p''} \ar[ur]_{p''} \\
& & & & & & G \times_{X''} P'' \ar[ul]^-{\rho_3} \ar[ur]_-{\rho_4}}$$
Then their images through $EM$ will be
$$\xymatrixcolsep{0.02cm} \xymatrixrowsep{0.02cm} \xymatrix{ & P\times_X F \ar[ddl]_-{\pi_1} \ar[rdd]^-{\pi_4 \alpha} & & & & P'\times_{X'} G \ar[ddl]_-{\rho_1} \ar[rdd]^-{\rho_4 \beta}\\
\\
P & & P' & \text{ and } & P' & & P''}$$
respectively. Consider also the pullbacks
$$\xymatrixcolsep{0.6cm} \xymatrix{
F\times_{X'} G \ar[d]_-{\tau_1} \ar[r]^-{\tau_2} & G \ar[d]^{g_1} & P \times_X (F\times_{X'} G) \ar[d] \ar[r] & F\times_{X'} G \ar[d]^{f_1 \tau_1} & (P \times_X F) \times_{X'} G \ar[d]_{\sigma_1} \ar[r]^-{\sigma_2} & G \ar[d]^{g_1} \\
F \ar[r]_{f_2} & X', & P \ar[r]_p & X, & P\times_X F \ar[r]_-{\pi_4 \alpha} & X', }$$
and
$$ \xymatrix{
(P\times_X F)\times_{P'} (P'\times_{X'} G) \ar[d]_{\upsilon_1} \ar[r]^-{\upsilon_2} & P' \times_{X'} G \ar[d]^{\rho_1} \\
P\times_X F \ar[r]_-{\pi_4 \alpha} & P'. }$$
The functor $EM$ is lax exactly because we have the unique arrow in the diagram below
$$\xymatrixcolsep{0.6cm} \xymatrix{
(P\times_X F)\times_{P'} (P'\times_{X'} G) \ar@{-->}[dr] \ar@/^2ex/[drr]^{(\pi_2 \upsilon_1)\times_{X'}(\rho_2 \upsilon_2)} \ar@/_2ex/[ddr]_{\pi_1 \upsilon_1 }\\
& P \times_X (F\times_{X'} G) \ar[d] \ar[r] & F\times_{X'} G \ar[d]^{f_1 \tau_1} \\
& P \ar[r]_p & X.}$$
It is furthermore pseudo because of the unique arrow in the following diagram
$$\xymatrixcolsep{0.6cm} \xymatrix{
(P \times_X F)\times_{X'} G \ar@{-->}[dr] \ar@/^2ex/[drr]^{(\pi_4 \alpha \sigma_1) \times_{X'} \sigma_2} \ar@/_2ex/[ddr]_{\sigma_1}\\
& (P\times_X F)\times_{P'} (P'\times_{X'} G) \ar[d]_{\upsilon_1} \ar[r]^-{\upsilon_2} & P'\times_{X'}G \ar[d]^{\rho_1} \\
& P\times_X F \ar[r]_-{\pi_4 \alpha} & P'}$$
and the fact that $P \times_X (F\times_{X'} G) \cong (P \times_X F)\times_{X'} G$.
\end{proof}

We will now discuss Eilenberg-Moore constructions for endomorphisms that are only equipped with a counit. If we were to follow Fiore, Gambino and Kock's definition then we would say that $\mathbf{D}$ admits Eilenberg-Moore objects for co-pointed endomorphisms if the inclusion double functor $I:\mathbf{D}\rightarrow \mathbf{Copt(\mathbb{D})}$ has a right adjoint. However, we will only use the one-dimensional condition of it:

\begin{definition}
A double category $\mathbf{D}$ \textbf{admits Eilenberg-Moore objects for co-pointed endomorphisms} if the inclusion functor between categories $I:D_0\rightarrow (\mathbf{Copt(\mathbb{D})})_0$ has a right adjoint.
\end{definition}

\begin{remark}
The category $(\mathbf{Copt(\mathbb{D})})_0$ in the above definition is just the category of co-pointed arrows and the morphisms between them in the category $D_0$.
\end{remark}

\begin{remark}\label{Eilenberg_definition}
As in the case of comonads, the above translates to a cell
$$\xymatrixcolsep{0.2cm} \xymatrixrowsep{0.2cm} \xymatrix{ EM(P) \haru{U} \ar[dd]_{u} & & EM(P) \ar[dd]^{u} \\ & \theta_P \\ A \hard{P}  & & A}$$
for every co-pointed endomorphism $P$, which is universal from the functor $U$ to the object $P$.
\end{remark}

\begin{example}
One can see from the proof of \Cref{Eilenberg_Moore_Spans} that the double category $\mathrm{\mathbb{S}pan}\mathcal{E}$, for $\mathcal{E}$ a category with pullbacks and terminal object, has Eilenberg-Moore objects for co-pointed endomorphisms.
\end{example}

\section{Fibrant Double Categories} \label{Fibrant Double Categories}

In a general double category there is no guarantee that there is some sort of relation between its two types of arrows. This is exactly what a fibrant double category provides. In such a double category we can always associate two horizontal arrows to every vertical arrow in such a way that the horizontal ones are parts of an adjunction in $\mathcal{H(\mathbb{D})}$. Classic references for the subject are Shulman's \cite{Shulman2008}, where he uses the name ``framed bicategories" instead, or the earlier \cite{AdjointsGP}, where Par\'e and Grandis use the term ``double categories with companions and conjoints" to define a structure equivalent to what we have here. In the appendices of \cite{Shulman2008} in fact, Shulman shows that to any fibrant double category we can associate a proarrow equipment, as in \cite{Woo82} or \cite{Woo85}, and to every proarrow equipment we can associate a fibrant double category.

 In the first half of this section we review the basic definitions and some important properties. In the second half, we consider fibrant double categories that satisfy specific product related properties.

\begin{definition}\cite{Shulman2008}
Consider a cell
$$\xymatrixcolsep{0.2cm} \xymatrixrowsep{0.2cm} \xymatrix{ A \haru{M} \ar[dd]_f & & C \ar[dd]^g \\ & \alpha \\ B \hard{N}  & & D. }$$ Then
\begin{enumerate}[label=\roman*.]
\item $\alpha$ is called \textbf{Cartesian} if every cell of the form
$$\xymatrixcolsep{0.2cm} \xymatrixrowsep{0.2cm} \xymatrix{ E \haru{M'} \ar[dd]_{fh} & & F \ar[dd]^{gk} \\ & \beta \\ B \hard{N}  & & D }$$
can be factored as
$$\xymatrixcolsep{0.2cm} \xymatrixrowsep{0.2cm} \xymatrix{ E \haru{M'} \ar[dd]_h & & F \ar[dd]^k \\
& \gamma \\
A \ar[dd]_f \haru{M} & & C \ar[dd]^g \\
& \alpha \\
B \hard{N} & & D,} $$
for a unique cell $\gamma$.
\item $\alpha$ is called \textbf{op-Cartesian} if every cell of the form
$$\xymatrixcolsep{0.2cm} \xymatrixrowsep{0.2cm} \xymatrix{ A \haru{M} \ar[dd]_{hf} & & C \ar[dd]^{kg} \\ & \beta \\ E \hard{N'}  & & F }$$
can be factored as
$$\xymatrixcolsep{0.2cm} \xymatrixrowsep{0.2cm} \xymatrix{ A \haru{M} \ar[dd]_f & & C \ar[dd]^g \\
& \alpha \\
B \ar[dd]_h \hard{N} & & D \ar[dd]^k \\
& \gamma \\
E \hard{N'} & & F,} $$
for a unique cell $\gamma$.
\end{enumerate}
\end{definition}

By the above definition it is clear that we have the following lemma:

\begin{lemma} \label{vertical_composite}
\begin{enumerate}
\item The vertical composite of two Cartesian cells is Cartesian.
\item The vertical composite of two op-Cartesian cells is op-Cartesian.
\end{enumerate}
\end{lemma}

\begin{proposition} \label{fibrant_equivalence}\cite{Shulman2008}
The following are equivalent for a double category $\mathbb{D}$:
\begin{enumerate}
\item For every vertical arrow $f:A\rightarrow B$, there exist horizontal arrows $f_*:\srarrow{A}{B}$ and $f^* : \srarrow{B}{A}$, together with cells
$$\xymatrixcolsep{0.2cm} \xymatrixrowsep{0.2cm} \xymatrix{
A \ar[dd]_f \haru{f_*} & & B \ar@{=}[dd] & & B \ar@{=}[dd] \haru{f^*} & & A \ar[dd]^f & & A \ar[dd]_f \haru{U_A} & & A \ar@{=}[dd] & & A \ar@{=}[dd] \haru{U_A} & & A \ar[dd]^f \\
& \quad & & , & & \quad & & , & & \quad & & \text{and} & & \quad \\
B \hard{U_B} & & B & & B \hard{U_B} & & B & & B \hard{f^*} & & A & & A \hard{f_*} & & B, } $$
such that
$$\xymatrixcolsep{0.2cm} \xymatrixrowsep{0.2cm} \xymatrix{
A \ar@{=}[dd] \haru{U_A} & & A \ar[dd]^f & & & & & & A \ar[dd]_f \haru{U_A} & & A \ar@{=}[dd] \\
& \quad & & &  A \ar[dd]_f \haru{U_A} & & A \ar[dd]^f & & & \quad \\
A \ar[dd]_f \hard{f_* } & & B \ar@{=}[dd] & = & & U_f & & = & B \ar@{=}[dd] \hard{f^*} & & A \: , \ar[dd]^f \\
& & & & B \hard{U_B} & & B \\
B \hard{U_B} & & B & & & & & & B \hard{U_B} & & B } $$
$$\xymatrixcolsep{0.2cm} \xymatrixrowsep{0.2cm} \xymatrix{
A \ar@{=}[dd] \ar[rrrr]^{f_*} \ar[rrrr]|{\vstretch{0.60}{|}} & & & & B \ar@{=}[dd] \\
& & \cong \\
A \ar@{=}[dd] \haru{U_A} & & A \ar[dd]^f \haru{f_* } & & B \ar@{=}[dd] & & A \ar@{=}[dd] \haru{f_*} & & B \ar@{=}[dd] \\
& \quad & & \quad & & = & & 1 \\
A \ar@{=}[dd] \hard{f_* } & & B \hard{U_B} & & B \ar@{=}[dd] & & A \hard{f_*} & & B \\
& & \cong \\
A \ar[rrrr]_{f_*} \ar[rrrr]|{\vstretch{0.60}{|}} & & & & B }$$
and
$$\xymatrixcolsep{0.2cm} \xymatrixrowsep{0.2cm} \xymatrix{
B \ar@{=}[dd] \ar[rrrr]^{f^*} \ar[rrrr]|{\vstretch{0.60}{|}} & & & & A \ar@{=}[dd] \\
& & \cong \\
B \ar@{=}[dd] \haru{U_A} & & A \ar[dd]^f \haru{f^* } & & A \ar@{=}[dd] & & B \ar@{=}[dd] \haru{f^*} & & A \ar@{=}[dd] \\
& \quad & & \quad & & = & & 1 \\
B \ar@{=}[dd] \hard{f^*} & & B \hard{U_B} & & A \ar@{=}[dd] & & B \hard{f^*} & & A \\
& & \cong \\
B \ar[rrrr]_{f^*} \ar[rrrr]|{\vstretch{0.60}{|}} & & & & A }$$
\item For every `niche' of the form
$$\xymatrixcolsep{0.2cm} \xymatrixrowsep{0.2cm} \xymatrix{ A \ar[dd]_f & & C \ar[dd]^g \\
\\ B \hard{M} & & D }$$
there is a horizontal arrow $g^*Mf_* : \srarrow{A}{C}$ and a Cartesian cell
$$\xymatrixcolsep{0.2cm} \xymatrixrowsep{0.2cm} \xymatrix{ A \ar[dd]_f \haru{g^*Mf_*} & & C \ar[dd]^g \\ 
\\ B \hard{M} & & D. }$$
\end{enumerate}
\end{proposition}
\begin{proof}
$\mathit{(1 \Rightarrow 2)}$ For a niche of the form
$$\xymatrixcolsep{0.2cm} \xymatrixrowsep{0.2cm} \xymatrix{ A \ar[dd]_f & & C \ar[dd]^g \\
\\ B \hard{M} & & D, }$$
consider the horizontal arrow $g^*Mf_*=(g^* \odot M) \odot f_*$, as in
$$\xymatrixcolsep{0.2cm} \xymatrixrowsep{0.2cm} \xymatrix{
A \ar[dd]_f \haru{f_* } & & B \ar@{=}[dd] \haru{M} & & D \ar@{=}[dd] \haru{g^*} & & C \ar[dd]^g \\
& & & 1 \\
B \ar@{=}[dd] \hard{U_B} & & B \hard{M} & & D \hard{U_D} & & D \ar@{=}[dd] \\
& & & \cong \\ 
B \ar[rrrrrr]|{\vstretch{0.60}{|}} \ar[rrrrrr]_M & & & & & & D .} $$
To show that this cell is Cartesian, consider a cell of the form
$$\xymatrixcolsep{0.2cm} \xymatrixrowsep{0.2cm} \xymatrix{ E \ar[dd]_{fh} \haru{M'} & & F \ar[dd]^{gk} \\
& \alpha \\
B \hard{M} & & D }$$
and note that we can factor it as follows :
$$\xymatrixcolsep{0.2cm} \xymatrixrowsep{0.2cm} \xymatrix{
& & & & E \ar@{=}[dd] \ar[rrrrrr]|{\vstretch{0.60}{|}} \ar[rrrrrr]^{M'} & & & & & & F \ar@{=}[dd] \\
& & & & & & & \cong \\
E \ar[dd]_{fh} \haru{M'} & & F \ar[dd]^{gk} & & E \ar[dd]_{fh} \haru{U_E} & & E \ar[dd]_{fh} \haru{M'} & & F \ar[dd]^{gk} \haru{U_F} & & F \ar[dd]^{gk} \\
& \alpha & & \overset{\text{nat. of }l \text{ and }r}{=} & & U_{fh} & & \alpha & & U_{gk} & & = \\
B \hard{M} & & D & & B \ar@{=}[dd] \hard{U_B} & & B \hard{M} & & D \hard{U_D} & & D \ar@{=}[dd] \\
& & & & & & & \cong \\
& & & & B \ar[rrrrrr]|{\vstretch{0.60}{|}} \ar[rrrrrr]_M & & & & & & D } $$
$$\xymatrixcolsep{0.2cm} \xymatrixrowsep{0.2cm} \xymatrix{
& & & & & & & & E \ar@{=}[dd] \ar[rrrrrr]|{\vstretch{0.60}{|}} \ar[rrrrrr]^{M'} & & & & & & F \ar@{=}[dd] \\
E \ar@{=}[dd] \ar[rrrrrr]|{\vstretch{0.60}{|}} \ar[rrrrrr]^{M'} & & & & & & F \ar@{=}[dd] & & & & & \cong \\
& & & \cong & & & & & E \ar[dd]_h \haru{U_E} & & E \ar[dd]^h \haru{M'} & & F \ar[dd]_k \haru{U_F} & & F \ar[dd]^k \\
E \ar[dd]_h \haru{U_E} & & E \ar[dd]^h \haru{M'} & & F \ar[dd]_k \haru{U_F} & & F \ar[dd]^k & & & U_h & & & & U_k \\
& U_h & & & & U_k & & & A \ar@{=}[dd] \haru{U_A} & & A \ar[dd]^f & \alpha & C \ar[dd]_g \haru{U_C} & & C \ar@{=}[dd] \\
A \ar[dd]_f \haru{U_A} & & A \ar[dd]^f & \alpha & C \ar[dd]_g \haru{U_C} & & C \ar[dd]^g & = \\
& U_f & & & & U_g & & & A \ar[dd]_f \haru{f_*} & & B \ar@{=}[dd] \hard{M} & & D \ar@{=}[dd] \haru{g^*} & & C \ar[dd]^g \\
B \ar@{=}[dd] \hard{U_B} & & B \hard{M} & & D \hard{U_D} & & D \ar@{=}[dd] & & & & & 1 \\
& & & \cong & & & & & B \ar@{=}[dd] \hard{U_B} & & B \hard{M} & & D \hard{U_D} & & D \ar@{=}[dd] \\
B \ar[rrrrrr]|{\vstretch{0.60}{|}} \ar[rrrrrr]_M & & & & & & D & & & & & \cong \\
& & & & & & & & B \ar[rrrrrr]|{\vstretch{0.60}{|}} \ar[rrrrrr]_M & & & & & & D & . } $$
For the uniqueness of the factorization, if there is another cell $\beta$ with
$$\xymatrixcolsep{0.2cm} \xymatrixrowsep{0.2cm} \xymatrix{
& & & & E \ar[dd]_h \ar[rrrrrr]|{\vstretch{0.60}{|}} \ar[rrrrrr]^{M'} & & & & & & F \ar[dd]^k \\
& & & & & & & \beta \\
E \ar[dd]_{fh} \haru{M'} & & F \ar[dd]^{gk} & & A \ar[dd]_f \haru{f_*} & & B \ar@{=}[dd] \haru{M} & & D \ar@{=}[dd] \haru{g^*} & & C \ar[dd]^g \\
& \alpha & & = & & & & 1 \\
B \hard{M} & & D & & B \ar@{=}[dd] \hard{U_B} & & B \hard{M} & & D \hard{U_D} & & D \ar@{=}[dd] \\
& & & & & & & \cong \\
& & & & B \ar[rrrrrr]|{\vstretch{0.60}{|}} \ar[rrrrrr]_M & & & & & & D &, } $$
then, by substituting the latter in
$$\xymatrixcolsep{0.2cm} \xymatrixrowsep{0.2cm} \xymatrix{
E \ar@{=}[dd] \ar[rrrrrr]|{\vstretch{0.60}{|}} \ar[rrrrrr]^{M'} & & & & & & F \ar@{=}[dd] \\
& & & \cong \\
E \ar[dd]_h \haru{U_E} & & E \ar[dd]^h \haru{M'} & & F \ar[dd]_k \haru{U_F} & & F \ar[dd]^k \\
& U_h & & & & U_k \\
A \ar@{=}[dd] \haru{U_A} & & A \ar[dd]^f & \alpha & C \ar[dd]_g \haru{U_C} & & C \ar@{=}[dd] \\
\\ A \hard{f_*} & & B \hard{M} & & D \hard{g^*} & & C &, } $$
we get that the last cell is the same as $\beta$. So the factorization is indeed unique.
\par
$\mathit{(2 \Rightarrow 1)}$ Let $f_* =U_Bf_*$ and $f^* =f^*U_B$. Clearly, by $\mathit{(2)}$ we can take the two Cartesian cells we need. Then, by using their universal property for the cell $U_f$, we get the third and fourth cell and the first double equation. For the next equation, note that
$$\xymatrixcolsep{0.2cm} \xymatrixrowsep{0.2cm} \xymatrix{
A \ar@{=}[dd] \ar[rrrr]|{\vstretch{0.60}{|}} \ar[rrrr]^{f_*} & & & & B \ar@{=}[dd] & & A \ar@{=}[dd] \ar[rrrr]|{\vstretch{0.60}{|}} \ar[rrrr]^{f_*} & & & & B \ar@{=}[dd] \\
& & \cong & & & & & & \cong \\
A \ar@{=}[dd] \haru{U_A} & & A \ar[dd]^f \haru{f_*} & & B \ar@{=}[dd] & & A \ar@{=}[dd] \haru{U_A} & & A \ar[dd]^f \haru{f_*} & & B \ar@{=}[dd] \\
& \quad & & \quad & & & & \quad & & \quad \\
A \ar@{=}[dd] \hard{f_*} & & B \hard{U_B} & & B \ar@{=}[dd] & \overset{\text{nat. of }l}{=} & A \ar[dd]_f \hard{f_*} & & B \ar@{=}[dd] \hard{U_B} & & B \ar@{=}[dd] & = \\
& & \cong \\
A \ar[dd]_f \ar[rrrr]|{\vstretch{0.60}{|}} \ar[rrrr]_{f_*} & & & & B \ar@{=}[dd] & & B \ar@{=}[dd] \hard{U_B} & & B \hard{U_B} & & B \ar@{=}[dd]  \\
& & & & & & & & \cong \\
B \ar[rrrr]|{\vstretch{0.60}{|}} \ar[rrrr]_{U_B} & & & & B & & B \ar[rrrr]|{\vstretch{0.60}{|}} \ar[rrrr]_{U_B} & & & & B } $$
$$\xymatrixcolsep{0.2cm} \xymatrixrowsep{0.2cm} \xymatrix{
A \ar@{=}[dd] \ar[rrrr]|{\vstretch{0.60}{|}} \ar[rrrr]^{f_*} & & & & B \ar@{=}[dd] & & A \ar@{=}[dd] \ar@{=}[dd] \ar[rrrr]|{\vstretch{0.60}{|}} \ar[rrrr]^{f_*} & & & & B \ar@{=}[dd] \\
& & \cong & & & & & & \cong & & & & A \ar@{=}[dd] \haru{f_*} & & B \ar@{=}[dd] \\
A \ar[dd]_f \haru{U_A} & & A \ar[dd]_f \haru{f_*} & & B \ar@{=}[dd] & & A \ar@{=}[dd] \haru{U_A} & & A \haru{f_*} & & B \ar@{=}[dd] & & & 1 \\
& U_f & & \quad & & \overset{\text{nat. of }r}{=} & & \quad & \cong & \quad & & = & A \ar[dd]_f \haru{f_*} & & B \ar@{=}[dd] \\
B \ar@{=}[dd] \hard{U_f} & & B \hard{U_B} & & B \ar@{=}[dd] & & A \ar[dd]_f \ar[rrrr]|{\vstretch{0.60}{|}} \ar[rrrr]_{f_*} & & & & B \ar@{=}[dd] & & & \quad \\
& & \cong & & & & & & & & & & B \hard{U_B} & & B & . \\
B \ar[rrrr]|{\vstretch{0.60}{|}} \ar[rrrr]_{U_B} & & & & B & & B \ar[rrrr]|{\vstretch{0.60}{|}} \ar[rrrr]_{U_B} & & & & B } $$
So, by the uniqueness of the factorization, the equation holds. We can show the last equation in a similar way.
\end{proof}

\begin{definition}\cite{Shulman2008}
A double category $\mathbb{D}$ is said to be \textbf{fibrant} if any of the two conditions in \ref{fibrant_equivalence} holds.
\end{definition}

The full sub-2-categories of $\mathbf{DblCat}$ and $\mathbf{DblCat}_{\mathcal{L}}$ determined by the fibrant double categories will be denoted by $\mathbf{FbrCat}$ and $\mathbf{FbrCat}_{\mathcal{L}}$, respectively.

\begin{proposition}\cite{Shulman2008}
The following are equivalent for a double category $\mathbb{D}$:
\begin{enumerate}
\item $\mathbb{D}$ is a fibrant double category.
 \item For every `niche' of the form
$$\xymatrixcolsep{0.2cm} \xymatrixrowsep{0.2cm} \xymatrix{ A \ar[dd]_f \haru{M} & & C \ar[dd]^g \\
\\ B & & D }$$
there is a horizontal arrow $g_*Mf^* : \srarrow{B}{D}$ and an op-Cartesian cell
$$\xymatrixcolsep{0.2cm} \xymatrixrowsep{0.2cm} \xymatrix{ A \ar[dd]_f \haru{M} & & C \ar[dd]^g \\ 
\\ B \hard{g_*Mf^*} & & D.}$$
\end{enumerate}
\end{proposition}
\begin{proof}
If $\mathbb{D}$ is a fibrant double category and we have a niche
$$\xymatrixcolsep{0.2cm} \xymatrixrowsep{0.2cm} \xymatrix{ A \ar[dd]_f \haru{M} & & C \ar[dd]^g \\
\\ B & & D, }$$
then the cell
$$\xymatrixcolsep{0.2cm} \xymatrixrowsep{0.2cm} \xymatrix{
A \ar@{=}[dd] \ar[rrrrrr]|{\vstretch{0.60}{|}} \ar[rrrrrr]^M & & & & & & C \ar@{=}[dd] \\
& & & \cong \\
A \haru{U_A} \ar[dd]_f & & A \ar@{=}[dd] \haru{M} & & C \ar@{=}[dd] \haru{U_C} & & C \ar[dd]^g \\
& & & 1 \\
B \hard{f^*} & & A \hard{M} & & C \hard{g_*} & & D }$$
is op-Cartesian. To prove this we use the dual argument of that in \Cref{fibrant_equivalence}. So we just define $g_*Mf^*$ to be exactly the horizontal arrow $(g_* \odot M) \odot f^*$. \par
On the other hand, if $\mathit{2}$ is true, we can prove \Cref{fibrant_equivalence}-$\mathit{1}$, dually to the latter's proof.
\end{proof}

\begin{definition}
We call the horizontal arrow $g^*Mf_*$ the \textbf{Cartesian filling} of the niche
$$\xymatrixcolsep{0.2cm} \xymatrixrowsep{0.2cm} \xymatrix{ A \ar[dd]_f & & C \ar[dd]^g \\
\\ B \hard{M} & & D, }$$
and the horizontal arrow $g_*Mf^*$ the \textbf{op-Cartesian filling} of the niche
$$\xymatrixcolsep{0.2cm} \xymatrixrowsep{0.2cm} \xymatrix{ A \ar[dd]_f \haru{M} & & C \ar[dd]^g \\
\\ B & & D. }$$
\end{definition}

\begin{example}
The double category $\mathbf{Set}$ is fibrant. Indeed, for every niche of the form
$$\xymatrixcolsep{0.2cm} \xymatrixrowsep{0.2cm} \xymatrix{
A \ar[dd]_f & & B \ar[dd]^g \\
\\ X \hard{S} & & Y, }$$
with $f$ and $g$ functions and $S$ a relation, define the relation $g^*Sf_*:\srarrow{A}{B}$ so that $a (g^*Sf_*) b \Leftrightarrow (fa)S(gb)$. Then clearly we have a cell
$$\xymatrixcolsep{0.2cm} \xymatrixrowsep{0.2cm} \xymatrix{
A \ar[dd]_f \haru{g^*Sf_*} & & B \ar[dd]^g \\
\\ X \hard{S} & & Y. }$$
Also, every cell of the form
$$\xymatrixcolsep{0.2cm} \xymatrixrowsep{0.2cm} \xymatrix{
U \ar[dd]_{fh} \haru{T} & & V \ar[dd]^{gk} \\
\\ X \hard{S} & & Y }$$
can be written as
$$\xymatrixcolsep{0.2cm} \xymatrixrowsep{0.2cm} \xymatrix{
U \ar[dd]_h \haru{T} & & V \ar[dd]^k \\
\\
A \ar[dd]_f \haru{g^*Sf_*} & & B \ar[dd]^g \\
\\ X \hard{S} & & Y, }$$
since $\forall (u,v) \in U\times V$, $uT v \Rightarrow (fhu)S(gkv) \Rightarrow (hu)(g^*Sf_*)(kv)$.
\end{example}

\begin{example}
The double category $\mathbb{S}\mathrm{pan}(\mathcal{E})$, for $\mathcal{E}$ a category with pullbacks, is also fibrant. In order to show that, consider a niche
$$\xymatrixcolsep{0.2cm} \xymatrixrowsep{0.2cm} \xymatrix{
A \ar[dd]_f & & & & B \ar[dd]^g \\
\\ X & & & & Y \\
& & S \ar[ull]^{s_0} \ar[urr]_{s_1} }$$
and the pullbacks
$$\xymatrixcolsep{0.4cm} \xymatrixrowsep{0.4cm} \xymatrix{
A \times_X S \ar[r]^-{p_1} \ar[d]_{p_0} & S \ar[d]^{s_0} & & S \times_Y B \ar[r]^-{r_1} \ar[d]_{r_0} & B \ar[d]^g \\
A \ar[r]_-f & X & , & S \ar[r]_-{s_1} & Y, }$$
and
$$\xymatrixcolsep{0.6cm} \xymatrixrowsep{0.6cm} \xymatrix{
(A\times_X S)\times_S (S \times_Y B) \ar[d]_{t_0} \ar[r]^-{t_1} & S\times_Y B \ar[d]^{r_0} \\
A \times_X S \ar[r]_-{p_1} & S. }$$
The Cartesian filling of the above niche will be the span
$$\xymatrixcolsep{0.6cm} \xymatrix{
A & (A\times_X S)\times_S (A \times_Y B) \ar[l]_-{p_0 t_0} \ar[r]^-{r_1 t_1} & B, }$$
together with the cell defined by the arrow $$p_1 t_0 = r_0 t_1 : (A\times_X S)\times_S (A \times_Y B) \rightarrow S,$$
since $s_0 p_1 t_0 = f p_0 t_0 $ and $ s_1 r_0 t_1 = g r_1 t_1$. The latter is a Cartesian cell, since, if we have a cell of the form
$$\xymatrixcolsep{0.2cm} \xymatrixrowsep{0.2cm} \xymatrix{
& & U \ar[dddddd]^{\alpha} \ar[dll]_{u_0} \ar[drr]^{u_1} \\
E \ar[dd]_h & & & & F \ar[dd]^k \\
\\ A \ar[dd]_f & & & & B \ar[dd]^g \\
\\ X & & & & Y \\
& & S \ar[ull]^{s_0} \ar[urr]_{s_1} }$$
it can be factored uniquely as $\alpha = p_1 t_0 [ (hu_0 \times_X \alpha) \times_S (\alpha \times_Y ku_1) ]$, through $p_1 t_0$.
\end{example}

\begin{lemma} \label{identity_filling}
For every object $A$ in a fibrant double category, ${1_A}_*\cong {1_A}^* \cong U_A$.
\end{lemma}
\begin{proof}
The horizontal arrow ${1_A}^*$ is defined to be the the Cartesian filling of the niche
$$\xymatrixcolsep{0.2cm} \xymatrixrowsep{0.2cm} \xymatrix{ A \ar@{=}[dd] & & A \ar@{=}[dd] \\
&\quad \\
A \hard{U_A} & & A} $$
and the following holds :
$$\xymatrixcolsep{0.2cm} \xymatrixrowsep{0.2cm} \xymatrix{
A \ar@{=}[dd] \haru{{1_A}^*} & & A \ar@{=}[dd] \\
& \quad & & & A \ar@{=}[dd] \haru{{1_A}^*} & & A \ar@{=}[dd] & & A \ar@{=}[dd] \haru{{1_A}^*} & & A \ar@{=}[dd] & & & & & & A \ar@{=}[dd] \haru{{1_A}^*} & & A \ar@{=}[dd] \\
A \ar@{=}[dd] \haru{U_A} & & A \ar@{=}[dd] & & & & & & & & & & A \ar@{=}[dd] \haru{1_A^* A} & & A \ar@{=}[dd] & & & 1 \\
& & & = & A \ar@{=}[dd] \haru{U_A} & & A \ar@{=}[dd] & = & A \ar@{=}[dd] \haru{U_A} & & A \ar@{=}[dd] & = & & & & = & A \ar@{=}[dd] \haru{{1_A}^*} & & A \ar@{=}[dd] \\
A \ar@{=}[dd] \haru{{1_A}^*} & & A \ar@{=}[dd] & & & U_{1_A} & & & & 1_{U_A} & & & A \hard{U_A} & & A \\
& & & & A \hard{U_A} & & A & & A \hard{U_A} & & A & & & & & & A \hard{U_A} & & A. \\
A \hard{U_A} & & A } $$
So by the universal property of ${1_A}_*$ we have
$$\xymatrixcolsep{0.2cm} \xymatrixrowsep{0.2cm} \xymatrix{
A \ar@{=}[dd] \haru{{1_A}^*} & & A \ar@{=}[dd] \\
& \quad & & & A \ar@{=}[dd] \haru{{1_A}^*} & & A \ar@{=}[dd] \\
A \ar@{=}[dd] \haru{U_A} & & A \ar@{=}[dd] & = \\
& & & & A \hard{{1_A}^*} & & A. \\
A \hard{{1_A}^*} & & A } $$
This, combined with the first identity of \Cref{fibrant_equivalence}, shows that ${1_A}^* \cong U_A$. Similarly, ${1_A}_* \cong U_A$.
\end{proof}

\begin{lemma}\cite{Shulman2008}
In a fibrant double category, for any niche of the form
$$\xymatrixcolsep{0.2cm} \xymatrixrowsep{0.2cm} \xymatrix{
A \ar[dd]_f & &  & & B \ar[dd]^g \\
\\ C \hard{N} & & D \hard{M} & & E, }$$
we have $g^* (M \odot N ) f_* \cong g^* \odot M \odot N \odot f_*$.
Also, for any
$$\xymatrixcolsep{0.2cm} \xymatrixrowsep{0.2cm} \xymatrix{
A \haru{N} \ar[dd]_f & & D \haru{M} & & B \ar[dd]^g \\
\\ C & & & & E, }$$
we have $g_* (M \odot N) f^* \cong g_* \odot M \odot N \odot f^*$.
\end{lemma}
\begin{proof}
Clear by the proof of \ref{fibrant_equivalence}.
\end{proof}

\begin{lemma}\label{f^*_functorial}\cite{Shulman2008}
For any vertical arrows $A \xrightarrow{f}B \xrightarrow{g}C$, we have $(gf)_* \cong g_* \odot f_*$ and $(gf)^* \cong f^* \odot g^*$.
\end{lemma}
\begin{proof}
The cell
$$\xymatrixcolsep{0.2cm} \xymatrixrowsep{0.2cm} \xymatrix{ 
 A \ar[dd]_f \haru{g_* \odot f_*} & & C \ar@{=}[dd] \\
\\
B \ar[dd]_g \haru{g_*} & & C \ar@{=}[dd] \\
\\
C \hard{U} & & C}$$
is Cartesian as the vertical compisite of two Cartesian cells. So $(gf)_* \cong g_* \odot f_*$.\\
Similarly the cell
$$\xymatrixcolsep{0.2cm} \xymatrixrowsep{0.2cm} \xymatrix{ 
 C \ar@{=}[dd] \haru{f^* \odot g^*} & & A \ar[dd]^f \\
\\
C \ar@{=}[dd] \haru{g^*} & & B \ar[dd]^g \\
\\
C \hard{U} & & C}$$
is Cartesian, so $(gf)^* \cong f^* \odot g^*$.
\end{proof}

\begin{proposition}\cite{Shulman2008}
In a fibrant double category, for every vertical arrow $f:A\rightarrow B$, we have $f_* \dashv f^*$ in $\mathcal{H(\mathbb{D})}$.
\end{proposition}
\begin{proof}
Consider the cells
$$\xymatrixcolsep{0.2cm} \xymatrixrowsep{0.2cm} \xymatrix{
& A \ar@{=}[dd] \ar[rrrr]^{U_A} \ar[rrrr]|{\vstretch{0.60}{|}} & & & & A \ar@{=}[dd] & & & B \ar@{=}[dd] \haru{f^*} & & A \ar[dd]^f \haru{f_*} & & B \ar@{=}[dd] \\
& & & \cong \\
\eta : & A \ar@{=}[dd] \haru{U_A} & & A \ar[dd]^f \haru{U_A} & & A \ar@{=}[dd] & \text{and} & \epsilon : & B \ar@{=}[dd] \hard{U_B} & & B \hard{U_B} & & B \ar@{=}[dd] \\
& & & & & & & & & & \cong \\
& A \hard{f_*} & & B \hard{f^*} & & A & & & B \ar[rrrr]_{U_B} \ar[rrrr]|{\vstretch{0.60}{|}} & & & & B. }$$
Then the triangle identity $(\epsilon \odot f_*) (f_* \odot \eta) = 1_{f_*}$ holds because of the following series of equalities:
$$\xymatrixcolsep{0.05cm} \xymatrixrowsep{0.2cm} \xymatrix{
A \ar@{=}[dd] \ar[rrrrrr]^{f^*} \ar[rrrrrr]|{\vstretch{0.60}{|}} & & & & & & B \ar@{=}[dd] \\
& & & \cong & & & & & A \ar@{=}[dd] \ar[rrrrrr]^{f_*} \ar[rrrrrr]|{\vstretch{0.60}{|}} & & & & & & B \ar@{=}[dd] \\
A \ar@{=}[dd] \ar[rrrr]^{U_A} \ar[rrrr]|{\vstretch{0.60}{|}} & & & &  A \ar@{=}[dd] \haru{f^*} & & B \ar@{=}[dddd] & & & & & \cong \\
& & \cong & & & & & & A \ar[rrrr]^{U_A} \ar[rrrr]|{\vstretch{0.60}{|}} \ar@{=}[dd] & & & & A \ar@{=}[dd] \haru{f_*} & & B \ar@{=}[dd] \\
A \haru{U_A} \ar@{=}[dd] & & A \ar[dd]_f \haru{U_A} & & A \ar@{=}[dd] & 1_{f_*} & & & & & \cong & & & 1_{f_*} \\
& & & \quad & & & & & A \haru{U_A} \ar@{=}[dd] & & A \ar[dd]_f \haru{U_A} & & A \ar[dd]^f \hard{f_*} & & B \ar@{=}[dd] \\
A \hard{f_*} \ar@{=}[dddd]&  & B \ar@{=}[dd] \hard{f^*} & & A \ar[dd]^f \hard{f_*} & & B \ar@{=}[dd] &\overset{\text{\Cref{fibrant_equivalence}}}{=} & & & & U_f & & & & \overset{\text{nat. of }\lambda}{=} \\
& & & & & & & & A \haru{f_*} \ar@{=}[dd] & & B \ar@{=}[dd] \hard{U_B} & & B \hard{U_B} & & B \ar@{=}[dd] \\
& 1_{f_*} & B \ar@{=}[dd] \hard{U_B} & & B \hard{U_B} & & B \ar@{=}[dd] & & & 1_{f_*} & & & \cong \\
& & & & \cong & & & & A \ar@{=}[dd] \hard{f_*} & & B \ar[rrrr]_{U_B} \ar[rrrr]|{\vstretch{0.60}{|}} & & & & B \ar@{=}[dd] \\
A \ar@{=}[dd] \hard{f_*} & & B \ar[rrrr]_{U_B} \ar[rrrr]|{\vstretch{0.60}{|}} & & & & B \ar@{=}[dd] & & & & & \cong \\
& & & \cong & & & & & A \ar[rrrrrr]_{f_*} \ar[rrrrrr]|{\vstretch{0.60}{|}} & & & & & & B \\
A \ar[rrrrrr]_{f_*} \ar[rrrrrr]|{\vstretch{0.60}{|}} & & & & & & B }$$
$$\xymatrixcolsep{0.05cm} \xymatrixrowsep{0.2cm} \xymatrix{
A \ar@{=}[dd] \ar[rrrrrr]^{f_*} \ar[rrrrrr]|{\vstretch{0.60}{|}} & & & & & & B \ar@{=}[dd] & & A \ar@{=}[dd] \ar[rrrrrr]^{f_*} \ar[rrrrrr]|{\vstretch{0.60}{|}} & & & & & & B \ar@{=}[dd] \\
& & & \cong & & & & & & & & \cong \\
A \ar[rrrr]^{U_A} \ar[rrrr]|{\vstretch{0.60}{|}} \ar@{=}[dd] & & & & A \ar[dd]^f \haru{f_*} & & B \ar@{=}[dd] & & A \ar[rrrr]^{U_A} \ar[rrrr]|{\vstretch{0.60}{|}} \ar@{=}[dd] & & & & A \ar[dd]^f \haru{f_*} & & B \ar@{=}[dd] \\
& & & \quad \\
A \ar@{=}[dd] \ar[rrrr]_{f_*} \ar[rrrr]|{\vstretch{0.60}{|}} & & & & B \ar@{=}[dd] \hard{U_B} & & B \ar@{=}[dd] & & A \ar@{=}[dd] \ar[rrrr]_{f_*} \ar[rrrr]|{\vstretch{0.60}{|}} & & & & B \ar@{=}[dd] \hard{U_B} & & B \ar@{=}[dd] \\
& & \cong & & & 1_{U_B} & & \overset{\text{unit id.}}{=} & & & \cong & & & 1_{U_B} & & = \\
A \haru{f_*} \ar@{=}[dd] & & B \ar@{=}[dd] \haru{U_B} & & B \haru{U_B} & & B \ar@{=}[dd] & & A \haru{f_*} \ar@{=}[dd] & & B \haru{U_B} & & B \ar@{=}[dd] \haru{U_B} & & B \ar@{=}[dd] \\
& 1_{f_*} & & & \cong & & & & & \quad & \cong & \quad & & 1_{U_B}  \\
A \ar@{=}[dd] \hard{f_*} & & B \ar[rrrr]_{U_B} \ar[rrrr]|{\vstretch{0.60}{|}} & & & & B \ar@{=}[dd] & & A \ar[rrrr]_{f_*} \ar[rrrr]|{\vstretch{0.60}{|}} \ar@{=}[dd] & & & & B \hard{U_B} & & B \ar@{=}[dd] \\
& & & \cong & & & & & & & & \cong \\
A \ar[rrrrrr]_{f_*} \ar[rrrrrr]|{\vstretch{0.60}{|}} & & & & & & B & & A \ar[rrrrrr]_{f_*} \ar[rrrrrr]|{\vstretch{0.60}{|}} & & & & & & B}$$
$$\xymatrixcolsep{0.2cm} \xymatrixrowsep{0.2cm} \xymatrix{
A \ar@{=}[dd] \ar[rrrr]^{f_*} \ar[rrrr]|{\vstretch{0.60}{|}} & & & & B \ar@{=}[dd] \\
& & \cong \\
A \ar@{=}[dd] \haru{U_A} & & A \ar[dd]^f \haru{f_*} & & B \ar@{=}[dd] & & A \ar@{=}[dd] \haru{f_*} & & B \ar@{=}[dd] \\
& \quad & & \quad & & \overset{\text{\Cref{fibrant_equivalence}}}{=} & & 1 \\
A \ar@{=}[dd] \hard{f_*} & & B \hard{U_B} & & B \ar@{=}[dd] & & A \hard{f_*} & & B & . \\
& & \cong \\
A \ar[rrrr]_{f_*} \ar[rrrr]|{\vstretch{0.60}{|}} & & & & B }$$
Similarly, we can also show that $(f^* \odot \epsilon) ( \eta \odot f^*) = 1_{f^*}$ is true.
\end{proof}

\vspace{0.3cm}
We will now consider fibrant double categories with some extra properties, regarding products in the vertical category and in the hom-categories of the horizontal bicategory. The following proposition is clearly inspired by the conditions that Carboni, Kelly, Walters, and Wood had in the definition of a precartesian bicategory.

\begin{proposition}\label{D1_products}
Consider a double category $\mathbb{D}$ such that:
\begin{enumerate}[label=\roman*.]
\item $\mathbb{D}$ is fibrant,
\item the vertical category $D_0$ has finite products and
\item the horizontal bicategory $\mathcal{H(\mathbb{D})}$ has finite products locally, i.e. every $\mathcal{H(\mathbb{D})(A,B)}$ has finite products.
\end{enumerate}
Then $D_1$ has finite products.
\end{proposition}
\begin{proof}
We use $\times, p_1, p_2$ for the product and the projections in $D_0$ and $\wedge, \pi_1, \pi_2$ for the local product and the projections in $\mathcal{H(\mathbb{D})}$. Also we use $I$ for the terminal in $D_0$ and $\top$ for the terminal in $\mathcal{H(\mathbb{D})}$. \par
Consider horizontal arrows $M:\srarrow{A}{B}$ and $N:\srarrow{X}{Y}$ and the Cartesian cells
$$\xymatrixcolsep{0.2cm} \xymatrixrowsep{0.2cm} \xymatrix{
A\times X \ar[dd]_{p_1} \haru{{p_1}^* M {p_1}_*} & & B\times Y \ar[dd]^{p_1} & & A\times X \ar[dd]_{p_2} \haru{{p_2}_* N {p_2}^*} & & B\times Y \ar[dd]^{p_2} \\
& \gamma_M & & \text{and} & & \gamma_N \\
A \hard{M} & & B & & X \hard{N} & & Y. } $$
Define $M \times N = ({p_1}^* M {p_1}_*) \wedge ({p_2}^* N {p_2}_* )$. We will show that this is the product in $D_1$ with projections 
$$\xymatrixcolsep{0.2cm} \xymatrixrowsep{0.2cm} \xymatrix{
A \times X \ar@{=}[dd] \haru{M\times N} & & B \times Y \ar@{=}[dd] & & A \times X \ar@{=}[dd] \haru{M\times N} & & B \times Y \ar@{=}[dd] \\
& \pi_1 & & & & \pi_2 \\
A\times X \ar[dd]_{p_1} \hard{{p_1}^* M {p_1}_* } & & B\times Y \ar[dd]^{p_1} & \text{and} & A \times X \ar[dd]_{p_2} \hard{{p_2}^* N{p_2}_*} & & B\times Y \ar[dd]^{p_2} \\
& \gamma_M & & & & \gamma_N \\
A \hard{M} & & B & & X \hard{N} & & Y. }$$
Consider cells
$$\xymatrixcolsep{0.2cm} \xymatrixrowsep{0.2cm} \xymatrix{
C \ar[dd]_f \haru{L} & & D \ar[dd]^g & & C \ar[dd]_h \haru{L} & & D \ar[dd]^k \\
& \alpha_M & & \text{and} & & \alpha_N \\
A \hard{M} & & B & & X \hard{N} & & Y, } $$
Then we have $f=p_1\langle f,h \rangle$, $g=p_1\langle g,k \rangle$, $h=p_2\langle f,h \rangle$, $k=p_2 \langle g,k \rangle$.
So there are unique cells
$$\xymatrixcolsep{0.2cm} \xymatrixrowsep{0.2cm} \xymatrix{
C \ar[dd]_{\langle f,h \rangle} \haru{L} & & D \ar[dd]^{\langle g,k \rangle} & & & & C \ar[dd]_{\langle f,h \rangle} \haru{L} & & D \ar[dd]^{\langle g,k \rangle} \\
& \beta_M & & & \text{ and } & & & \beta_N \\
A\times X \hard{{p_1}^* M {p_1}_* } & & B\times Y & & & & A\times X \hard{{p_2}^* N{p_2}_*} & & B\times Y, } $$
such that
\begin{figure}[!h]
\centering
\begin{tikzpicture}[scale=0.4]
\draw (0,0) node{$\alpha_M$} (1,0) node {=} (5,0) node {and} (8,0) node {$\alpha_N$} (9,0) node {=};
\draw	(2,1) node {$\beta_M$} (2,-1) node {$\gamma_M$} (10,1) node {$\beta_N$} (10,-1) node {$\gamma_N$};
\draw (2-0.4,0) -- (2+0.4,0) (10-0.4,0) -- (10+0.4,0);
\end{tikzpicture}
\end{figure}
\\
Consider also the op-Cartesian cell
$$\xymatrixcolsep{0.6cm} \xymatrixrowsep{0.2cm} \xymatrix{
C \ar[dd]_{\langle f,h \rangle} \haru{L} & & D \ar[dd]^{\langle g,k \rangle} \\
& \zeta \\
A \times X \hard{\langle g,k \rangle_* L \langle f,h \rangle^* } & & B \times Y. } $$
Then we can factor $\beta_M$ and $\beta_N$ as
\begin{figure}[!h]
\centering
\begin{tikzpicture}[scale=0.4]
\draw (0,0) node{$\beta_M$} (1,0) node {=} (5,0) node {and} (8,0) node {$\beta_N$} (9,0) node {=};
\draw	(2,1) node {$\zeta$} (2,-1) node {$\theta_M$} (10,1) node {$\zeta$} (10,-1) node {$\theta_N$};
\draw (2-0.4,0) -- (2+0.4,0) (10-0.4,0) -- (10+0.4,0);
\end{tikzpicture}
\end{figure}
\\
for unique cells
$$\xymatrixcolsep{0.4cm} \xymatrixrowsep{0.2cm} \xymatrix{
A\times X \ar@{=}[dd] \haru{\langle g,k \rangle_* L \langle f,h \rangle^*} & & B\times Y \ar@{=}[dd] & & & & A\times X \ar@{=}[dd] \haru{\langle g,k \rangle_* L \langle f,h \rangle^*} & & B\times Y \ar@{=}[dd] \\
& \theta_M & & & \text{ and } & & & \theta_N \\
A\times X \hard{{p_1}^* M {p_1}_* } & & B\times Y & & & & A\times X \hard{{p_2}^* N{p_2}_*} & & B\times Y. } $$
Now we can consider the cell $\langle \theta_M,\theta_N \rangle$, since they are both globular, and we can see that the diagram
$$\xymatrixcolsep{2cm} \xymatrixrowsep{1cm} \xymatrix{
& L \ar[dl]_{\alpha_M} \ar[dr]^{\alpha_N} \ar@{-->}[d]|-{\langle \theta_M, \theta_N \rangle \zeta} \\
M & M \times N \ar[l]^{\gamma_M \pi} \ar[r]_{\gamma_N \rho} & N }$$
commutes since
\begin{figure}[!h]
\centering
\begin{tikzpicture}[scale=0.4]
\draw (1,1) node{$\langle \theta_M , \theta_N \rangle$} (1,3) node{$\zeta$} (1,-1) node{$\pi_1$} (1,-3) node{$\gamma_M$} (3,0) node{=} (5,0) node{$\theta_M$} (5,2) node{$\zeta$} (5,-2) node{$\gamma_M$} (7,0) node{=} (9,1) node{$\beta_M$} (9,-1) node{$\gamma_M$} (11,0) node{=} (13,0) node{$\alpha_M$} (17,0) node{and} (21,1) node{$\langle \theta_M , \theta_N \rangle$} (21,3) node{$\zeta$} (21,-1) node{$\pi_2$} (21,-3) node{$\gamma_N$} (23,0) node{=} (25,0) node{$\theta_N$} (25,2) node{$\zeta$} (25,-2) node{$\gamma_N$} (27,0) node{=} (29,1) node{$\beta_N$} (29,-1) node{$\gamma_N$} (31,0) node{=} (33,0) node{$\alpha_N.$};
\draw (0,0)--(2,0) (8,0)--(10,0) (0,2)--(2,2) (0,-2)--(2,-2) (4,1)--(6,1) (4,-1)--(6,-1);
\draw (20,0)--(22,0) (28,0)--(30,0) (20,2)--(22,2) (20,-2)--(22,-2) (24,1)--(26,1) (24,-1)--(26,-1);
\end{tikzpicture}
\end{figure}
\\
Moreover, $\langle \theta_M, \theta_N \rangle \zeta$ is the unique cell with the above property. Indeed, if we have a cell
$$\xymatrixcolsep{0.2cm} \xymatrixrowsep{0.2cm} \xymatrix{
C \haru{L} \ar[dd]_{i} & & D \ar[dd]^{j} \\
& \lambda \\
A\times X \hard{M \times N} & & B\times Y}$$
with $\gamma_M \pi_1 \lambda = \alpha_M$ and $\gamma_N \pi_2 \lambda = \alpha_N$, then, by looking at the sources and the targets on these identites, we have $pi=f$, $ri=h$, $pj=g$ and $rj=k$, which implies that $i= \langle f,h \rangle$ and $j=\langle g,k \rangle$. So, since $\zeta$ is op-Cartesian, we can write $\lambda$ as
$$\xymatrixcolsep{0.6cm} \xymatrixrowsep{0.2cm} \xymatrix{
& & & & C \ar[dd]_{\langle f,h \rangle} \haru{L} & & D \ar[dd]^{\langle g,k \rangle} \\
C \ar[dd]_{\langle f,h \rangle} \haru{L} & & D \ar[dd]^{\langle g,k \rangle} & & & \zeta \\
& \lambda & & = & A \times X \ar@{=}[dd] \haru{\langle f,h \rangle ! L \langle g,k \rangle !} & & B \times Y \ar@{=}[dd] \\
A \times X \hard{M\times N} & & B \times Y & & & \kappa \\
& & & & A \times X \hard{M\times N} & & B \times Y, } $$
for a unique globular cell $\kappa$. It follows that
$$\begin{tikzpicture}[scale=0.4]
\draw (1,1) node{$\kappa$} (1,3) node{$\zeta$} (1,-1) node{$\pi_1$} (1,-3) node{$\gamma_M$} (3,0) node{=} (5,0) node{$\pi_1$} (5,2) node{$\lambda$} (5,-2) node{$\gamma_M$} (7,0) node{=} (9,0) node{$\alpha_M$} (11,0) node{=} (13,0) node{$\theta_M$} (13,-2) node{$\gamma_M$} (13,2) node{$\zeta$} (17,0) node{and} (21,1) node{$\kappa$} (21,3) node{$\zeta$} (21,-1) node{$\pi_2$} (21,-3) node{$\gamma_N$} (23,0) node{=} (25,0) node{$\pi_2$} (25,2) node{$\lambda$} (25,-2) node{$\gamma_N$} (27,0) node{=} (29,0) node{$\alpha_N$} (31,0) node{=} (33,-2) node{$\gamma_N$.} (33,0) node{$\theta_N$} (33,2) node{$\zeta$};
\draw (0,0)--(2,0) (0,2)--(2,2) (0,-2)--(2,-2) (4,1)--(6,1) (4,-1)--(6,-1) (12,-1)--(14,-1) (12,1)--(14,1);
\draw (20,0)--(22,0) (20,2)--(22,2) (20,-2)--(22,-2) (24,1)--(26,1) (24,-1)--(26,-1) (32,-1)--(34,-1) (32,1)--(34,1);
\end{tikzpicture}$$
So
$$\begin{tikzpicture}[scale=0.4]
\draw (0,0)--(2,0) (12,0)--(14,0);
\draw (1,1) node{$\kappa$} (1,-1) node{$\pi_1$} (3,0) node{=} (5,0) node{$\theta_M$} (9,0) node{and} (13,1) node{$\kappa$} (13,-1) node{$\pi_2$} (15,0) node{=} (17,0) node{$\theta_N$};
\end{tikzpicture}$$
by the universal property of $\gamma_M, \gamma_N, \zeta$. So $\kappa=\langle \theta_M, \theta_N \rangle$.
\par
\vspace{0.5cm}
For the terminal object in $D_1$, we will show that this is the terminal object $\top_{I,I}$ of $\mathcal{H(\mathbb{D})}(I,I)$. Indeed, consider a horizontal arrow $M: \srarrow{A}{B}$ and the unique vertical arrows $t_A:A \rightarrow I$ and $t_B: B \rightarrow I$, together with the op-Cartesian cell
$$\xymatrixcolsep{0.2cm} \xymatrixrowsep{0.2cm} \xymatrix{
A \haru{M} \ar[dd]_{t_A} & & B \ar[dd]^{t_B} \\
& \zeta \\
I \hard{{t_B}_* M t_A^*} & & I. }$$
Composing $\zeta$ with the unique cell to the terminal $\top_{I,I}$, 
$$\xymatrixcolsep{0.2cm} \xymatrixrowsep{0.2cm} \xymatrix{
I \haru{{t_B}_* M t_A^*} \ar@{=}[dd] & & I \ar@{=}[dd] \\
& \tau \\
I \hard{\top_{I,I}} & & I, }$$
we get a cell from $M$ to $\top_{I,I}$. Any cell of the form
$$\xymatrixcolsep{0.2cm} \xymatrixrowsep{0.2cm} \xymatrix{
A \haru{M} \ar[dd]_{t_A} & & B \ar[dd]^{t_B} \\
& \beta \\
I \hard{\top_{I,I}} & & I }$$
can be factored uniquely through $\zeta$ as
$$\xymatrixcolsep{0.2cm} \xymatrixrowsep{0.2cm} \xymatrix{
A \haru{M} \ar[dd]_{t_A} & & B \ar[dd]^{t_B} \\
& \zeta \\
I \ar@{=}[dd] \hard{{t_B}_* M t_A^*} & & I \ar@{=}[dd] \\
& \beta' \\
I \hard{\top{I,I}} & & I.
}$$
But now $\beta'$ is a cell in $\mathcal{H(\mathbb{D})}(I,I)$ to its terminal, so $\beta'=\tau$, which implies that $\tau \zeta$ is the unique cell from $M$ to $\top_{I,I}$.
\end{proof}

\begin{remark}
Note that from the proof above it follows that $$S \langle \alpha_M , \alpha_N \rangle = S (\langle \theta_M, \theta_N \rangle \zeta) = S \langle \theta_M, \theta_N \rangle S \zeta = S\zeta = \langle S \alpha_M , S\alpha_N \rangle $$ and similarly $$T (\langle \alpha_M , \alpha_N \rangle) = \langle T \alpha_M , T \alpha_N \rangle,$$
where $S$ and $T$ are the functors from $D_1$ to $D_0$ in the definition of a double category.
\end{remark}

\begin{lemma}\label{canonical_natural}
If $F:C\rightarrow D$ is a functor between categories and $C$ and $D$ have finite products, then the canonical arrow $\langle Fp_1, Fp_2 \rangle: F(A\times B) \rightarrow FA \times FB$, where $p_1$ and $p_2$ are the projections, is natural in both $A$ and $B$.
\end{lemma}
\begin{proof}
We need to show that for arrows $f:A\rightarrow C$ and $g:B\rightarrow D$, the following diagram
$$\xymatrix{
F(A\times B) \ar[d]_{F(f\times g)} \ar[rr]^{\langle F{p_1}, F{p_2} \rangle} & & FA \times FB \ar[d]^{Ff \times Fg} \\
F(C\times D) \ar[rr]_{\langle F{p_1}, F{p_2} \rangle} & & FC \times FD }$$
commutes. Composing with the projection $p_1$ of $FC\times FD$ we take:
$$p_1 \langle Fp_1, Fp_2 \rangle F(f\times g) = Fp_1 F(f\times g) = F(p_1 (f\times g)) = F(fp_1)$$
and $$p_1 (Ff \times Fg) \langle Fp_1 , Fp_2 \rangle = (Ff)p_1 \langle Fp_1 , Fp_2 \rangle = Ff Fp_1 = F(fp_1).$$
Similarly, $$p_2 \langle Fp_1, Fp_2 \rangle F(f\times g)=p_2 (Ff \times Fg) \langle Fp_1 , Fp_2 \rangle.$$
So the above diagram commutes.
\end{proof}

\begin{proposition}\label{lax_cartesian}
Consider a double category $\mathbb{D}$ such that:
\begin{enumerate}[label=\roman*.]
\item $\mathbb{D}$ is fibrant,
\item the vertical category $D_0$ has finite products and
\item the horizontal bicategory $\mathcal{H(\mathbb{D})}$ has finite products locally.
\end{enumerate}
Then the products on $D_0$ and the products on $D_1$ define lax double functors $\times: \mathbb{D}\times \mathbb{D} \rightarrow \mathbb{D}$ and $I: \mathbbm{1} \rightarrow \mathbb{D}$.
\end{proposition}
\begin{proof}
In \Cref{D1_products} we proved that, under the above conditions, $D_1$ has finite products so we can define a functor $\times:D_1 \times D_1 \rightarrow D_1$ with $M\times N = (p_1^*M{p_1}_*)\wedge (p_2^*N{p_2}_*)$. For cells
$$\xymatrixcolsep{0.2cm} \xymatrixrowsep{0.2cm} \xymatrix{
A \ar[dd]_f \haru{M} & & B \ar[dd]^g & & A' \ar[dd]_{f'} \haru{M'} & & B' \ar[dd]^{g'} \\
& \alpha & & \text{and} & & \alpha' \\
X \hard{N} & & Y & & X' \hard{N'} & & Y', } $$
their product $\alpha \times \alpha'$ is defined to be the unique cell that fits in the diagram
$$\xymatrixcolsep{0.6cm} \xymatrixrowsep{0.2cm} \xymatrix{
& & M\times M' \ar[dl]_-{\gamma_M \pi_1} \ar[dr]^-{\gamma_M' \pi_2} \ar@{-->}[dd] \\
& M \ar[dl]_{\alpha} & & M' \ar[dr]^{\alpha'} \\
N & & N\times N' \ar[ll]^{\gamma_N \pi_1} \ar[rr]_{\gamma_N' \pi_2} & & N'. } $$
Consider in addition the functor $\times :D_0 \times D_0 \rightarrow D_0$ defined by the product in $D_0$.\par
For horizontal arrows $M:\srarrow{A}{B}$ and $N:\srarrow{X}{Y}$, we have $S(M\times N)=A\times X=SM \times SN$ and $T(M\times N) =B \times Y = TM \times TN$. Also, for cells $\alpha$ and $\alpha'$ as above, we can see in \Cref{D1_products} that $S(\alpha \times \alpha') = \langle fp_1 , f'p_2 \rangle$ and $T(\alpha \times \alpha')= \langle gp_1 , g'p_2 \rangle$, which are respectively the arrows $f\times f'$ and $g \times g'$. I.e. $S(\alpha \times \alpha')=S(\alpha) \times S(\alpha')$ and $T(\alpha \times \alpha')=T(\alpha) \times T(\alpha')$.

Moreover, for horizontal arrows
$$\xymatrixcolsep{0.2cm} \xymatrixrowsep{0.2cm} \xymatrix{
A \haru{M} & & B \haru{M'} & & C \\
X \haru{N} & & Y \haru{N'} & & Z, }$$
the canonical cell $\langle \gamma_{M'} \pi_1 \odot \gamma_M \pi_1 , \gamma_{N'} \pi_2 \odot \gamma_N \pi_2 \rangle$ will provide us with a natural transformation 
$$\times_{\odot} :(M'\times N') \odot (M\times N) \rightarrow (M' \odot M )\times (N' \odot N),$$
and the cell $\langle U_{p_1} , U_{p_2}\rangle$ with a natural transformation $$\times_U : U_{A\times B} \rightarrow U_A \times U_B.$$
Note that these cells are globular since $$S \langle \gamma_{M'} \pi_1 \odot \gamma_M \pi_1 , \gamma_{N'} \pi_2 \odot \gamma_N \pi_2 \rangle = \langle S( \gamma_{M'} \pi_1 \odot \gamma_M \pi_1) ,S( \gamma_{N'} \pi_2 \odot \gamma_N \pi_2) \rangle =$$ $$ \langle S (\gamma_M \pi_1 ), S(\gamma_N \pi_2 ) \rangle = \langle p_1,p_2\rangle = 1,$$ $$S \langle U_{p_1} ,U_{p_2} \rangle = \langle SU_{p_1}, SU_{p_2} \rangle = \langle p_1,p_2\rangle = 1$$ and similarly for the targets.

We can similarly show that we have a lax double functor $I: \mathbbm{1} \rightarrow \mathbb{D}$ that maps the double category $\mathbbm{1}$ to the terminal object $I$, together with the vertical identity $1_I$, the horizontal arrow $\top_{I,I}$ and the idenity cell $1_{\top_{I,I}}$.
\end{proof}

In \Cref{adjoint_grandis_pare} we characterized an adjunction for general double categories $\mathbb{B}$ and $\mathbb{D}$. In the following we look at the case where $\mathbb{B}$ and $\mathbb{D}$ are fibrant.

\begin{theorem}\cite{Shulman2008} \label{fibrant_right_adjoint}
Consider fibrant double categories $\mathbb{D}$ and $\mathbb{E}$. A double functor $F:\mathbb{D} \rightarrow \mathbb{E}$ has right adjoint if and only if the following hold:
\begin{enumerate}[label=\roman*.]
\item For every object $A$ in $\mathbb{E}$, there is an object $GA$ in $\mathbb{D}$ and a universal vertical arrow $e_A: FGA \rightarrow A$ from $F$ to $A$.
\item For every horizontal arrow $M:\srarrow{A}{B}$ in $\mathbb{E}$, there is a horizontal arrow $GM:\srarrow{C}{D}$ in $\mathbb{D}$ and a unicersal cell $\epsilon_M: FGM \rightarrow M$ from $F$ to $M$.
\item If $\epsilon_M:FGM \rightarrow M$, $\epsilon_N: FGN \rightarrow N$ are universal cells, then so is the composite
$$\xymatrixcolsep{0.2cm} \xymatrixrowsep{0.2cm} \xymatrix{
\: \ar[rrrr]^{F(GN \odot GM)} \ar[rrrr]|{\vstretch{0.60}{|}} \ar@{=}[dd] & & & & \: \ar@{=}[dd] \\
& & \cong \\
\: \haru{FGM} \ar[dd] & & \: \haru{FGN} \ar[dd] & & \: \ar[dd] \\
& \epsilon_M & & \epsilon_N \\
\: \hard{M} & & \: \hard{N} & & \: } $$
\item If $e:FGA \rightarrow A$ is universal then so is the composite
$$\xymatrixcolsep{0.2cm} \xymatrixrowsep{0.2cm} \xymatrix{
FGA \ar@{=}[dd] \haru{U_{FGA}} & & FGA \ar@{=}[dd] \\
& \cong \\
FGA \ar[dd]_{e} \haru{FU_{GA}} & & FGA \ar[dd]^{e} \\
& U \\
A \hard{U_A} & & A }$$
\end{enumerate}
\end{theorem}

In his paper, Shulman proved the following proposition:

\begin{lemma}\cite{Shulman2008}\label{lax_preserves_Cartesian}
Any lax double functor between fibrant double categories preserves cartesian cells, and any oplax double functor preserves op-Cartesian cells.
\end{lemma}

Here we consider double categories not necessarily fibrant, and we prove that a lax double functor also preserves op-Cartesian cells if it happens to have a lax right adjoint. Dually we can show that an oplax double functor preserves Cartesian cells if it has a left oplax adjoint.

\begin{lemma}\label{right_adjoint_preserves_Cartesian}
For every adjunction $F \dashv G$ in $\mathbf{DblCat}_{\mathcal{L}}$, the lax double functor F preserves op-Cartesian cells and the lax double functor G preserves Cartesian cells.
\end{lemma}
\begin{proof}
Consider a Cartesian cell
$$\xymatrixcolsep{0.2cm} \xymatrixrowsep{0.2cm} \xymatrix{ A \haru{M} \ar[dd]_f & & B \ar[dd]^g \\ & \alpha \\ C \hard{N} & & D. }$$
and its image
$$\xymatrixcolsep{0.2cm} \xymatrixrowsep{0.2cm} \xymatrix{ FA \haru{FM} \ar[dd]_{Ff} & & FB \ar[dd]^{Fg} \\ & F\alpha \\ FC \hard{FN} & & FD.}$$
Consider also a cell of the form
$$\xymatrixcolsep{0.2cm} \xymatrixrowsep{0.2cm} \xymatrix{X \ar[dd]_h \haru{H} & & Y \ar[dd]^k \\
\\ FA \ar[dd]_{Ff} & \gamma & FB \ar[dd]^{Fg} \\
\\ FC \hard{FN} & & FD. }$$
We will show that the above cell can be factored uniquely through $F\alpha$. In fact, the cell
$$\xymatrixcolsep{0.2cm} \xymatrixrowsep{0.2cm} \xymatrix{
& GX \ar[dd]_{Gh} \haru{GH} & & GY \ar[dd]^{Gk} \\
\\
& GFA \ar[dd]_{GFf} \ar@/_1pc/[ddl]_-{\epsilon_A} & G \gamma & GFB \ar[dd]^{GFg} \ar@/^1pc/[ddr]^-{\epsilon_B} \\
\\
A \ar@/_1pc/[ddr]_-f & GFC \ar[dd]_{\epsilon_C} \haru{GFN} & & GFD \ar[dd]^{\epsilon_D} & B \ar@/^1pc/[ddl]^-g \\
& & \epsilon_N \\
& C \hard{N} & & D,}$$
where $\epsilon$ is the counit of the adjunction, can be factored uniquely through the Cartesian cell $\alpha$, i.e. there is a unique $\beta$ such that
$$\begin{tikzpicture}[scale=0.4]
\draw (0,0)--(2,0) (4,0)--(6,0);
\draw (1,1) node{$\beta$} (1,-1) node {$\alpha$} (5,1) node{$G\gamma$} (5,-1) node{$\epsilon_N$.};
\end{tikzpicture}$$
So now we have the cell
$$\xymatrixcolsep{0.2cm} \xymatrixrowsep{0.2cm} \xymatrix{
X \ar[dd]_{\eta_X} \haru{H} & & Y \ar[dd]^{\eta_Y} \\
& \eta_H \\
FGX \ar[dd]_{FGh} \hard{FGH} & & FGY \ar[dd]^{FGk}\\
\\
FGFA \ar[dd]_{F \epsilon_A} & F \beta & FGFB \ar[dd]^{F \epsilon_B}\\
\\
FA \hard{FM} & & FB,}$$
where $\eta$ is the unit of the adjunction and for which the following is true:
$$\begin{tikzpicture}[scale=0.4]
\draw (1,0) node{$F\beta$} (1,2) node{$\eta_H$} (1,-2) node{$F\alpha$} (0,1)--(2,1) (0,-1)--(2,-1);
\draw (3,0) node{=};
\draw (5,0) node{$FG\gamma$} (5,2) node{$\eta_H$} (5,-2) node{$F\epsilon$} (4,1)--(6,1) (4,-1)--(6,-1);
\draw (7,0) node{=};
\draw (9,0) node{$\eta F$} (9,2) node{$\gamma$} (9,-2) node{$F \epsilon$} (8,1)--(10,1) (8,-1)--(10,-1);
\draw (11,0) node{=};
\draw (12,0) node{$\gamma$.};
\end{tikzpicture}$$
The above factorization is unique since for any cell $\theta$ with
$$\begin{tikzpicture}[scale=0.4]
\draw (0,0)--(2,0) (1,1) node{$\theta$} (1,-1) node{$F\alpha$};
\draw (3,0) node{=};
\draw (4,1)--(6,1) (4,-1)--(6,-1) (5,0) node{$F\beta$} (5,2) node{$\eta_H$} (5,-2) node{$F\alpha$};
\draw (7,0) node{=};
\draw (8,0) node{$\gamma$,};
\end{tikzpicture}$$
we have
$$\begin{tikzpicture}[scale=0.4]
\draw (0,0)--(2,0) (1,1) node{$G\gamma$} (1,-1) node{$\epsilon_N$};
\draw (3,0) node{=};
\draw (4,1)--(6,1) (4,-1)--(6,-1) (5,0) node{$GF\alpha$} (5,2) node{$G\theta$} (5,-2) node{$\epsilon_N$};
\draw (7,0) node{=};
\draw (8,1)--(10,1) (8,-1)--(10,-1) (9,0) node{$\epsilon_M$} (9,2) node{$G\theta$} (9,-2) node{$\alpha$,};
\end{tikzpicture}$$
and by the uniqueness of the factorization of $\epsilon_N \: G\gamma$, we get
$$\begin{tikzpicture}[scale=0.4]
\draw (0,0) node{$\beta$} (1,0) node{=} (2,0)--(4,0) (3,1) node{$G\theta$} (3,-1) node{$\epsilon_M$};
\draw (5.5,0) node{, \enspace or};
\draw (7,0)--(9,0) (8,1) node{$\eta$} (8,-1) node{$F\beta$} (10,0) node{=} (11,0) node{$\theta$};
\end{tikzpicture}$$
by adjunction.
Similarly we can prove that the left adjoint $F$ preserves op-Cartesian cells.
\end{proof}

\chapter{Cartesian Double Categories} \label{Cartesian Double Categories}

\section{Precartesian Double Categories} \label{Precartesian Double Categories}

Following the approach of \cite{Carboni2008}, before defining Cartesian double categories, we have a short introductory section devoted to precartesian double categories. The central point in this section is the proposition below which shows that given simple product-related conditions, a fibrant double category is precartesian. 

\begin{definition}
A double category $\mathbb{D}$ is said to be \textbf{precartesian} if the diagonal double functor $\Delta: \mathbb{D} \rightarrow \mathbb{D} \times \mathbb{D}$ and the double functor $!: \mathbb{D} \rightarrow \mathbbm{1}$ have lax right adjoints $\times : \mathbb{D}\times \mathbb{D} \rightarrow \mathbb{D}$ and $I: \mathbbm{1} \rightarrow \mathbb{D}$, ie.e they are adjoint in $\mathbf{DblCat}_{\mathcal{L}}$.
\end{definition}

\begin{proposition}\label{fibrant_precartesian}
Consider a double category $\mathbb{D}$ such that:
\begin{enumerate}[label=\roman*.]
\item $\mathbb{D}$ is fibrant,
\item the vertical category $D_0$ has finite products and
\item the horizontal bicategory $\mathcal{H(\mathbb{D})}$ has finite products locally.
\end{enumerate}
Then $\mathbb{D}$ is a precartesian double category.
\end{proposition}
\begin{proof}
Consider the lax double functor $\times: \mathbb{D}\times \mathbb{D} \rightarrow \mathbb{D}$ of \Cref{lax_cartesian}. Since $\times$ is the product in $D_0$ we have the adjunction $\Delta \dashv \times : D_0 \times D_0 \rightarrow D_0$ with unit the natural transformation $d_A = \langle 1_A, 1_A \rangle : A \rightarrow A\times A$ and counit the natural transformation $(p, r)_{A,B} :(A\times B , A\times B) \rightarrow (A,B)$. We also have the adjunction $\Delta \dashv \times : D_1 \times D_1 \rightarrow D_1$ with unit the natural transformation $\delta_M=\langle 1_M, 1_M \rangle : M \rightarrow M \times M$ and counit the natural tranformation $(\gamma_M \pi_1, \gamma_N \pi_2)_{M,N}:(M\times N , M\times N ) \rightarrow (M,N)$. It suffices to show that these transformations combined give vertical natural trasformations $1_{\mathbb{D}} \rightarrow \times \Delta$ and $\Delta \times \rightarrow 1_{\mathbb{D} \times \mathbb{D}}$, since the triangle identites will follow from the above. \par
Observe that $S (\delta_M ) = S \langle 1_M , 1_M \rangle = \langle S 1_M , S 1_M \rangle = \langle 1_{SM}, 1_{SM} \rangle = d_{SM}$ and similarly, $T (\delta_M ) = d_{TM}$, for any horizontal arrow $M$. Also, for any $M$ and $N$, $S(\gamma_M \pi_1, \gamma_N \pi_2)_{M,N} = (p_1,p_2)_{SM,SN}$ and $T(\gamma_M \pi_1, \gamma_N \pi_2)_{M,N} = (p_1,p_2)_{TM,TN}$. It remains to show the identities
$$\xymatrixcolsep{0.2cm} \xymatrixrowsep{0.2cm} \xymatrix{ A \haru{N} \ar@{=}[dd] & & B \haru{M} & & C \ar@{=}[dd] & & A \ar[dd]_{d_A} \haru{N} & & B \ar[dd]_{d_B} \haru{M} & & C \ar[dd]^{d_C}\\
& & 1_{\odot} & & & & & \delta_N & & \delta_M \\
A \ar[dd]_{d_A}\ar[rrrr]|{\vstretch{0.60}{|}} \ar[rrrr]^{(M\odot N)} & & & & C \ar[dd]^{d_C} & = & A\times A \ar@{=}[dd] \hard{N\times N} & & B\times B \hard{M\times M} & & C\times C \ar@{=}[dd]\\
& & \delta_{M \odot N} & & & & & & \times_{\odot} \\
A\times A \ar[rrrr]|{\vstretch{0.60}{|}} \ar[rrrr]_{(M\odot N) \times (M\odot N)} & & & & C\times C & & A\times A \ar[rrrr]|{\vstretch{0.60}{|}} \ar[rrrr]_{(M\odot N) \times (M\odot N)} & & & & C\times C }$$
and
$$\xymatrixcolsep{0.2cm} \xymatrixrowsep{0.2cm} \xymatrix{A \haru{U_{A}} \ar@{=}[dd] & & A \ar@{=}[dd] & & A \ar[dd]_{d_A} \haru{U_{A}} & & A \ar[dd]^{d_A} \\
& 1 & & & & U_{d_A}\\
A \ar[dd]_{d_A}\haru{U_A} & & A \ar[dd]^{d_A} & = & A\times A \ar@{=}[dd] \hard{U_{A\times A}} & & A\times A \ar@{=}[dd]\\
& \delta_{U_A} & & & & \times_U\\
A\times A \hard{U_A \times U_A} & & A\times A & & A\times A \hard{U_A \times U_A} & & A\times A. }$$
However, by composing for instance both of the cells in the first one with the first projection $\gamma_{M\odot N} \pi_1 $ we get $$(\gamma_{M\odot N} \pi_1) \langle \gamma_M \pi_1 \odot \gamma_N \pi_1, \gamma_M \pi_2 \odot \gamma_N \pi_2 \rangle (\delta_M \odot \delta_N ) =$$
$$(\gamma_M \pi_1 \odot \gamma_N \pi_1 ) (\delta_M \odot \delta_N ) =$$
$$(\gamma_M \pi_1 \delta_M ) \odot (\gamma_N \pi_1 \delta_N ) = $$
$$ 1_M \odot 1_N = 1_{M \odot N} = (\gamma_{M\odot N } \pi_1) \delta_{M\odot N}.$$
Similarly $$(\gamma_{M\odot N} \pi_2) \langle \gamma_M \pi_1 \odot \gamma_N \pi_1, \gamma_M \pi_2 \odot \gamma_N \pi_2 \rangle (\delta_M \odot \delta_N ) = (\gamma_{M\odot N } \pi_2) \delta_{M\odot N}.$$
So the first identity is true. Similarly, we can show that the second identity is true as well, and this will complete the proof of the adjunction $\Delta \dashv \times : \mathbb{D} \times \mathbb{D} \rightarrow \mathbb{D}$.

We can also show that the terminal object $I$ of $D_0$ and the terminal object of $\mathcal{H(\mathbb{D})}(I,I)$ provide an adjunction $! \dashv I: \mathbbm{1} \rightarrow \mathbb{D}$.
\end{proof}

Not surprisingly, the double category $\mathrm{\mathbb{S}et}$ with objects sets, vertical arrows functions, and horizontal arrows relations between sets is precartesian. Moreover, the double category $\mathrm{\mathbb{S}pan}\mathcal{E}$ with objects and vertical arrows the objects and the arrows of $\mathcal{E}$, and horizontal arrows spans in $\mathcal{E}$ is precartesian as well. We use the above proposition to prove it.

\begin{proposition}\label{set_precartesian}
The double category $\mathrm{\mathbb{S}et}$ is precartesian.
\end{proposition}
\begin{proof}
We have already seen that the double category $\mathrm{\mathbb{S}et}$ is fibrant. We also know that its vertical category has finite products, in particular, the Cartesian product of sets and functions, together with the singleton set as the terminal. In addition, its horizontal bicategory $\mathcal{H(\mathrm{\mathbb{S}et})}$ has finite products locally. Indeed, if $R,S:\srarrow{A}{B}$ are relations, then define their product $R\wedge S:\srarrow{A}{B}$ such that $a (R\wedge S) b \Leftrightarrow (a Rb)\wedge (aSb)$ and the terminal relation $\top_{A,B}:\srarrow{A}{B}$ to be the whole $A\times B$, i.e. $\forall (a,b), a \top_{A,B} b$. So, by \Cref{fibrant_precartesian}, $\mathrm{\mathbb{S}et}$ is a precartesian double category.
\end{proof}
\begin{proposition}\label{Topco_precartesian}
The double category $\mathrm{\mathbb{S}pan}\mathcal{E}$, for $\mathcal{E}$ a category with pullbacks and terminal object, is precartesian.
\end{proposition}
\begin{proof}
Again, since $\mathrm{\mathbb{S}pan}\mathcal{E}$ is fibrant, we only need to show that its vertical category has finite products and its horizontal bicategory has finite products locally. The vertical category $\mathrm{\mathbb{S}pan}\mathcal{E}_0$ has finite products, since $\mathcal{E}$ has pullbacks and terminal object.
Also, the product of two spans, $R$ and $S$, in $\mathcal{H(\mathrm{\mathbb{S}pan}\mathcal{E})}(A,B)$, is given by the pullback
$$\xymatrix{
R\times_{A\times B} S \ar[d]_-{p_1} \ar[r]^-{p_2} & S \ar[d]^-{\langle s_o, s_1 \rangle } \\
R \ar[r]_-{\langle r_0, r_1 \rangle } & A\times B }$$
together with the composites $$p_1\langle r_0, r_1 \rangle p_1 = p_1 \langle s_o, s_1 \rangle p_2 : R\times_{A\times B} S \rightarrow A \text{ and}$$ $$p_2\langle r_0, r_1 \rangle p_1 = p_2 \langle s_o, s_1 \rangle p_2 : R\times_{A\times B} S  \rightarrow B,$$
and the terminal span $\top_{A,B}$ is the product $A\times B$ together with its projections.
\end{proof}

\section{Cartesian Double Categories} \label{Cartesian Double Categories}

\begin{definition}
A double category $\mathbb{D}$ is said to be \textbf{Cartesian} if the diagonal double functor $\Delta: \mathbb{D} \rightarrow \mathbb{D} \times \mathbb{D}$ and the double functor $!: \mathbb{D} \rightarrow \mathbbm{1}$ have right adjoints $\times : \mathbb{D}\times \mathbb{D} \rightarrow \mathbb{D}$ and $I: \mathbbm{1} \rightarrow \mathbb{D}$ in $\mathbf{DblCat}$.
\end{definition}

As we see now, because of the fact that $\mathbf{DblCat}$ is a strict 2-category and we are allowed to consider adjunctions in it, the definition of a Cartesian double category becomes simple. This was very briefly mentioned in Shulman's remark 2.11. of \cite{Shulman2010}, as a specific case of a Cartesian object in a 2-category with finite products. He also says that many of the examples of Cartesian bicategories in \cite{Carboni1987} and \cite{Carboni2008} can be seen as the horizontal bicategories of some Cartesian objects in the 2-category $\mathbf{DblCat}$. Of course, this is exactly what we want to develop in this section.

If now we have a Cartesian double category $\mathbb{D}$, then we have the universal vertical arrows $\eta_A: A \rightarrow A\times A$ and $\epsilon_{(A,B)}: (A\times B, A\times B) \rightarrow (A,B)$. We write the former as $d_A$ and the latter as the pair of the vertical arrows $p_1:A\times B \rightarrow A$ and $p_2:A\times B \rightarrow B$. We also have the universal cell
$$\xymatrixcolsep{0.2cm} \xymatrixrowsep{0.2cm} \xymatrix{ A \haru{M} \ar[dd]_{\eta_A} & & B \ar[dd]^{\eta_B} \\ & \eta_M \\ A\times A \hard{M\times M} & & B\times B }$$
which will be denoted by $\delta_M$ instead, and the universal cell
$$\xymatrixcolsep{0.2cm} \xymatrixrowsep{0.2cm} \xymatrix{ (A\times B, A\times B) \haru{(M\times N , M\times N)} \ar[dd]_{\epsilon_{(A,B)}} & & (C\times D, C\times D) \ar[dd]^{\epsilon_{(C,D)}} \\ & \epsilon_{(M,N)} \\ (A,B) \hard{(M,N)} & & (C,D) }$$
which will be written as the pair of the cells
$$\xymatrixcolsep{0.2cm} \xymatrixrowsep{0.2cm} \xymatrix{ A\times B \ar[dd]_{p_1} \haru{M\times N} & & C\times D \ar[dd]^{p_1} & & A\times B \ar[dd]_{p_2} \haru{M\times N} & & C\times D \ar[dd]^{p_2} \\
& \pi_1 & & \text{and} & & \pi_2 \\
A \hard{M} & & C & & B \hard{N} & & D. } $$
Note that the triangle identities of the adjunction imply that $p_1 d_A = 1_A=p_2 d_A$, $(p_1 \times p_2)d_{A\times B}=1$, $\pi_1 \delta_M = 1_M = \pi_2 \delta_M$ and $(\pi_1 \times \pi_1) \delta_{M\times N} =1$.\par
Moreover, the adunction $! \dashv I$ implies that there exists an object $I$ such that for each object $A$ there is a unique vertical arrow $t_A:A \rightarrow I$, and the horizontal arrow $U_I:I\rightarrow I$ is such that for every horizontal arrow $M$ there is a unique cell
$$\xymatrixcolsep{0.2cm} \xymatrixrowsep{0.2cm} \xymatrix{ A \ar[dd]_{t_A} \haru{M} & & B \ar[dd]^{t_B} \\
\\ I \hard{U_I} & & I. } $$
It is clear now that the following proposition holds:

\begin{proposition}\label{422}
If $\mathbb{D}$ is a Cartesian double category then it's vertical category $D_0$ has finite products $\times$ with projections $p_1$ and $p_2$, and $I$ the terminal object. Moreover, the category $D_1$ has finite products with projections $\pi_1$ and $\pi_2$, and $U_I$ the terminal object.
\end{proposition}

The following proposition ensures that for fibrant double categories we can also consider the converse of \Cref{422}:

\begin{proposition} \label{fibrant_cart}
Consider a double category $\mathbb{D}$ such that:
\begin{enumerate}[label=\roman*.]
\item $\mathbb{D}$ is fibrant,
\item the vertical category $D_0$ has finite products,
\item the horizontal bicategory $\mathcal{H(\mathbb{D})}$ has finite products locally and
\item the lax double functors $\times : \mathbb{D} \times \mathbb{D} \rightarrow \mathbb{D}$ and $I:\mathbbm{1} \rightarrow \mathbb{D}$ of \Cref{lax_cartesian} are pseudo.
\end{enumerate}
Then $\mathbb{D}$ is a Cartesian double category.
\end{proposition}
\begin{proof}
The proof is straightforward since condition \textit{iv.} shows that the adjunctions of the \Cref{fibrant_precartesian} take place in the 2-category $\mathbf{DblCat}$.
\end{proof}

\begin{remark}
Through the rest of the text, $\times$ on objects and horizontal arrows will often be abbreviated by juxtaposition. We avoid doing the same for vertical arrows and cells since in their case we use juxtaposition for the vertical composition.
\end{remark}

In the rest of this section we prove that some of our main examples are Cartesian double categories. Of course again, it is not surprising that the double category of sets $\mathrm{\mathbb{S}et}$ and the double category of spans $\mathrm{\mathbb{S}pan}\mathcal{E}$ is Cartesian. However, now we are also able to show that the same is true for the double category of $V$-matrices. The proof uses just the definition of a Cartesian double category, without the need of using the fact that $V\mbox{-}\mathrm{\mathbb{M}at}$ is fibrant.

\begin{proposition}
If $V=(V, \times, 1, p_1,p_2)$ is a Cartesian monoidal category, then $V\mbox{-}\mathrm{\mathbb{M}at}$ is a Cartesian double category.
\end{proposition}
\begin{proof}
Construct the double functor $\times : V\mbox{-}\mathrm{\mathbb{M}at} \times V\mbox{-}\mathrm{\mathbb{M}at} \rightarrow V\mbox{-}\mathrm{\mathbb{M}at}$ as follows:
\begin{enumerate}
\item If $S$ and $T$ are sets, then $S\times T$ is their Cartesian product and if $f$ and $g$ are functions, then $f\times g$ is their Cartesian product as well.
\item If $M:\srarrow{S}{T} $ and $N:\srarrow{X}{Y}$ are $V$-matrices, define $M\times N:  \srarrow{S\times X}{T\times Y}$ so that $$(M\times N) ((t,y),(s,x)) = M(t,s) \times N (y,x).$$ If $\phi : M\Rightarrow M' : \srarrow{S}{T}$ and $\psi : N \Rightarrow N' :  \srarrow{X}{Y} $ are cells, define $\phi \times \psi : M\times M' \Rightarrow N\times N' : \srarrow{S\times X}{T\times Y}$ to be the family with components $$(\phi \times \psi )((t,y),(s,x))= \phi (t,s) \times \psi (x,y).$$
\end{enumerate}
We will show that $\times$ is right adjoint to $\Delta$. Since $\times$ describes the Cartesian product for functions, we have the adjunction $\Delta_0 \dashv \times_0$ with unit $${\eta_0}_S: S \rightarrow S\times S, \: s\mapsto (s,s)$$ and counit $${\epsilon_0}_{(S,T)} : (S\times T, S\times T) \rightarrow (S,T) : ((s,t),(s',t') ) \mapsto (s,t').$$ Now, if $M:\srarrow{S}{T}$ is a matrix, set ${\eta_1}_M$ to be the cell
$$\xymatrixcolsep{0.2cm} \xymatrixrowsep{0.2cm} \xymatrix{ S \haru{M} \ar[dd]_{{\eta_0}_S} & & T \ar[dd]^{{\eta_0}_T} \\ & {\eta_1}_M \\ S\times S \hard{M\times M} & & T\times T }$$
with $${\eta_1}_M (t,s): M(t,s) \rightarrow M(t,s)\times M(t,s)$$ the diagonal arrow in $V$.
Also, for $M: \srarrow{S}{T}$ and $N:\srarrow{X}{Y}$, the cell
$$\xymatrixcolsep{0.2cm} \xymatrixrowsep{0.2cm} \xymatrix{ (S\times X, S\times X) \haru{(M\times N , M\times N)} \ar[dd]_{{\epsilon_0}_{(S,X)}} & & (T\times Y, T\times Y) \ar[dd]^{{\epsilon_0}_{(T,Y)}} \\ & {\epsilon_1}_{(M,N)} \\ (S,X) \hard{(M,N)} & & (T,Y) }$$
is the family of arrows $${\epsilon_1}_{(M,N)}(((t,y),(t',y')),((s,x),(s',x'))):$$ $$(M(t,s)\times N(y,x), M(t',s') \times N(y',x')) \rightarrow (M(t,s), N(y',x'))$$ given by the pairs of projections in $V$. The triangles identities for $\eta_1$ and $\epsilon_1$ follow by the triangle identities of the adjunction $\Delta \dashv \times :V\times V\rightarrow V$. \par
By construction, we see that $S{\eta_1}_M={\eta_0}_{SM}$, $T{\eta_1}_M={\eta_0}_{TM}$, $S{\epsilon_1}_{(M,N)}={\epsilon_0}_{(SM,SN)}$ and $T{\epsilon_1}_{(M,N)}={\epsilon_0}_{(TM,TN)}$. So in order to show that $\eta=\{\eta_0,\eta_1\}$ and $\epsilon=\{\epsilon_0,\epsilon_1\}$ are vertical natural transformations it suffices to show that the identities in \Cref{vertical_transf} hold. However, this is easy to verify since the composition of matrices is given by the coproduct in $V$ and the product distributes over it.\par
Lastly, consider the double functor $I: \mathbbm{1} \rightarrow V\mbox{-}\mathrm{\mathbb{M}at}$, which maps $\mathbbm{1}$ to the singleton $\{ \ast \}$, together with the identity function on it, the matrix $I(\ast, \ast)= 1$ and the cell $1_I(\ast,\ast)=1_1$. We know that $\{ \ast \}$ is the terminal object on sets and functions. Also the matrix $I(\ast, \ast)= 1$ is the terminal matrix from $\{ \ast \}$ to $\{ \ast \}$. Indeed, a matrix from $\{ \ast \}$ to $\{ \ast \}$ is just an object of $V$ and $1$ is the terminal object in $V$. The above proves that we also have an adjunction $! \dashv I : \mathbbm{1} \rightarrow V\mbox{-}\mathrm{\mathbb{M}at}$.
\end{proof}

\begin{proposition}
The double category $\mathrm{\mathbb{S}et}$ is Cartesian.
\end{proposition}
\begin{proof}
In \Cref{set_precartesian} we showed that $\mathrm{\mathbb{S}et}$ is precartesian. It remains to show that the lax double functors $\times : \mathrm{\mathbb{S}et} \times \mathrm{\mathbb{S}et} \rightarrow \mathrm{\mathbb{S}et}$ and $I:\mathbbm{1} \rightarrow \mathrm{\mathbb{S}et}$, given by the product on $\mathrm{\mathbb{S}et}_0$ and $\mathrm{\mathbb{S}et}_1$, are pseudo. First we describe these functors on the horizontal cells: If $R:\srarrow{A}{B}$ and $S:\srarrow{X}{Y}$ are relations, then $R\times S: A\times X \rightarrow B\times Y$ is given by
$$ \begin{array}{cr}
\underline{(a,x)(R\times S)(b,y)} \\
\underline{(a,x)[(p^*Rp^*)\wedge(r^*Sr^*)](b,y)}\\
\underline{[(a,x)(p^*Rp^*)(b,y)]\wedge[(a,x)(r^*Sr^*)(b,y)]}\\
aRb \wedge xSy
\end{array}$$  
and the terminal object $I:\srarrow{\{\ast\}}{\{\ast\}}$ is the set $\{(\ast,\ast)\}$. To show that $\times$ and $I$ are pseudo double functors it is to show that the canonical cells $$(R' \times S') \odot (R \times S ) \rightarrow (R' \odot R) \times (S' \odot S),$$ for $\xymatrixcolsep{0.2cm} \xymatrixrowsep{0.2cm} \xymatrix{ A \haru{R} & & B \haru{R'} & & C }$ and $\xymatrixcolsep{0.2cm} \xymatrixrowsep{0.2cm} \xymatrix{ X \haru{S} & & Y \haru{S'} & & Z }$, and $$U_{A\times B} \rightarrow U_A \times U_B$$ are invertible. Consider the following series of implications
$$\begin{array}{cr}
(a,x) [(R' \odot R) \times (S' \odot S)] (c,z) \Rightarrow \\
a (R' \odot R) c \wedge x (S' \odot S) z \Rightarrow \\
(\exists b)(aRb \wedge b R' c) \wedge (\exists y)(x S y \wedge y S' z) \Rightarrow \\
(\exists (b,y) )[(a,x) (R\times S )(b,y)\wedge (b,y)(R' \times S')(c,z)] \Rightarrow \\
(a,x) [(R' \times S') \odot (R \times S )] (c,z)
\end{array} $$
and
$$\begin{array}{cr}
(a,b) (U_A \times U_B ) (a',b') \Rightarrow \\
(a U_A a') \wedge (b U_B b') \Rightarrow \\
a=a' \wedge b=b' \Rightarrow \\
(a,b)=(a',b') \Rightarrow \\
(a,b) U_{A\times B} (a',b').
\end{array}$$
These will give the required inverses.
\end{proof}
\begin{proposition}
The double category $\mathrm{\mathbb{S}pan}\mathcal{E}$, for $\mathcal{E}$ a category with pullbacks and terminal object, is Cartesian.
\end{proposition}
\begin{proof}
The product in $\mathrm{\mathbb{S}pan}\mathcal{E}_1$, for spans $\xymatrixcolsep{0.4cm} \xymatrix{A & R \ar[l]_{r_0} \ar[r]^{r_1} & B}$ and $\xymatrixcolsep{0.4cm} \xymatrix{ X & S \ar[l]_{s_0} \ar[r]^{s_1} & Y }$
is given by the span $\xymatrix{A\times X & R\times S \ar[l]_{r_0\times s_0} \ar[r]^{r_1\times s_1} & B\times Y}$, together with the projections
$$\xymatrixcolsep{0.4cm} \xymatrixrowsep{0.4cm} \xymatrix{
& R\times S \ar[dl]_{r_0\times s_0} \ar[dr]^{r_1 \times s_1} \ar[ddd]^{p_{R,S}} & & & & R\times S \ar[dl]_{r_0\times s_0} \ar[dr]^{r_1 \times s_1} \ar[ddd]^{r_{R,S}} \\
A\times X \ar[d]_p & & B\times Y \ar[d]^p & & A\times X \ar[d]_r & & B\times Y \ar[d]^r \\
A & & B & \text{and} & X & & Y. \\
& R \ar[ul]^{r_0} \ar[ur]_{r_1} & & & & S \ar[ul]^{s_0} \ar[ur]_{s_1} }$$
Note that since the product in a category is unique up to isomorphism, the above product will be isomorphic to the product constructed by extending the local product on $\mathcal{H(\mathrm{\mathbb{S}pan}\mathcal{E})}$, i.e. the product used to construct the lax double functors in \Cref{lax_cartesian}. \par
In order to show that $\mathrm{\mathbb{S}pan}\mathcal{E}$ is Cartesian, we need to show that the product commutes with the horizontal composition of spans. This is true though, since the horizontal composition is given by pullbacks and the product commutes with pullbacks. Also, if we consider the identity spans $\xymatrixcolsep{0.4cm} \xymatrix{A & A \ar[l]_1 \ar[r]^1 & A}$ and $\xymatrixcolsep{0.4cm} \xymatrix{B & B \ar[l]_1 \ar[r]^1 & B}$, then their product is the span $\xymatrixcolsep{0.4cm} \xymatrix{A\times B & A \times B \ar[l]_{1 \times 1} \ar[r]^{1\times 1} & A\times B}$, which is exactly the identity span on $A\times B$. So $\mathrm{\mathbb{S}pan}\mathcal{E}$ is indeed a Cartesian double category.
\end{proof}

\section{Cartesian and Fibrant Double Categories} \label{Cartesian and Fibrant Double Categories}

In this last section of Chapter 4 we prove some properties of Cartesian and fibrant double categories that will be used as tools in the rest of the thesis. At the end of it we will define unit-pure double categories, which will be used later, towards our characterization theorem for the double category of spans.

\begin{lemma}\label{product_of_Cartesians}
In any Cartesian and fibrant double category, the product of two Cartesian cells is Cartesian and the product of two op-Cartesian cells is op-Cartesian.
\end{lemma}
\begin{proof}
It follows by \Cref{lax_preserves_Cartesian} since the product double functor is both lax and oplax.
\end{proof}

\begin{proposition}\label{horizontal_products}
If $\mathbb{D}$ is a Cartesian and fibrant double category then $\mathcal{H(\mathbb{D})}$ has finite products locally.
\end{proposition}
\begin{proof}
Consider two horizontal arrows $M$ and $N: \srarrow{A}{B}$. We will show that 
the Cartesian filling $M \wedge N = d_B^* (M\times N) {d_A}_*$ of the niche
$$\xymatrixcolsep{0.2cm} \xymatrixrowsep{0.2cm} \xymatrix{ A \ar[dd]_{d_A} & & B \ar[dd]^{d_B} \\
&\quad \\
A\times A \hard{M\times N} & & B\times B,} $$
is the product of $M$ and $N$, with projections
$$\xymatrixcolsep{0.2cm} \xymatrixrowsep{0.2cm} \xymatrix{
A \ar[dd]_{d_A} \haru{M \wedge N} & & B \ar[dd]^{d_B} & & A \ar[dd]_{d_A} \haru{M\wedge N} & & B \ar[dd]^{d_B} \\
& \gamma_{M,N} & & & & \gamma_{M,N} \\
A\times A \ar[dd]_{p_1} \hard{M\times N} & & B\times B \ar[dd]^{p_1} & \text{and} & A\times A \ar[dd]_{p_2} \hard{M\times N} & & B\times B \ar[dd]^{p_2} \\
& \pi_1 & & & & \pi_2 \\
A \hard{M} & & B & & A \hard{N} & & B, } $$
where $\gamma_{M,N}$ is Cartesian. Note that the cells above are globular since $p_1 d_A = 1_A=p_2d_A$ for each $A$, and so they are actually 2-cells in $\mathcal{H(\mathbb{D})}$.
Consider now cells of the form
$$\xymatrixcolsep{0.2cm} \xymatrixrowsep{0.2cm} \xymatrix{
A \ar@{=}[dd] \haru{L} & & B \ar@{=}[dd] & & A \ar@{=}[dd] \haru{L} & & B \ar@{=}[dd] \\
& \alpha & & \text{and} & & \beta \\
A \hard{M} & & B & & A \hard{N} & & B. } $$
Then we have the cell
$$\xymatrixcolsep{0.2cm} \xymatrixrowsep{0.2cm} \xymatrix{ A \ar[dd]_f \haru{L} & & B \ar[dd]^g \\
& \langle \alpha, \beta \rangle \\
A\times A \hard{M\times N} & & B\times B,} $$
with $\pi_1 \langle \alpha, \beta \rangle = \alpha$ and $\pi_2 \langle \alpha, \beta \rangle = \beta$. So $p_1 f = 1_A$ and $p_2 g = 1_B$. However, we can write the arrows $f$ and $g$ uniquely as $d_A f'$ and $d_B g'$ respectively, with $f' =1_A f' = p_1 d_A f' = p_1f = 1_A$ and similarly $g'=1_B$. Hence $f$ and $g$ are exactly the diagonal vertical arrows $d_A$ and $d_B$. Now, since $\gamma_{M,N}$ is Cartesian we can factor the above cell uniquely as
$$\xymatrixcolsep{0.2cm} \xymatrixrowsep{0.2cm} \xymatrix{
& & & & A \ar@{=}[dd] \haru{L} & & B \ar@{=}[dd] \\
A \ar[dd]_{d_A} \haru{L} & & B \ar[dd]^{d_B} & & & \langle \alpha , \beta \rangle_{\mathcal{H}} \\
& \langle \alpha, \beta \rangle & & = & A \ar[dd]_{d_A} \haru{M \wedge N} & & B \ar[dd]^{d_B} \\
A\times A \hard{M\times N} & & B\times B & & & \gamma_{M,N} \\
& & & & A \times A \hard{M\times N} & & B\times B. } $$
Moreover, $ \pi_1 \gamma_{M,N} \langle \alpha , \beta \rangle_{\mathcal{H}} = \pi_1 \langle \alpha , \beta \rangle = \alpha$ and $ \pi_2 \gamma_{M,N} \langle \alpha , \beta \rangle_{\mathcal{H}} = \pi_2 \langle \alpha , \beta \rangle = \beta $. \par
It remains to show that $\mathcal{H(\mathbb{D})(A,B)}$ has a terminal object. We will show that this is the Cartesian filling $\top=t_B^* U_I {t_A}_*$ of the niche
$$\xymatrixcolsep{0.2cm} \xymatrixrowsep{0.2cm} \xymatrix{ A \ar[dd]_{t_A} & & B \ar[dd]^{t_B} \\
& \quad \\
I \hard{U_I} & & I } $$
is the terminal object in $\mathcal{H(\mathbb{D})(A,B)}$. We have seen that for every horizontal arrow $M: \srarrow{A}{B}$ there is a unique cell
$$\xymatrixcolsep{0.2cm} \xymatrixrowsep{0.2cm} \xymatrix{ A \ar[dd]_{t_A} \haru{M} & & B \ar[dd]^{t_B} \\
& \tau \\
I \hard{U_I} & & I,} $$
which will be factored uniquely as
$$\xymatrixcolsep{0.2cm} \xymatrixrowsep{0.2cm} \xymatrix{ A \ar@{=}[dd] \haru{M} & & B \ar@{=}[dd] \\
& \tau' \\
A \ar[dd]_{t_A} \hard{\top} & & B \ar[dd]^{t_B} \\
\\ I \hard{U_I} & & I. } $$
Suppose that there is another globular cell $\alpha$ from $M$ to $\top$. Then by composing with
$$\xymatrixcolsep{0.2cm} \xymatrixrowsep{0.2cm} \xymatrix{
A \ar[dd]_{t_A} \haru{\top} & & B \ar[dd]^{t_B} \\
\\ I \hard{U_I} & & I } $$
we get a cell from $M$ to $U_I$, which has to be the same as $\tau$. Because of the uniqueness of the factorization of $\tau$, $\alpha$ must be the same as $\tau'$.
\end{proof}

\begin{corollary}
For a fibrant double category $\mathbb{D}$, the following are equivalent:
\begin{enumerate}
\item $\mathbb{D}$ is Cartesian.
\item $D_0$ has finite products, $\mathcal{H(\mathbb{D})}$ has finite products locally, and the induced lax double functors of \Cref{lax_cartesian} are pseudo.
\end{enumerate}
\end{corollary}
\begin{proof}
It follows by \Cref{horizontal_products} and \Cref{fibrant_cart}.
\end{proof}

\begin{remark}
A similar result for Cartesian bicategories was presented by Verity in the revised version of his PhD thesis \cite{Verity11}. In particular, he showed that Cartesian bicategories can be regarded as ``Cartesian objects'' in specific bicategory-enriched categories.
\end{remark}

In the following, we will often ignore the isomorphism $U\times U \cong U$ in the diagrams, in order to present our proofs in a simpler way. We will also ignore the isomorphisms $\alpha,\lambda,\rho$ that express the horizontal associativity and the horizontal unitary property of a double category.

\begin{lemma}\label{horizontal_naturality}
In a Cartesian and fibrant double category there are isomorphisms of the form:
$$\xymatrixcolsep{0.2cm} \xymatrixrowsep{0.2cm} \xymatrix{
AB \ar@{=}[dd] \haru{FB} & & XB \haru{{p_1}_*} & & X \ar@{=}[dd] & & BA \ar@{=}[dd] \haru{BF} & & AX \haru{{p_2}_*} & & X \ar@{=}[dd]\\
& & & & & ,\\
AB \hard{{p_1}_*} & & A \hard{F} & & X & & BA \hard{{p_2}_*} & & A \hard{F} & & X}$$
and
$$\xymatrixcolsep{0.2cm} \xymatrixrowsep{0.2cm} \xymatrix{
A \ar@{=}[dd] \haru{{p_1}^*} & & AB \haru{FB} & & XB \ar@{=}[dd] & & A \ar@{=}[dd] \haru{{p_2}^*} & & BA \haru{BF} & & BX \ar@{=}[dd]\\
& & & & & ,\\
A \hard{F} & & X \hard{{p_1}^*} & & XB & & A \hard{F} & & X \hard{{p_2}^*} & & BX.}$$
\end{lemma}
\begin{proof}
First observe that the cell
$$\xymatrixcolsep{0.2cm} \xymatrixrowsep{0.2cm} \xymatrix{
XB \ar@{=}[dddd] \haru{U} & & XB \ar[dddd]_{X\times t_B} \haru{U} & & XB \ar[dd]_{p_1} \haru{{p_1}_*} & & X \ar@{=}[dd] \\
\\
& & & U & X \haru{U} \ar[dd]_{\cong} & & X \ar[dd]^{\cong}\\
& & & & & U \\
XB \hard{X {t_B}_*} & & XI \hard{U} & & XI \hard{U} & & XI }$$
is invertible in $\mathbb{D}$, with inverse:
$$\xymatrixcolsep{0.2cm} \xymatrixrowsep{0.2cm} \xymatrix{
XB \ar@{=}[dddd] \haru{U} & & XB \ar[dddd]_{p_1} \haru{U} & & XB \ar[dd]_{X \times t_B} \haru{X{t_B}_*} & & XI \ar@{=}[dd]\\
\\
& & & U & XI \haru{U} \ar[dd]_{\cong} & & XI \ar[dd]^{\cong}\\
& & & & & U_{p_1} \\
XB \hard{{p_1}_*} & & X \hard{U} & & X \hard{U} & & X.}$$
Similarly, there is a cell
$$\xymatrixcolsep{0.2cm} \xymatrixrowsep{0.2cm} \xymatrix{
AB \haru{A {t_B}_*} \ar@{=}[dd] & & AI \ar[dd]^{\cong}\\
\\
AB \hard{{p_1}_*} & & A,}$$
which is invertible in $\mathbb{D}$. Now we have the isomorphism:
$$\xymatrixcolsep{0.2cm} \xymatrixrowsep{0.2cm} \xymatrix{
AB \ar@{=}[dd] \haru{FB} & & XB \ar@{=}[dd] \haru{{p_1}_*} & & X \ar[dd]^{\cong}\\
& 1 & & \cong\\
AB \ar@{=}[dd] \hard{FB} & & XB \hard{X {t_B}_*} & & XI \ar@{=}[dd] \\
& & \cong\\
AB \ar@{=}[dd] \haru{A {t_B}_*} & & AI \ar[dd]_{\cong} \haru{FI} & & XI \ar[dd]^{\cong}\\
& \cong & & \cong \\
AB \hard{{p_1}_*} & & A \hard{F} & & X,}$$
where the isomorphism in the middle follows by the functoriality of $\times$.\par
To prove the third isomorphism we can consider the diagram below instead:
$$\xymatrixcolsep{0.2cm} \xymatrixrowsep{0.2cm} \xymatrix{
A \ar[dd]_{\cong} \haru{{p_1}^*} & & AB \ar@{=}[dd] \haru{FB} & & XB \ar@{=}[dd] \\
& \cong & & 1 \\
AI \ar@{=}[dd] \hard{A {t_B}^*} & & AB \hard{FB} & & XB \ar@{=}[dd]\\
& & \cong\\
AI \ar[dd]_{\cong} \haru{FI} & & XI \ar[dd]^{\cong} \haru{X {t_B}^*} & & XB \ar@{=}[dd]\\
& \cong & & \cong \\
A \hard{F} & & X \hard{{p_1}^*} & & XB. }$$
The isomorphisms to the right follow in a similar way.
\end{proof}

\begin{remark}
A similar result as the one above was also proven for Cartesian bicategories in \cite{Carboni2008}, in proposition 4.7.
\end{remark}

\begin{definition}
A double category $\mathbb{D}$ is called \textbf{unit-pure} if the unit functor $U:D_0 \rightarrow D_1$ is full.
\end{definition}

\begin{remark}
In a unit-pure double category, whenever we have a cell of the form
$$\xymatrixcolsep{0.2cm} \xymatrixrowsep{0.2cm} \xymatrix{ A \haru{U_A} \ar[dd]_{f} & & A \ar[dd]^{g} \\ & \alpha \\ B \hard{U_B} & & B }$$
then $f=g$ and $\alpha=U_f$. Notice also that the functor $U$ is always faithful, so in the above case U is fully faithful.

\begin{example}
The double category $\mathbb{S}et$ is unit-pure since it is locally discrete.

The double category $\mathbb{S}pan\mathcal{E}$ is also unit-pure since a cell of the above form is as follows:
$$\xymatrixcolsep{0.2cm} \xymatrixrowsep{0.2cm} \xymatrix{
& & A \ar[dll]_{1} \ar[drr]^{1} \ar[dddd]^{f} \\
A \ar[dd]_f & & & & A \ar[dd]^f \\
\\ B & & & & B \\
& & B \ar[ull]^{1} \ar[urr]_{1}.}$$
\end{example}

\begin{example}
The double category of profunctors is not unit-pure. We can just consider the profunctors internal to sets for simplicity: Asking for this double category to be unit-pure would be the same as asking that for every two objects $A,A'$ in a category $\mathcal{C}$ and functors $F,G$ with domain $\mathcal{C}$, every function $\text{Hom}(A,A') \rightarrow \text{Hom}(GA,FA')$ is the same as the function that sends an arrow $f$ to $Ff$. One can also show that the double category of $V$-matrices is not unit-pure either.
\end{example}

For a  Cartesian double category in particular, the unit-pure condition implies that
$$\xymatrixcolsep{0.2cm} \xymatrixrowsep{0.2cm} \xymatrix{
AX \haru{U} \ar[dd]_{p_1} & & AX \ar[dd]^{p_1} & & AX \haru{U} \ar[dd]_{p_1} & & AX \ar[dd]^{p_1}\\
& \pi_1 & & = & & U_{p_1}\\
A \hard{U} & & A & & A \hard{U} & & A,}$$
$$\xymatrixcolsep{0.2cm} \xymatrixrowsep{0.2cm} \xymatrix{
AX \haru{U} \ar[dd]_{p_2} & & AX \ar[dd]^{p_2} & & AX \haru{U} \ar[dd]_{p_2} & & AX \ar[dd]^{p_2}\\
& \pi_2 & & = & & U_{p_2}\\
X \hard{U} & & X & & X \hard{U} & & X.}$$
and
$$\xymatrixcolsep{0.2cm} \xymatrixrowsep{0.2cm} \xymatrix{
A \haru{U} \ar[dd]_{d} & & A \ar[dd]^{d} & & A \haru{U} \ar[dd]_{d} & & A \ar[dd]^{d}\\
& \delta & & = & & U_{d}\\
AA \hard{U} & & AA & & AA \hard{U} & & AA.}$$
\end{remark}

\begin{lemma}
In a Cartesian, fibrant and unit-pure double category, for every horizontal arrow $F:\srarrow{A}{X}$, the cell
$$\xymatrixcolsep{0.2cm} \xymatrixrowsep{0.2cm} \xymatrix{ AAX \haru{AFX} \ar[dd]_{p_2} & & AXX \ar[dd]^{p_2} \\ & \pi_2 \\ A \hard{F} & & X}$$
is op-Cartesian.
\end{lemma}
\begin{proof}
It suffices to show that the cell
$$\xymatrixcolsep{0.2cm} \xymatrixrowsep{0.2cm} \xymatrix{
A \ar@{=}[dd] \haru{{p_2}^*} & & AAX \haru{AFX} \ar[dd]_{p_2} & & AXX \ar[dd]^{p_2} \haru{{p_2}_*} & & X \ar@{=}[dd] \\
& & & \pi_2 \\
A \hard{U} & & A \hard{F} & & X \hard{U} & & X}$$
is an isomorphism, with inverse $\beta$. Indeed, if this is the case and we consider a cell
$$\xymatrixcolsep{0.2cm} \xymatrixrowsep{0.2cm} \xymatrix{ AAX \haru{AFX} \ar[dd]_{p_2} & & AXX \ar[dd]^{p_2} \\
\\
A \ar[dd]_b & \theta & X \ar[dd]^c\\
\\
B \hard{M} & & C,}$$
then we can show that it factors uniquely through $\pi_2$ as follows:
$$\xymatrixcolsep{0.2cm} \xymatrixrowsep{0.2cm} \xymatrix{
AAX \ar@{=}[dd] \ar[rrrrrr]^{AFX} \ar[rrrrrr]|{\vstretch{0.60}{|}} & & & & & & AXX \ar@{=}[dd]\\
& & & \pi_2\\
A \ar@{=}[dd] \ar[rrrrrr]^{F} \ar[rrrrrr]|{\vstretch{0.60}{|}} & & & & & & X \ar@{=}[dd]\\
& & & \beta\\
A \ar@{=}[dd] \haru{{p_2}^*} & & AAX \haru{AFX} \ar[dd]_{p_2} & & AXX \ar[dd]^{p_2} \haru{{p_2}_*} & & X \ar@{=}[dd] \\
\\
A \ar[dd]_b \hard{U} & & A \ar[dd]_b & \theta & X \ar[dd]^c \hard{U} & & X \ar[dd]^c\\
& U_b & & & & U_c\\
B \hard{U} & & B \hard{M} & & C \hard{U} & & C.}$$
Now the following equality shows that it has a right inverse $\beta$:
$$\xymatrixcolsep{0.2cm} \xymatrixrowsep{0.2cm} \xymatrix{
A \ar[dd]_{d_3} \ar[rrrrrrrrrr]^{F} \ar[rrrrrrrrrr]|{\vstretch{0.60}{|}} & & & & & & & & & & X \ar[dd]_{d_3}\\
& & & & & \delta_3\\
AAA \ar@{=}[dd] \ar[rrrrrrrrrr]^{FFF} \ar[rrrrrrrrrr]|{\vstretch{0.60}{|}} & & & & & & & & & & XXX \ar@{=}[dd]\\
& & & & & \cong\\
AAA \ar[dd]_{p_2} \haru{AAF} & & AAX \ar[dd]_{p_2} \haru{U} & & AAX \ar@{=}[dd] \haru{AFX} & & AXX \ar@{=}[dd] \haru{U} & & AXX \ar[dd]^{p_2} \haru{FXX} & & XXX \ar[dd]^{p_2} & =\\
& \pi_2 & & & & 1 & & & & \pi_2\\
A \ar@{=}[dd] \haru{U} & & A \ar@{=}[dd] \haru{{p_2}^*} & & AAX \ar[dd]_{p_2} \haru{AFX} & & AXX \ar[dd]^{p_2} \haru{{p_2}_*} & & X \ar@{=}[dd] \haru{U} & & X\ar@{=}[dd] \\
& 1 & & & & \pi_2 & & & & 1\\
A \hard{U} & & A \hard{U} & & A \hard{F} & & X \hard{U} & & X \hard{U} & & X}$$

$$\xymatrixcolsep{0.14cm} \xymatrixrowsep{0.2cm} \xymatrix{
A \ar[dd]_{d_3} \ar[rrrrrrrrrr]^{F} \ar[rrrrrrrrrr]|{\vstretch{0.60}{|}} & & & & & & & & & & X \ar[dd]_{d_3}\\
& & & & & \delta_3\\
AAA \ar@{=}[dd] \ar[rrrrrrrrrr]^{FFF} \ar[rrrrrrrrrr]|{\vstretch{0.60}{|}} & & & & & & & & & & XXX \ar@{=}[dd] & & A \ar@{=}[dd] \haru{F} & & X \ar@{=}[dd]\\
& & & & & \cong & & & & & & = & & 1\\
AAA \ar[dd]_{p_2} \haru{AAF} & & AAX \ar[dd]_{p_2} \haru{U} & & AAX \ar[dd]_{p_2} \haru{AFX} & & AXX \ar[dd]^{p_2} \haru{U} & & AXX \ar[dd]^{p_2} \haru{FXX} & & XXX \ar[dd]^{p_2} & & A \hard{F} & & X\\
& \pi_2 & & U_{p_2} & & \pi_2 & & U_{p_2} & & \pi_2\\
A \hard{U} & & A \hard{U} & & A \hard{F} & & X \hard{U} & & X \hard{U} & & X.}$$
Then we also have
$$\xymatrixcolsep{0.00001cm} \xymatrixrowsep{0.2cm} \xymatrix{
AAX \ar[dd]_{p_2} \ar[rrrrrr]^{AFX} \ar[rrrrrr]|{\vstretch{0.60}{|}} & & & & & & AXX \ar[dd]^{p_2} & & AAX \ar@{=}[dd] \ar[rrrrrr]^{AFX} \ar[rrrrrr]|{\vstretch{0.60}{|}} & & & & & & AXX \ar@{=}[dd] \\
& & & \pi_2 & & & & & & & & A \beta X \\
A \ar@{=}[dd] \ar[rrrrrr]^{F} \ar[rrrrrr]|{\vstretch{0.60}{|}} & & & & & & X \ar@{=}[dd] & = & AAX \ar[dd]_{p_2} \haru{A{p_2}^*X} & & AAAXX \ar[dd]_{p_2} \haru{AAFXX} & & AAXXX \ar[dd]^{p_2} \haru{A{p_2}_*X} & & AXX \ar[dd]^{p_2} \\
& & & \beta & & & & & & \pi_2 & & \pi_2 & & \pi_2 \\
A \hard{{p_2}^*} & & AAX \hard{AFX} & & AXX \hard{{p_2}_*} & & X & & A \hard{{p_2}^*} & & AAX \hard{AFX} & & AXX \hard{{p_2}_*} & & X,}$$
by the naturality of $\pi_2$ and the functoriality of the product. The property $U_{p_2}=\pi_2$ of a unit-pure Cartesian double category is used to show that the second diagram is equal to:
$$\xymatrixcolsep{0.2cm} \xymatrixrowsep{0.2cm} \xymatrix{
AAX \ar[dd]_{p_2} \haru{U} & & AAX \ar@{=}[dd] \haru{AFX} & & AXX \ar@{=}[dd] \haru{U} & & AXX \ar[dd]^{p_2}\\
& & & 1\\
A \hard{{p_2}^*} & & AAX \hard{AFX} & & AXX \hard{{p_2}_*} & & X}$$
and finally
$$\xymatrixcolsep{0.2cm} \xymatrixrowsep{0.2cm} \xymatrix{
A \ar@{=}[dd] \haru{{p_2}^*} & & AAX \haru{AFX} \ar[dd]_{p_2} & & AXX \ar[dd]^{p_2} \haru{{p_2}_*} & & X \ar@{=}[dd] \\
& & & \pi_2 & & & & & A \ar@{=}[dd] \haru{{p_2}^*} & & AAX \haru{AFX} & & AXX \haru{{p_2}_*} & & X \ar@{=}[dd] \\
A \ar@{=}[dd] \hard{U} & & A \hard{F} & & X \hard{U} & & X \ar@{=}[dd] & = & & & & 1\\
& & & \beta & & & & & A \hard{{p_2}^*} & & AAX \hard{AFX} & & AXX \hard{{p_2}_*} & & X \\
A \hard{{p_2}^*} & & AAX \hard{AFX} & & AXX \hard{{p_2}_*} & & X,}$$
i.e. $\beta$ is also a left inverse for the given cell.
\end{proof}

\chapter{Spans} \label{Spans}

In this chapter we focus on the double category of spans and we show conditions under which a Cartesian and fibrant double category $\mathbb{D}$ is of the form $\mathbb{S}pan\mathcal{E}$ for some category $\mathcal{E}$ with pullbacks and terminal object.

\section{Tabulators} \label{Tabulators}

Tabulators for double categories were introduced in Grandis and Par\'e's paper \cite{LimitsGP} where they defined the tabulator of a given horizontal arrow $F$ as the double limit of the double diagram formed by $F$. However, in their later papers, they only considered the one-dimensional universal property of this definition. This is what we use here as well.

\begin{definition}\cite{Grandis2017}
We say that a double category has \textbf{tabulators} if for every horizontal arrow $F:\srarrow{A}{X}$ there is an object $T$ and a cell
$$\xymatrixcolsep{0.2cm} \xymatrixrowsep{0.2cm} \xymatrix{ T \haru{U} \ar[dd]_{q_1} & & T \ar[dd]^{q_2} \\ & \iota \\ A \hard{F} & & X}$$
which is universal in the sense that if there is another cell $\beta:U\rightarrow F$ then there is a unique vertical arrow $b:H\rightarrow T$ such that:
$$\xymatrixcolsep{0.2cm} \xymatrixrowsep{0.2cm} \xymatrix{
& & & & H \ar[dd]_b \haru{U} & & H \ar[dd]^b \\
H \haru{U} \ar[dd]_{h_2} & & H \ar[dd]^{h_1} & & & U \\
& \beta & & = & T \ar[dd]_{q_1} \haru{U} & & T \ar[dd]^{q_2} \\
A \hard{F} & & X & & & \iota \\
& & & & A \hard{F} & & X.}$$
Equivalently, the double category $\mathbb{D}$ has tabulators if the functor $U:D_0 \rightarrow D_1$ has right adjoint.\\
If the double category is fibrant, we say that the tabulators are \textbf{strong} if for each horizontal arrow $F:\srarrow{A}{X}$, the cell
$$\xymatrixcolsep{0.2cm} \xymatrixrowsep{0.2cm} \xymatrix{
A \ar@{=}[dd] \haru{{q_1}^*} & & T \haru{{q_2}_*} & & X \ar@{=}[dd] \\
& & \upsilon\\
A \ar[rrrr]_{F} \ar[rrrr]|{\vstretch{0.60}{|}} & & & & X}$$
we get from the universal property of the horizontal arrow ${q_2}_* \odot {q_1}^*$ and the cell $\iota$ are invertible.
\end{definition}

The main result of this section is that a Cartesian and fibrant double category with Eilenberg-Moore objects for co-pointed endomorphisms, as defined in \Cref{Eilenberg-Moore Objects}, has tabulators. In order to show that, first observe that the functor $U$ can be written as the composite of the two inclusions
$$D_0 \xrightarrow{I_0} \mathbf{Copt} (\mathbb{D})_0 \xrightarrow{K} D_1.$$
By definition, a double category has Eilenberg-Moore objects if the inclusion $\mathbb{D} \xrightarrow{I} \mathbf{Copt}(\mathbb{D})$ has a right adjoint, so the inclusion $I_0$ in the diagram above between the vertical categories has a right adjoint as well in this case. Below we build a functor $G: D_1 \rightarrow \mathbf{Copt} (\mathbb{D})_0$, right adjoint to $K$.

\begin{proposition}\label{functor_G}
For every Cartesian and fibrant double category there is a functor $G: D_1 \rightarrow \mathbf{Copt(\mathbb{D})}_0$ that maps a horizontal arrow $F:\srarrow{A}{X}$ to the Cartesian filling $G(F)$ as we see below
$$\xymatrixcolsep{0.4cm} \xymatrixrowsep{0.2cm} \xymatrix{A X \ar[dd]_{d_A \times X} \haru{G(F)} & & A X \ar[dd]^{A \times d_X}\\
& \gamma_F \\
A  A  X \hard{A  F  X} & & A  X  X.}$$
The endomorphism $G(F)$ is co-pointed with counit $\epsilon_{G(F)}$: 
$$\xymatrixcolsep{0.4cm} \xymatrixrowsep{0.2cm} \xymatrix{A X \ar[dd]_{d_A \times X} \ar@{=}@/_4pc/[dddd] \haru{G(F)} & & A X \ar[dd]^{A \times d_X} \ar@{=}@/^4pc/[dddd]\\
& \gamma_F \\
A  A  X \ar[dd]_{p_{1,3}} \hard{A F X} & & A X X \ar[dd]^{p_{1,3}}\\
& \pi_{1,3}\\
A X \hard{U_{A X}} & & A X.}$$
\end{proposition}
\begin{proof}
First we extend the definition to the arrows of $D_1$, i.e. the cells: Consider a cell
$$\xymatrixcolsep{0.4cm} \xymatrixrowsep{0.2cm} \xymatrix{A \ar[dd]_f  \haru{F} & & X \ar[dd]^g\\
& \alpha \\
A' \hard{F'} & & X'.}$$
The cell
$$\xymatrixcolsep{0.2cm} \xymatrixrowsep{0.2cm} \xymatrix{
& A X \ar@/_1pc/[ddl]_{f \times g} \ar[dd]_{d_A \times X} \haru{G(F)} & & A X \ar[dd]^{A \times d_X} \ar@/^1pc/[ddr]^{f\times g}\\
& & \gamma_F\\
\quad A'  X' = \ar@/_1pc/[ddr]_-{d_{A'}\times X'} & A  A  X \ar[dd]_{f\times f\times g} \haru{A F  X} & & A X  X \ar[dd]^{f \times g \times g} & = A'  X' \ar@/^1pc/[ddl]^-{A' \times d_{X'}} \quad \\
& & U_f \times \alpha \times U_g \\
& A'  A'  X' \hard{A'  F'  X'} & & A'  X'  X' }$$ 
can be factored uniquely through $\gamma_{F'}$ as follows:
$$\xymatrixcolsep{0.2cm} \xymatrixrowsep{0.2cm} \xymatrix{
A X \ar[dd]_{f \times g} \haru{G(F)} & & A X \ar[dd]^{f \times g}\\
\\
A'  X' \ar[dd]_{d_{A'} \times X'} \haru{G(F')} & & A' X' \ar[dd]^{A' \times d_{X'}}\\
& \gamma_{F'}\\
A'  A'  X' \hard{A'  F'  X'} & & A'  X'  X'.}$$
Define $G(\alpha)$ to be the unique cell of the above factorization.
Notice that by the universality of $\gamma_{F'}$ the following property is true, which we will often be using later on:
$$\begin{tikzpicture}[scale=0.4]
\draw (1,2) node{$\gamma_F$} (6,2) node{$G(\alpha$)} (1,0) node{$U_f \times \alpha \times U_g$} (6,0) node{$\gamma_{F'}.$};
\draw (0,1)--(2,1) (5,1)--(7,1) (4,1) node{=};
\end{tikzpicture}$$
Moreover, $G(\alpha)$ is compatible with the counits of $G(F)$ and $G(F')$ since:
$$\begin{tikzpicture}[scale=0.4]
\draw (1,4) node{$G(\alpha)$} (6,4) node{$\gamma_F$} (11,4) node{$\gamma_F$} (15,4) node{$\gamma_F$};
\draw (1,2) node{$\gamma_{F'}$} (6,2) node{$U_f \times \alpha \times U_g$} (11,2) node{$\pi_{1,3}$} (15,2) node{$\pi_{1,3}$};
\draw (1,0) node{$\pi_{1,3}$} (6,0) node{$\pi_{1,3}$} (11,0) node{$U_f\times U_g$} (15,0) node{$U_{f\times g}$};
\draw (0,3)--(2,3) (4.3,3)--(7.7,3) (10,3)--(12,3) (14,3)--(16,3);
\draw (0,1)--(2,1) (4.3,1)--(7.7,1) (10,1)--(12,1) (14,1)--(16,1);
\draw (2.5,2) node{=} (9.6,2) node{=} (13,2) node{=};
\end{tikzpicture}$$
We also have that $G(1_F)=1_{G(F)}$ by the universal property of $\gamma_F$ and the equality
$$\begin{tikzpicture}[scale=0.4]
\draw (1,0) node{$\gamma_F$} (5,0) node{$1$} (1,2) node{$G(1_F)$} (5,2) node{$\gamma_F$} (3,1) node{=} (7,1) node{=} (9,1) node{$\gamma_F$.};
\draw (0,1)--(2,1) (4,1)--(6,1);
\end{tikzpicture}$$
Similarly, $G(\alpha' \alpha)=G(\alpha') G(\alpha)$ for vertically composable cells
$$\xymatrixcolsep{0.2cm} \xymatrixrowsep{0.2cm}\xymatrix{
A \ar[dd]_{f} \haru{F} & & X \ar[dd]^{g}\\
& \alpha \\
A' \ar[dd]_{f'} \haru{F'} & & X' \ar[dd]^{g'}\\
& \alpha' \\
A'' \hard{F''} & & X'',}$$
by the universal property of $\gamma_{F''}$ and the equality
$$\begin{tikzpicture}[scale=0.4]
\draw (1,4) node{$G(\alpha)$} (7,4) node{$G(\alpha)$} (15,4) node{$\gamma_F$};
\draw (1,2) node{$G(\alpha')$} (7,2) node{$\gamma_{F'}$} (15,2) node{$U_f \times \alpha \times U_g$};
\draw (1,0) node{$\gamma_{F''}$} (7,0) node{$U_{f'}\times \alpha' \times U_{g'}$} (15,0) node{$U_{f'}\times \alpha' \times U_{g'}$};
\draw  (24,3) node{$\gamma_F$} (32,3) node{$G(\alpha' \alpha)$} (24,1) node{$U_{f'f} \times (\alpha'\alpha) \times U_{g'g}$} (32,1) node{$\gamma_{F''}.$};
\draw (0,3)--(2,3) (6,3)--(8,3) (14,3)--(16,3);
\draw (0,1)--(2,1) (6,1)--(8,1) (14,1)--(16,1);
\draw (22,2)--(26,2) (31,2)--(33,2);
\draw (4,2) node{=} (10,2) node{=} (19,2) node{=} (29,2) node{=};
\end{tikzpicture}$$
\end{proof}

\begin{remark}
A suitable version for bicategories of the above construction was observed in both \cite{Carboni2008} and \cite{SpansLWW}. In the latter, they consider a Cartesian bicategory that satifies the Frobenius and the separable axioms, and they build for any 1-cell $F:A\rightarrow X$ a comonoid $G$ on $AX$. Here we did not assume those extra conditions and we showed that we can always build a co-pointed structure on $AX$, which in addition, can be extended to a functor.
\end{remark}

\begin{proposition}
For every Cartesian and fibrant double category, the functor $G$ defined above is right adjoint to the inclusion $K:{\mathbf{Copt} (\mathbb{D})}_0 \rightarrow D_1$.
\end{proposition}
\begin{proof}
From the second version of the known characterization of adjunctions in \cite{MacLane}, it suffices to show that for every co-pointed endomorphism $P:\srarrow{H}{H}$ and cell
$$\xymatrixcolsep{0.2cm} \xymatrixrowsep{0.2cm} \xymatrix{ H \ar[dd]_{h_1} \haru{P} & & H \ar[dd]^{h_2}\\
& \beta\\
A \hard{F} & & X}$$
there is a unique co-pointed morphism
$$\xymatrixcolsep{0.2cm} \xymatrixrowsep{0.2cm} \xymatrix{ H \ar[dd]_{\langle h_1, h_2 \rangle } \haru{P} & & H \ar[dd]^{\langle h_1, h_2 \rangle }\\
& \alpha \\
AX \hard{G(F)} & & AX}$$
such that
$$\begin{tikzpicture}[scale=0.4]
\draw (0,0) node{$\beta$} (1,0) node{=} (3,0) node{$\gamma_F$} (3,2) node{$\alpha$} (3,-2) node{$\pi_2$.} (2,1)--(4,1) (2,-1)--(4,-1);
\end{tikzpicture}$$
Consider the cells
$$\xymatrixcolsep{0.2cm} \xymatrixrowsep{0.2cm} \xymatrix{
H \ar@{=}[dd] \haru{P} & & H \ar@{=}[dd] & & H \ar@{=}[dd] \haru{P} & & H \ar@{=}[dd]\\
& \epsilon_P & & & & \epsilon_P\\
H \ar[dd]_{h_1} \haru{U_H} & & H \ar[dd]^{h_1} & \text{ and } & H \ar[dd]_{h_2} \haru{U_H} & & H \ar[dd]^{h_2}\\
& U_{h_1} & & & & U_{h_2}\\
A \hard{U} & & A & & X \hard{U_X} & & X.}$$
Then we have the cell
$$\xymatrixcolsep{0.2cm} \xymatrixrowsep{0.2cm} \xymatrix{
H \ar[dd]_{\langle h_1, h_1, h_2 \rangle} \haru{P} & & H \ar[dd]^{\langle h_1, h_2, h_2 \rangle}\\
& \langle \frac{\epsilon_P}{U_{h_1}}, \beta, \frac{\epsilon_P}{U_{h_2}} \rangle\\
AAX \hard{AFX} & & AXX.}$$
So by the universal property of $G(F)$, there is a unique
$$\xymatrixcolsep{0.2cm} \xymatrixrowsep{0.2cm} \xymatrix{
H \ar[dd]_{\langle h_1,h_2\rangle} \haru{P} & & H \ar[dd]^{\langle h_1, h_2 \rangle}\\
& \alpha \\
AX \hard{G(F)} & & AX}$$
such that
$$\begin{tikzpicture}[scale=0.4]
\draw (0,0) node{$\Big{\langle}$} (0.4,0)--(2.4,0) (1.4,1) node{$\epsilon_P$} (1.2,-1) node{$U_{h_1}$} (2.8,0) node{,} (3.6,0) node{$\beta$} (4.4,0) node{,};
\draw (4.8,0)--(6.8,0) (5.8,1) node{$\epsilon_P$} (5.8,-1) node{$U_{h_2}$} (7.2,0) node{$\Big{\rangle}$};
\draw (8,0) node{=} (9,0)--(11,0) (10,1) node{$\alpha$} (10,-1) node{$\gamma_F$};
\draw (13,0) node{ or };
\draw (15,0) node{$\beta$} (16,0) node{=} (18,0) node{$\gamma_F$} (18,2) node{$\alpha$} (18,-2) node{$\pi_2$} (17,1)--(19,1) (17,-1)--(19,-1);
\end{tikzpicture}$$
The pair $(\langle h_1,h_2 \rangle, \alpha)$ is a co-pointed morphism since we have
$$\begin{tikzpicture}[scale=0.4]
\draw (0,0)--(2,0) (1,1) node{$\alpha$} (1,-1) node{$\epsilon_{G(F)}$} (3,0) node{=};
\draw (4,1)--(6,1) (4,-1)--(6,-1) (5,0) node{$\gamma_F$} (5,2) node{$\alpha$} (5,-2) node{$\pi_{1,3}$} (7,0) node{=};
\draw (8.2,0)--(10.2,0) (10.8,0)--(12.8,0) (9.2,1) node{$\epsilon_P$} (9.2,-1) node{$U_{h_1}$} (11.8,1) node{$\epsilon_P$} (11.8,-1) node{$U_{h_2}$};
\draw (7.8,0) node{$\Big{\langle}$} (13.2,0) node{$\Big{\rangle}$} (10.5,0) node{,};
\draw (14,0) node{=} (15,0)--(17,0) (16,1) node{$\epsilon_P$} (16,-1) node{$U_{\langle h_1, h_2 \rangle.}$};
\end{tikzpicture}$$
It remains to show that $\alpha$ is unique with the required property. For every cell of the form
$$\xymatrixcolsep{0.2cm} \xymatrixrowsep{0.2cm} \xymatrix{ H \ar[dd]_{f'} \haru{P} & & H \ar[dd]^{f'}\\
& \alpha' \\
AX \hard{G(F)} & & AX}$$
that satisfies
$$\begin{tikzpicture}[scale=0.4]
\draw (0,0)--(2,0) (3,0) node{=} (4,0)--(6,0) (1,1) node{$\alpha'$} (1,-1) node{$\epsilon_{G(F)}$} (5,1) node{$\epsilon_P$} (5,-1) node{$U_{f'}$};
\draw (7.5,0) node{ and };
\draw (9,0) node{$\beta$} (10,0) node{=} (12,0) node{$\gamma_F$} (12,2) node{$\alpha'$} (12,-2) node{$\pi_2$} (11,1)--(13,1) (11,-1)--(13,-1);
\end{tikzpicture}$$
we have that $f'=\langle h_1,h_2 \rangle$ by the second property and that
$$\begin{tikzpicture}[scale=0.4]
\draw (4,1)--(6,1) (4,-1)--(6,-1) (5,0) node{$\gamma_F$} (5,2) node{$\alpha'$} (5,-2) node{$\pi_{1,3}$} (7,0) node{=};
\draw(8,0)--(10,0) (9,1) node{$\epsilon_P$} (9,-1) node{$U_{\langle h_1, h_2 \rangle.}$};
\end{tikzpicture}$$
by the first one. So
$$\begin{tikzpicture}[scale=0.4]
\draw (-1.2,0) node{=} (-2.2,0)--(-4.2,0) (-3.2,1) node{$\alpha'$} (-3.2,-1) node{$\gamma_F$};
\draw (0,0) node{$\langle$} (0.4,0)--(2.4,0) (1.4,1) node{$\epsilon_P$} (1.2,-1) node{$U_{h_1}$} (2.8,0) node{,} (3.6,0) node{$\beta$} (4.4,0) node{,};
\draw (4.8,0)--(6.8,0) (5.8,1) node{$\epsilon_P$} (5.8,-1) node{$U_{h_2}$} (7.2,0) node{$\rangle$};
\draw (8,0) node{=} (9,0)--(11,0) (10,1) node{$\alpha$} (10,-1) node{$\gamma_F$};
\end{tikzpicture}$$
and by the universal property of $\gamma_F$, $\alpha'=\alpha$.\par
In particular, the unit of the adjunction is given by the unique cell
$$\xymatrixcolsep{0.2cm} \xymatrixrowsep{0.2cm} \xymatrix{
H \ar[dd]_d \haru{P} & & H \ar[dd]^d\\
& \alpha \\
HH \hard{G(P)} & & HH}$$
for which
$$\begin{tikzpicture}[scale=0.4]
\draw (0,0) node{$1_P$} (1,0) node{=} (3,0) node{$\gamma_P$} (3,2) node{$\alpha$} (3,-2) node{$\pi_2$.} (2,1)--(4,1) (2,-1)--(4,-1);
\end{tikzpicture}$$
\end{proof}

\begin{corollary}
If $\mathbb{D}$ is a Cartesian and fibrant double category, then $\mathbf{Copt}_0(\mathbb{D})$ has finite products, and the functor $G$ defined above preserves them.
\end{corollary}
\begin{proof}
The products in $\mathbf{Copt}_0(\mathbb{D})$ are given by $(PR:\srarrow{AB}{AB}, \epsilon_P \times \epsilon_R)$ for co-pointed arrows $(P:\srarrow{A}{A}, \epsilon_P)$ and $(R:\srarrow{B}{B}, \epsilon_R)$. That $G$ preserves them follows by the fact that it is a right adjoint.
\end{proof}

\begin{lemma}\label{F_iso_GF}
If $\mathbb{D}$ is a Cartesian and fibrant double category, then for any horizontal arrow $F$, $F$ is isomorphic to the horizontal arrow
$$\xymatrixcolsep{0.2cm} \xymatrixrowsep{0.2cm} \xymatrix{
A \haru{{p_1}^*} & & AX \haru{G(F)} & & AX \haru{{p_2}_*} & & X}$$
in $\mathcal{H(\mathbb{D})}.$
\end{lemma}
\begin{proof}
We have:
\begin{align*} 
{p_2}_* \odot G(F) \odot {p_1}^*  &\cong  {p_2}_* \odot (A\times d)^* \odot AFX \odot (d\times X)_* \odot {p_1}^* && (\text{definition of } G(F)) \\ 
&\cong  {p_2}_* \odot A\times d^* \odot AFX \odot (d\times X)_* \odot {p_1}^* && (\text{\Cref{product_of_Cartesians}}) \\
&\cong {p_2}_* \odot (A\times (d^* \odot FX )) \odot (d\times X)_* \odot {p_1}^* && (\text{functoriality of } \times)\\
&\cong d^* \odot FX \odot {p_2}_* \odot  (d\times X)_* \odot {p_1}^* && (\text{\Cref{horizontal_naturality}})\\
& \cong d^* \odot FX \odot ({p_2} (d\times X))_* \odot {p_1}^* && (\text{\Cref{f^*_functorial}})\\
&\cong d^* \odot FX \odot {1_{AX}}_* \odot {p_1}^* \\
&\cong d^* \odot FX \odot {p_1}^* \\
&\cong d^* \odot {p_1}^* \odot F && (\text{\Cref{horizontal_naturality}})\\
&\cong (p_1 d)^* \odot F && (\text{\Cref{f^*_functorial}})\\
&\cong {1_X}^* \odot F \\
&\cong F
\end{align*}
\end{proof}

As we mentioned in \Cref{Eilenberg_definition}, if a double category admits Eilenberg-Moore objects for co-pointed endomorphisms, then for every co-pointed endomorphism $P$ we have a universal cell
$$\xymatrixcolsep{0.2cm} \xymatrixrowsep{0.2cm} \xymatrix{ EM(P) \haru{U} \ar[dd]_{u} & & EM(P) \ar[dd]^{u} \\ & \theta_P \\ A \hard{P}  & & A}$$
from $U$ to $P$. In a fibrant double category particularly, we can consider the following type of Eilenberg-Moore objects:

\begin{definition}
We say that a fibrant double category admits \textbf{strong Eilenberg-Moore objects for co-pointed endomorphisms} if $P\cong u_* u^*$, where $P$ and $u$ are as in the diagram above.
\end{definition}

\begin{remark}
If we were using comonads instead of co-pointed arrows, then there wouldn't be a need for a similar definition as the one above. In fact, it is well known by \cite{Street1972} that the Eilenberg-Moore objects for comonads are always strong in the above sense. In the case of co-pointed arrows however, we do not have enough structure to prove something similar.
\end{remark}

\begin{corollary}\label{Cart_fibrant_EM_give tab}
If $\mathbb{D}$ is a Cartesian and fibrant double category with Eilenberg-Moore objects for co-pointed endomorphisms, then it has tabulators.

If moreover the Eilenberg-Moore objects are strong, then the tabulators are strong as well.
\end{corollary}
\begin{proof}
As mentioned above, if $EM:\mathbf{Copt}(\mathbb{D}) \rightarrow \mathbb{D}$ is the double functor that assigns Eilenberg-Moore objects to co-pointed arrows, then the composite $EM_0 \circ G: D_1 \rightarrow D_0$ is right adjoint to the functor $U:D_0 \rightarrow D_1$, i.e. it assigns tabulators to horizontal arrows. In particular, for every horizontal arrow $F:\srarrow{A}{X}$, the cell
$$\xymatrixcolsep{0.2cm} \xymatrixrowsep{0.2cm} 
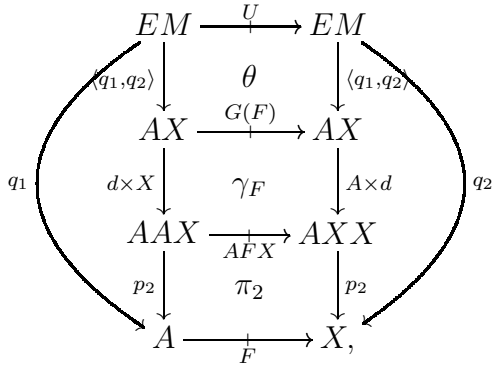
\xymatrix{
EM \ar@/_4pc/[dddddd]_{q_1} \ar[dd]_{\langle q_1, q_2 \rangle} \haru{U} & & EM \ar@/^4pc/[dddddd]^{q_2} \ar[dd]^{\langle q_1,q_2\rangle}\\
& \theta\\
AX \ar[dd]_{d \times X} \haru{G(F)} & & AX \ar[dd]^{A\times d}\\
& \gamma_F\\
AAX \ar[dd]_{p_2} \hard{AFX} & & AXX \ar[dd]^{p_2}\\
& \pi_2\\
A \hard{F} & & X,}$$
provides $F$ with a tabulator.

To show that the tabulators are strong we need to show that ${q_2}_* \odot {q_1}^* \cong F$, which is true since:
\begin{align*} 
{q_2}_* \odot {q_1}^* &\cong (p_2 \langle q_1, q_2 \rangle)_* \odot (p_1 \langle q_1, q_2 \rangle)^* \\
&\cong {p_2}_* \odot \langle q_1, q_2 \rangle_* \odot \langle q_1, q_2 \rangle^* \odot {p_1}^* &&  (\text{\Cref{f^*_functorial}})\\
&\cong {p_2}_* \odot G(F) \odot {p_1}^* && (\text{strong EM objects})\\
&\cong F && (\text{\Cref{F_iso_GF}})
\end{align*}
\end{proof}

\section{The functor $C: \mathrm{\mathbb{S}pan}(D_0) \rightarrow \mathbb{D}$ and the Beck-Chevalley condition} \label{The functorC and the Beck-Chevalley condition}

In this section we build a functor $C: \mathrm{\mathbb{S}pan}(D_0) \rightarrow \mathbb{D}$ which will be our base for the characterization theorem in the following section. This is not the first time that this construction appears. It can be found in Niefield's \cite{SpanCospan}, and in \cite{SpansLWW} in its bicategorical version. Niefield showed that in every fibrant double category that has pullbacks for its vertical arrows, we can build an oplax and normal $C: \mathrm{\mathbb{S}pan}(D_0) \rightarrow \mathbb{D}$. We will consider here the Beck-Chevalley condition because this is exactly what we need for this $C$ to be pseudo. A more general notion of the Beck-Chevalley condition for double categories has been studied by Koudenburg in \cite{Roald}.

Consider a pullback of vertical arrows
$$\xymatrix{ D \ar[d]_{q_1} \ar[r]^-{q_2} & B \ar[d]^g \\ A \ar[r]_f & C }$$
in a fibrant double category $\mathbb{D}$. Then we have the cell
$$\xymatrixcolsep{0.2cm} \xymatrixrowsep{0.2cm} \xymatrix{
A \haru{q_1^*} \ar@{=}[dd] & & D \ar[dd]_{q_1} \haru{U_D} & & D \ar[dd]^{q_2} \haru{{q_2}_*} & & B \ar@{=}[dd] \\
\\
 A \ar@{=}[dd] \haru{U_A} & & A \ar[dd]_f & U & B \ar[dd]^g \haru{U_B} & & B \ar@{=}[dd] \\
\\
A \hard{f_*} & & C \hard{U_C} & & C \hard{g^*} & & B }$$
which we will call $\mathcal{Y}(q_1,q_2,f,g)$.

\begin{definition}\label{Beck_Chev}
We say that a fibrant double category $\mathbb{D}$ satisfies the \textbf{Beck-Chevalley condition} if for every pullback of vertical arrows as above, the cell $\mathcal{Y}(q_1,q_2,f,g)$ is invertible.
\end{definition}

\begin{example}
The double category $\mathrm{\mathbb{S}pan}\mathcal{E}$, for $\mathcal{E}$ a category with pullbacks and terminal object, satisfies the Beck-Chevalley condition. Indeed, if
$$\xymatrix{ D \ar[d]_{q_1} \ar[r]^{q_2} & B \ar[d]^g \\
A \ar[r]_f & C }$$
is a pullback square in $\mathcal{E}$, then the cell $\mathcal{Y}(q_1,q_2,f,g)$ turns out to be
$$\xymatrixcolsep{0.2cm} \xymatrixrowsep{0.2cm} \xymatrix{ & D \ar[dl]_{q_1} \ar[dr]^{q_2} \ar[dddd]^1 \\
A \ar@{=}[dd] & & B \ar@{=}[dd] \\
\\
A & & B \\
& D \ar[ul]^{q_1} \ar[ur]_{q_2} }$$
which is invertible.
\end{example}

\begin{proposition}\label{tab_give_pullbacks}
Consider a unit-pure and fibrant double category $\mathbb{D}$ that has tabulators. Then the category $D_0$ has pullbacks.\\
If moreover the tabulators are strong, then $\mathbb{D}$ satisfies the Beck-Chevalley condition.
\end{proposition}
\begin{proof}
For a pair of vertical arrows 
$$\xymatrix{& B \ar[d]^g \\ A \ar[r]_f & C }$$
consider the tabulator of the horizontal arrow $g^* f_*$:
$$\xymatrixcolsep{0.2cm} \xymatrixrowsep{0.2cm} \xymatrix{
T \ar[dd]_{q_1} \ar[rrrr]^{U_T} \ar[rrrr]|{\vstretch{0.60}{|}} & & & & T \ar[dd]^{q_2}\\
& & \iota\\
A \hard{f_*} & & C \hard{g^*} & & B.}$$
We claim that the square
$$\xymatrix{T \ar[d]_{q_1} \ar[r]^{q_2} & B \ar[d]^g \\
A \ar[r]_f & C }$$
is a pullback square. We have the cell
$$\xymatrixcolsep{0.2cm} \xymatrixrowsep{0.2cm} \xymatrix{
T \ar[dd]_{q_1} \ar[rrrr]^{U_T} \ar[rrrr]|{\vstretch{0.60}{|}} & & & & T \ar[dd]^{q_2}\\
& & \iota\\
A \ar[dd]_f \hard{f_*} & & C \ar@{=}[dd] \hard{g^*} & & B \ar[dd]^g \\
\\
C \hard{U} & & C \hard{U} & & C }$$
and since $U$ is full, $f q_1 = g q_2$. The universal property of the pullback follows by the universal property of the tabulator.\\
To say that the tabulator is strong, means that the cell
$$\xymatrixcolsep{0.2cm} \xymatrixrowsep{0.2cm} \xymatrix{
A \ar@{=}[dd] \haru{{q_1}^*} & & T \ar[dd]_{q_1} \ar[rrrr]^{U_T} \ar[rrrr]|{\vstretch{0.60}{|}} & & & & T \ar[dd]^{q_2} \haru{{q_2}_*} & & B \ar@{=}[dd]\\
& & & & \iota\\
A \hard{U_A} & & A \hard{f_*} & & C \hard{g^*} & & B \hard{U_B} & & B}$$
is invertible. Then the following cell is invertible as well:
$$\xymatrixcolsep{0.2cm} \xymatrixrowsep{0.2cm} \xymatrix{
A \ar@{=}[dd] \haru{{q_1}^*} & & T \ar[dd]_{q_1} \ar[rrrr]^{U_T} \ar[rrrr]|{\vstretch{0.60}{|}} & & & & T \ar[dd]^{q_2} \haru{{q_2}_*} & & B \ar@{=}[dd] & & A \ar@{=}[dd] \haru{{q_1}^*} & & T \ar[dd]_{q_1} \haru{U_T} & & T \ar[dd]^{q_2} \haru{{q_2}_*} & & B \ar@{=}[dd] \\
& & & & \iota\\
A \ar@{=}[dd] \hard{U_A} & & A \ar[dd]_f \hard{f_*} & & C \ar@{=}[dd] \hard{g^*} & & B \ar[dd]^g \hard{U_B} & & B \ar@{=}[dd] & \overset{U\text{ is full}}{=} & A \ar@{=}[dd] \hard{U_A} & & A \ar[dd]_f & U & B \ar[dd]^g \hard{U_B} & & B \ar@{=}[dd] \\
\\
A \hard{f_*} & & C \hard{U_C} & & C \hard{U_C} & & C \hard{g^*} & & B & & A \hard{f_*} & & C \hard{U_C} & & C \hard{g^*} & & B,}$$
i.e. $\mathbb{D}$ satisfies the Beck-Chevalley condition.
\end{proof}

\begin{corollary}\label{Cart_fibrant_EM_give_pullbacks}
If $\mathbb{D}$ is Cartesian, fibrant, unit-pure, and has strong Eilenberg-Moore objects for co-pointed endomorphisms, then $D_0$ has pullbacks that satisfy the Beck-Chevalley condition.
\end{corollary}
\begin{proof}
This follows by \Cref{Cart_fibrant_EM_give tab} and \Cref{tab_give_pullbacks}.
\end{proof}

\begin{proposition}\cite{SpanCospan}
Suppose that $\mathbb{D}$ is a double category and $D_0$ has pullbacks. Then $\mathbb{D}$ is fibrant if and only if the identity functor on $D_0$ extends to an oplax and normal double functor 
$C: \mathrm{\mathbb{S}pan}(D_0) \rightarrow \mathbb{D}$.
\end{proposition}

We give here the construction of the above opnormal double functor $C$, if $\mathbb{D}$ is fibrant.

 If $\xymatrixcolsep{0.4cm} \xymatrix{A & R \ar[l]_{r_1} \ar[r]^{r_2} & B}$ is a span in $D_0$, then define its image to be the op-Cartesian filling ${r_2}_* {r_1}^*$ as in
$$\xymatrixcolsep{0.2cm} \xymatrixrowsep{0.2cm} \xymatrix{ R \haru{U_R} \ar[dd]_{r_1} & & R \ar[dd]^{r_2} \\ & \zeta_R \\ A \hard{{r_2}_* {r_1}^*} & & B. }$$
Consider now a cell
$$\xymatrixcolsep{0.2cm} \xymatrixrowsep{0.2cm} 
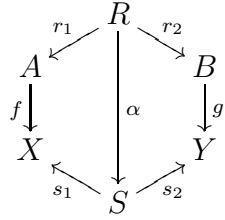
\xymatrix{
& & R \ar[dll]_{r_1} \ar[drr]^{r_2} \ar[dddd]^{\alpha} \\
A \ar[dd]_f & & & & B \ar[dd]^g \\
\\ X & & & & Y \\
& & S \ar[ull]^{s_1} \ar[urr]_{s_2} }$$
in $\mathrm{\mathbb{S}pan}(D_0)$. Then the cell
$$\xymatrixcolsep{0.2cm} \xymatrixrowsep{0.2cm} 
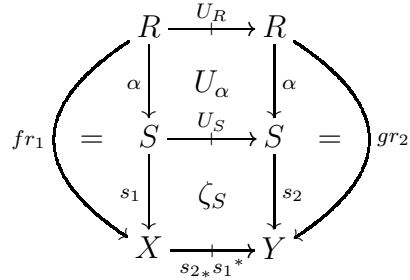
\xymatrix{
& R \ar[dd]_{\alpha} \ar@/_3pc/[dddd]_{f r_1} \haru{U_R} & & R \ar[dd]^{\alpha} \ar@/^3pc/[dddd]^{g r_2} \\
& & U_{\alpha} \\
= & S \ar[dd]_{s_1} \haru{U_S} & & S \ar[dd]^{s_2} & = \\
& & \zeta_S \\
& X \hard{{s_2}_* {s_1}^*} & & Y }$$
can be factored through the op-Cartesian cell $\zeta_R$ as
$$\xymatrixcolsep{0.2cm} \xymatrixrowsep{0.2cm} 
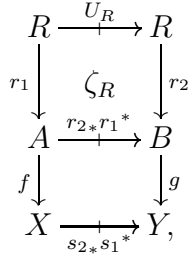
\xymatrix{
R \ar[dd]_{r_1} \haru{U_R} & & R \ar[dd]^{r_2} \\
& \zeta_R \\
A \ar[dd]_{f} \haru{{r_2}_* {r_1}^*} & & B \ar[dd]^{g} \\
 \\
X \hard{{s_2}_* {s_1}^*} & & Y, }$$
where the bottom cell will be our $C(\alpha)$.\\

\begin{corollary}
If $D_0$ has pullbacks, $\mathbb{D}$ is fibrant and satisfies the Beck-Chevalley condition, then the opnormal double functor $C$ is pseudo.
\end{corollary}
\begin{proof}
For every pair of composable spans
$$\xymatrixcolsep{0.2cm} \xymatrixrowsep{0.2cm} 
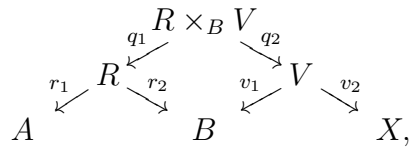
\xymatrix{ & & R\times_B V \ar[dl]_-{q_1} \ar[dr]^-{q_2} \\
& R \ar[dl]_-{r_1} \ar[dr]^-{r_2} & & V \ar[dl]_-{v_1} \ar[dr]^-{v_2} \\
\quad A \quad & & \quad B \quad & & \quad X ,\quad}$$
the natural transformation $C_{\odot}$ is given by the cell
$$(v_2 q_2)_* (r_1 q_1)^* \cong {v_2}_* {q_2}_* {q_1}^* {r_1}^* \xrightarrow{{v_2}_* \mathcal{Y}(q_1,q_2,f,g) {r_1}^*} {v_2}_* {v_1}^* {r_2}_* {r_1}^*,$$
where the isomorphism holds by \Cref{f^*_functorial} and $\mathcal{Y}(q_1,q_2,f,g)$ is the cell defined in the Beck-Chevalley condition. If $\mathbb{D}$ satisfies this condition, then $\mathcal{Y}(q_1,q_2,f,g)$ is invertible and so $C_{\odot}$ is invertible as well.
\end{proof}

\section{A Characterization of Spans as Cartesian Double Categories} \label{A Characterization of Spans as Cartesian Double Categories}

We now have all the necessary tools to prove our characterization theorem. We will consider the functor $C$ from the previous section and we will give conditions under which this functor is an equivalence. First though we need the following lemma:

\begin{lemma}\label{tabulators_of_spans}
If $\mathbb{D}$ is a Cartesian, fibrant, and unit-pure double category which has strong Eilenberg-Moore objects for copointed endomorphisms, then for every span
$\xymatrixcolsep{0.4cm} \xymatrix{A & R \ar[l]_{r_1} \ar[r]^{r_2} & X}$
of vertical arrows, the tabulator of ${r_2}_* {r_1}^*$ is given by $R$ itself, and the vertical arrows $r_1$ and $r_2$.
\end{lemma}
\begin{proof}
First we show that $G({r_2}_* {r_1}^*)$, where $G$ is the functor that was defined in \Cref{functor_G}, is isomorphic to $\langle r_1,r_2 \rangle_* \langle r_1, r_2 \rangle^*$. From the definition of $G$ we have that
\begin{align*}
G({r_2}_* {r_1}^*) &\cong (A \times d)^* \odot (A \times ({r_2}_* {r_1}^*) \times X) \odot (d \times X)_*  \\
&\cong (A \times d)^* \odot A {r_2}_* X \odot A {r_1}^* X \odot (d \times X)_* \\
&\cong (A\times d)^* \odot (A \times r_2 \times X)_* \odot (A \times r_1 \times X)^* \odot (d\times X)_*.
\end{align*}
Consider the following two pullback squares of vertical arrows:
$$\xymatrixcolsep{1cm} \xymatrix{
RX \ar[d]_{r_1 \times X} \ar[r]^{\langle r_1,R \rangle \times X} & ARX \ar[d]^{A \times r_1 \times X} & & AR \ar[d]_{A \times \langle A, r_2 \rangle} \ar[r]^{A \times r_2} & AX \ar[d]^{A \times d}\\
AX \ar[r]_{d \times X} & AAX & \text{and} & ARX \ar[r]_{A \times r_2 \times X} & AXX.}$$
By \Cref{Cart_fibrant_EM_give_pullbacks} $\mathbb{D}$ satisfies the Beck-Chevalley condition. So we have the canonical isomorphisms
$$(A \times r_1 \times X)^* \odot (d\times X)_* \cong (\langle r_1, R \rangle \times X )_* \odot (r_1 \times X)^*$$
and
$$(A\times d)^* \odot (A \times r_2 \times X)_* \cong (A r_2)_* \odot (A \times \langle R, r_2 \rangle)^*.$$
So
$$G({r_2}_* {r_1}^*) \cong (A r_2)_* \odot (A \times \langle R, r_2 \rangle)^* \odot (\langle r_1, R \rangle \times X )_* \odot (r_1 \times X)^*.$$
Also, from the pullback
$$\xymatrixcolsep{1.2cm} \xymatrix{
R \ar[d]_{\langle R, r_2 \rangle} \ar[r]^{\langle r_1, R \rangle} & AR \ar[d]^{A \times \langle R, r_2 \rangle}\\
RX \ar[r]_{\langle r_1, R \rangle \times X} & ARX,}$$
we take
$$(A \times \langle R, r_2 \rangle)^* \odot (\langle r_1, R \rangle \times X )_* \cong \langle r_1, R \rangle_* \odot \langle R, r_2 \rangle^*.$$
Then
\begin{align*}
G({r_2}_* {r_1}^*) &\cong (A r_2)_* \odot \langle r_1, R \rangle_* \odot \langle R, r_2 \rangle^* \odot (r_1 \times X)^* \\
&\cong ( (A r_2) \langle r_1, R \rangle)_* \odot ((r_1 \times X)\langle R, r_2 \rangle)^* \\
&\cong \langle r_1, r_2 \rangle_* \odot \langle r_1, r_2 \rangle^*.
\end{align*}
Now, the tabulator of ${r_2}_* {r_1}^*$ is the Eilenberg-Moore object of $G({r_2}_* {r_1}^*)$, which is $R$, as one can see from the discussion on Eilenberg-Moore objects for spans in \Cref{Eilenberg-Moore Objects}.
\end{proof}

In the following, we will use the definition of an equivalence between two double categories that was given in \cite{AdjointsGP}, according to which a double functor $C$ is an equivalence if and only if it is fully faithful (meaning that both $C_0$ and $C_1$ are fully faithful), and essentially surjective on horizontal arrows.

\begin{theorem}\label{characterization}
Consider a double category $\mathbb{D}$. The following are equivalent:
\begin{enumerate}
\item $\mathbb{D}$ is of the form $\mathbb{S}\mathrm{pan}(\mathcal{E})$, for some category $\mathcal{E}$ with finite limits.
\item $\mathbb{D}$ is Cartesian, fibrant, unit-pure, and has strong Eilenberg-Moore objects for copointed endomorphisms.
\item $D_0$ has pullbacks that satisfy the Beck-Chevalley condition, and the canonical double functor $C: \mathrm{\mathbb{S}pan}(D_0) \rightarrow \mathbb{D}$ is an equivalence.
\end{enumerate}
\end{theorem}
\begin{proof}
$\mathit{(1 \Rightarrow 2)}$ We have proved that the double category $\mathbb{S}\mathrm{pan}(\mathcal{E})$ is Cartesian, fibrant, unit-pure, and has Eilenberg-Moore objects for copointed endomorphisms. From the proof of \Cref{Eilenberg_Moore_Spans}, we can also see that the Eilenberg-Moore objects are strong.

$\mathit{(2 \Rightarrow 3)}$ From \Cref{tab_give_pullbacks}, $D_0$ has pullbacks that satisfy the Beck-Chevalley condition. Also, the functor $C_0$ is just the identity, so it remains to show that the functor $C_1$ is an equivalence as well, in a compatible way with the sources and the targets.\par
We first show that $C_1$ is essentially surjective. For a horizontal arrow $F$ and its tabulator
$$\xymatrixcolsep{0.2cm} \xymatrixrowsep{0.2cm} \xymatrix{ 
T \ar[dd]_{q_1} \haru{U} & & T \ar[dd]^{q_2} \\
& \iota \\
A \hard{F} & & X,}$$
we can consider the span in $D_0$
$$\xymatrixcolsep{0.2cm} \xymatrixrowsep{0.2cm} \xymatrix{& T \ar[dl]_-{q_1} \ar[dr]^-{q_2} \\
A & & X.}$$
Since by \Cref{Cart_fibrant_EM_give tab}, tabulators are strong under the above assumptions, we have ${q_2}_* {q_1}^* \cong F$.\par
We now show that $C_1$ is full. Consider two spans
$$\xymatrixcolsep{0.2cm} \xymatrixrowsep{0.2cm} \xymatrix{
& & R \ar[dll]_{r_1} \ar[drr]^{r_2} \\
A & & & & B,\\
& & S \ar[dll]^{s_1} \ar[drr]_{s_2} \\
X & & & & Y,}$$
and a cell
$$\xymatrixcolsep{0.2cm} \xymatrixrowsep{0.2cm} \xymatrix{
A \ar[dd]_f \haru{{r_2}_* {r_1}^*} & & X \ar[dd]^g \\
& \alpha \\
B \hard{{s_2}_* {s_1}^*}& & Y.}$$
Then, by \Cref{tabulators_of_spans}, we have the tabulators $R$ and $S$ of ${r_2}_* {r_1}^*$ and ${s_2}_* {s_1}^*$ respectively and so we get a unique vertical arrow $a$ such that the following holds:
$$\xymatrixcolsep{0.2cm} \xymatrixrowsep{0.2cm} \xymatrix{
R\ar[dd]_{r_1} \haru{U} & & R \ar[dd]^{r_2} & & R \ar[dd]_{a} \haru{U} & & R \ar[dd]^{a} \\
& \iota_R & & & & U_a\\
A \ar[dd]_f \haru{{r_2}_* {r_1}^*} & & X \ar[dd]^g & = & S \ar[dd]_{s_1} \haru{U} & & S \ar[dd]^{s_2} \\
& \alpha & & & & \iota_S \\
B \hard{{s_2}_* {s_1}^*}& & Y & & B \hard{{s_2}_* {s_1}^*}& & Y.}$$
Then we have the cell
$$\xymatrixcolsep{0.2cm} \xymatrixrowsep{0.2cm} \xymatrix{
& & R \ar[dll]_{r_1} \ar[drr]^{r_2} \ar[dddd]^{\alpha} \\
A \ar[dd]_f & & & & B \ar[dd]^g \\
\\ X & & & & Y \\
& & S \ar[ull]^{s_1} \ar[urr]_{s_2} }$$
in $\mathrm{\mathbb{S}pan}(D_0)$ and $C(a)=\alpha$.\par
Lastly, we show that $C_1$ is faithful. For spans like above, suppose that we have equal cells
$$\xymatrixcolsep{0.2cm} \xymatrixrowsep{0.2cm} \xymatrix{
A \ar[dd]_f \haru{{r_2}_* {r_1}^*} & & X \ar[dd]^g & & A \ar[dd]_f \haru{{r_2}_* {r_1}^*} & & X \ar[dd]^g \\
& C(\alpha) & & = & & C(\beta) \\
B \hard{{s_2}_* {s_1}^*}& & Y & & B \hard{{s_2}_* {s_1}^*}& & Y.}$$
Then
$$\begin{tikzpicture}[scale=0.4]
\draw (0,0)--(2,0) (4,0)--(6,0) (8,0)--(10,0) (12,0)--(14,0) (3,0) node{=} (7,0) node{=} (11,0) node{=};
\draw (1,1) node{$U_{\alpha}$} (1,-1) node{$\iota_S$} (5,1) node{$\iota_R$} (5,-1) node{$C(\alpha)$} (9,1) node{$\iota_R$} (9,-1) node{$C(\beta)$} (13,1) node{$U_{\beta}$} (13,-1) node{$\iota_S,$};
\end{tikzpicture}$$
and by the universal property of $\iota_S$, we get $\alpha=\beta$.

$\mathit{(3 \Rightarrow 1)}$ Trivial.
\end{proof}

From the construction above we have essentially defined a double functor $S:\mathbb{D}\rightarrow \mathrm{\mathbb{S}pan}(D_0)$, which is identity on objects and vertical arrows, and maps a horizontal arrow $F:\srarrow{A}{X}$ to the span
$$\xymatrixcolsep{0.2cm} \xymatrixrowsep{0.2cm} \xymatrix{& T \ar[dl]_-{q_1} \ar[dr]^-{q_2} \\
A & & X,}$$
given by the tabulator of $F$. This extends to the cells as follows:
Consider a cell
$$\xymatrixcolsep{0.2cm} \xymatrixrowsep{0.2cm} \xymatrix{
A \ar[dd]_{f} \haru{F} & & X \ar[dd]^{g} \\
& \alpha \\
B \hard{G}& & Y.}$$
We define $S(\alpha)$ to be the unique vertical arrow we get from the universal property of the tabulator $T(G)$ of $G$ as in the diagram below:
$$\xymatrixcolsep{0.2cm} \xymatrixrowsep{0.2cm} \xymatrix{
T(F) \ar[dd]_{q_1^F} \haru{U} & & T(F) \ar[dd]^{q_2^F} & & T(F) \ar[dd]_{S(\alpha)} \haru{U} & & T(F) \ar[dd]^{S(\alpha)} \\
& \iota_F & & & & U_{S(\alpha)}\\
A \ar[dd]_f \haru{F} & & X \ar[dd]^g & = & T(G) \ar[dd]_{q_1^G} \haru{U} & & T(G) \ar[dd]^{q_2^G} \\
& \alpha & & & & \iota_G \\
B \hard{G}& & Y & & B \hard{G}& & Y.}$$

\vspace{0.3cm}

In \cite{SpanCospan}, proposition 5.3.3., Niefield showed that for a fibrant double category $\mathbb{D}$ with tabulators and pullbacks for its vertical arrows, there is an oplax/lax adjunction of the form
$$\xymatrixcolsep{0.2cm} \xymatrixrowsep{0.2cm} \xymatrix{
 \mathrm{\mathbb{S}pan}(D_0) \ar@/^1pc/[rrr]^-{C} & \perp & & \mathbb{D} \ar@/^1pc/[lll]^-{S}},$$
where her $S$ coincides with the double functor $S$ above. Since every equivalence can be refined to an adjoint equivalence by modifying one of the natural isomorphisms, and adjunction data between two functors is unique, we can say that under the assumptions of \Cref{characterization}, there is an adjoint equivalence from $\mathrm{\mathbb{S}pan}(D_0)$ to $\mathbb{D}$.

\vspace{0.3cm}

In \cite{Grandis2017}, Grandis and Par\'e considered double categories $\mathbb{D}$ with tabulators and pullbacks for vertical arrows, without using the condition of $\mathbb{D}$ being fibrant. They showed that we can always build a lax double functor $S$ as the one above, and they gave the following definition:

\begin{definition}\cite{Grandis2017}
A double category $\mathbb{D}$ is called \textbf{span representative} if:
\begin{enumerate}
\item $\mathbb{D}$ has tabulators.
\item $D_0$ has pullbacks.
\item The lax double functor $S:\mathbb{D}\rightarrow \mathrm{\mathbb{S}pan}(D_0)$ is vertically faithful, i.e. both $S_0$ and $S_1$ are faithful.
\end{enumerate}
\end{definition}

\begin{proposition}
If $\mathbb{D}$ is Cartesian, fibrant, unit-pure, and has strong Eilenberg-Moore objects for co-pointed endomorphisms, then it is span representative.
\end{proposition}
\begin{proof}
We saw in \Cref{Cart_fibrant_EM_give tab} and \Cref{Cart_fibrant_EM_give_pullbacks} that, given the assumptions above, $\mathbb{D}$ has tabulators and $D_0$ has pullbacks. That $S$ is vertically faithful follows by our discussion after \Cref{characterization}, since it is the right adjoint of an adjoint equivalence.
\end{proof}

\chapter{Profunctors} \label{Profunctors}

\section{Modules in Double Categories}

In this section we consider the construction of modules over monads in a double category $\mathbb{D}$. Modules here are defined in the usual way, i.e. as arrows with a left and a right action.

\begin{definition}\cite{Shulman2008}
A \textbf{module} from a monad $(A,S)$ to a monad $(B,T)$ is a horizontal arrow $M:\srarrow{A}{B}$ equipped with two globular cells $\rho: M\odot S \rightarrow M$ and $\lambda :T \odot M \rightarrow M$, such that the following hold:
$$\xymatrixcolsep{0.2cm} \xymatrixrowsep{0.2cm} \xymatrix{ A \haru{U_A} \ar@{=}[dd] & & A \haru{M} \ar@{=}[dd] & & B \ar@{=}[dd] \\
& \eta & & 1_M & & & A \ar@{=}[dd] \haru{U_A} & & A \haru{M} & & B \ar@{=}[dd] \\
A \ar@{=}[dd] \haru{S} & & A \haru{M} & & B \ar@{=}[dd] & = & & & r\\
& & \rho & & & & A \ar[rrrr]_M \ar[rrrr]|{\vstretch{0.60}{|}} & & & & B,\\
A \ar[rrrr]_M \ar[rrrr]|{\vstretch{0.60}{|}} & & & & B}$$
$$\xymatrixcolsep{0.2cm} \xymatrixrowsep{0.2cm} \xymatrix{A \ar@{=}[dd] \haru{S} & & A \haru{S} & & A \ar@{=}[dd] \haru{M} & & B \ar@{=}[dd] & & A \ar@{=}[dd] \haru{S} & & A\ar@{=}[dd] \haru{S} & & A \haru{M} & & B \ar@{=}[dd] \\
& & \mu & & & 1_M & & & & 1_S & & & \rho \\
A \ar@{=}[dd] \ar[rrrr]_S \ar[rrrr]|{\vstretch{0.60}{|}} & & & & A \hard{M} & & B \ar@{=}[dd] & = & A \ar@{=}[dd] \hard{S} & & A \ar[rrrr]_M \ar[rrrr]|{\vstretch{0.60}{|}} & & & & B \ar@{=}[dd] \\
& & & \rho & & & & & &  & & \rho\\
A \ar[rrrrrr]_M \ar[rrrrrr]|{\vstretch{0.60}{|}} & & & & & & B & & A \ar[rrrrrr]_M \ar[rrrrrr]|{\vstretch{0.60}{|}} & & & & & & B, }$$
$$\xymatrixcolsep{0.2cm} \xymatrixrowsep{0.2cm} \xymatrix{ A \haru{M} \ar@{=}[dd] & & B \haru{U_B} \ar@{=}[dd] & & B \ar@{=}[dd] \\
& 1_M & & \eta & & & A \ar@{=}[dd] \haru{M} & & B \haru{U_B} & & B \ar@{=}[dd] \\
A \ar@{=}[dd] \haru{M} & & B \haru{T} & & B \ar@{=}[dd] & = & & & l\\
& & \lambda & & & & A \ar[rrrr]_M \ar[rrrr]|{\vstretch{0.60}{|}} & & & & B,\\
A \ar[rrrr]_M \ar[rrrr]|{\vstretch{0.60}{|}} & & & & B }$$
$$\xymatrixcolsep{0.2cm} \xymatrixrowsep{0.2cm} \xymatrix{A \ar@{=}[dd] \haru{M} & & B \ar@{=}[dd] \haru{T} & & B \haru{T} & & B \ar@{=}[dd] & & A \ar@{=}[dd] \haru{M} & & B \haru{T} & & B \ar@{=}[dd] \haru{T} & & B \ar@{=}[dd] \\
& 1_M & & & \mu & & & & & & \lambda & & & 1_T \\
A \ar@{=}[dd] \hard{M} & & B \ar[rrrr]_T \ar[rrrr]|{\vstretch{0.60}{|}} & & & & B \ar@{=}[dd] & = & A \ar@{=}[dd] \ar[rrrr]_M \ar[rrrr]|{\vstretch{0.60}{|}} & & & & B \hard{T} & & B \ar@{=}[dd] \\
& & & \lambda & & & & & &  & & \lambda \\
A \ar[rrrrrr]_M \ar[rrrrrr]|{\vstretch{0.60}{|}} & & & & & & B & & A \ar[rrrrrr]_M \ar[rrrrrr]|{\vstretch{0.60}{|}} & & & & & & B }$$
$$\xymatrixcolsep{0.2cm} \xymatrixrowsep{0.2cm} \xymatrix{A \ar@{=}[dd] \haru{S} & & A \haru{M} & & B \ar@{=}[dd] \haru{T} & & B \ar@{=}[dd] & & A \ar@{=}[dd] \haru{S} & & A\ar@{=}[dd]  \haru{M} & & B \haru{T} & & B \ar@{=}[dd] \\
& & \rho & & & 1_T & & & & 1_S & & & \lambda \\
A \ar@{=}[dd] \ar[rrrr]_M \ar[rrrr]|{\vstretch{0.60}{|}} & & & & B \hard{T} & & B \ar@{=}[dd] & = & A \ar@{=}[dd] \hard{S} & & A \ar[rrrr]_M \ar[rrrr]|{\vstretch{0.60}{|}} & & & & B \ar@{=}[dd]\ar@{=}[dd] \\
& & & \lambda & & & & & &  & & \rho\\
A \ar[rrrrrr]_M \ar[rrrrrr]|{\vstretch{0.60}{|}} & & & & & & B & & A \ar[rrrrrr]_M \ar[rrrrrr]|{\vstretch{0.60}{|}} & & & & & & B, }$$

Let $(f,\phi):(A,S)\rightarrow (C,U)$ and $(g,\psi):(B,T)\rightarrow (D,V)$ be monad morphisms and $M:(A,S)\rightarrow (B,T)$, $N:(C,U)\rightarrow(D,V)$ modules. A $(\phi,\psi)-$\textbf{equivariant map} is a cell
$$\xymatrixcolsep{0.2cm} \xymatrixrowsep{0.2cm} \xymatrix{ A \haru{M} \ar[dd]_f & & B \ar[dd]^g \\ & \alpha \\ C \hard{N} & & D }$$
which is compatible with the actions of the modules, namely the following hold:
$$\xymatrixcolsep{0.2cm} \xymatrixrowsep{0.2cm} \xymatrix{
A \haru{S} \ar[dd]_f & & A \haru{M} \ar[dd]_f & & B \ar[dd]^g & & A \haru{S} \ar@{=}[dd] & & A \haru{M} & & B \ar@{=}[dd] \\
& \phi & & \alpha & & & & & \rho_M \\
C \hard{U} \ar@{=}[dd] & & C \hard{N} & & D \ar@{=}[dd] & = & A \ar[rrrr]_M \ar[rrrr]|{\vstretch{0.60}{|}} \ar[dd]_f & & & & B \ar[dd]^g \\
& & \rho_N & & & & & & \alpha \\
C \ar[rrrr]_N \ar[rrrr]|{\vstretch{0.60}{|}} & & & & D & & C \ar[rrrr]_N \ar[rrrr]|{\vstretch{0.60}{|}} & & & & D }$$
and
$$\xymatrixcolsep{0.2cm} \xymatrixrowsep{0.2cm} \xymatrix{
A \haru{M} \ar[dd]_f & & B \haru{T} \ar[dd]_g & & B \ar[dd]^g & & A \haru{M} \ar@{=}[dd] & & B \haru{T} & & B \ar@{=}[dd] \\
& \alpha & & \psi & & & & & \lambda_M \\
C \hard{N} \ar@{=}[dd] & & D \hard{V} & & D \ar@{=}[dd] & = & A \ar[rrrr]_M \ar[rrrr]|{\vstretch{0.60}{|}} \ar[dd]_f & & & & B \ar[dd]^g \\
& & \lambda_N & & & & & & \alpha \\
C \ar[rrrr]_N \ar[rrrr]|{\vstretch{0.60}{|}} & & & & D & & C \ar[rrrr]_N \ar[rrrr]|{\vstretch{0.60}{|}} & & & & D. }$$
\end{definition}

\begin{theorem}\cite{Shulman2008}
Consider a fibrant double category $\mathbb{D}$ such that every category $\mathcal{H(\mathbb{D})}(A,A')$ has coequalizers and $\odot$ preserves them in both variables. Then the monads in $\mathbb{D}$, together with the monad morphisms as vertical arrows, the modules as horizontal arrows and the equivariant maps as cells, form a fibrant double category $\mathbb{M}\mathbf{od(\mathbb{D})}$.
\end{theorem}

We will write $\mathbf{FbrCat}_{\mathcal{L}}^{\mathcal{Q}}$ and $\mathbf{FbrCat}^{\mathcal{Q}}$ for the full sub-2-categories of $\mathbf{FbrCat}_{\mathcal{L}}$ and $\mathbf{FbrCat}$, respectively, determined by the fibrant double categories $\mathbb{D}$ in which every category $\mathcal{H(\mathbb{D})(A,B)}$ has coequalizers and $\odot$ preserves them in both variables.

\begin{proposition}\cite{Shulman2008} \label{ModShulman}
$\mathbf{Mod}$ defines a 2-functor $\mathbf{FbrCat}_{\mathcal{L}}^{\mathcal{Q}} \rightarrow \mathbf{FbrCat}_{\mathcal{L}}^{\mathcal{Q}}$, which restricts to a 2-functor $\mathbf{FbrCat}^{\mathcal{Q}}\rightarrow \mathbf{FbrCat}^{\mathcal{Q}}$.
\end{proposition}

\begin{proposition}\label{Mod_Cart}
\begin{enumerate}
\item If $\mathbb{D}$ is a fibrant precartesian double category then $\mathbf{Mod(\mathbb{D})}$ is a fibrant precartesian double category too.
\item If $\mathbb{D}$ is a fibrant Cartesian double category then $\mathbf{Mod(\mathbb{D})}$ is a fibrant Cartesian double category too.
\end{enumerate}
\end{proposition}
\begin{proof}
It follows by \Cref{ModShulman}, and the fact that 2-functors preserve adjunctions.
\end{proof}

We prove the following lemma so that we will be able to consider the double category of modules of the double category of monads.

\begin{lemma}
If $\mathbb{D}$ is a double category such that every hom-category $\mathcal{H(\mathbb{D})}(A,A')$ has coequalizers and $\odot$ preserves them in both variables, then every hom-category $\mathcal{H(\mathbf{Mnd(\mathbb{D})})}((A,S),(A,S'))$ has coequalizers and $\odot$ preserves them in both variables.
\end{lemma}
\begin{proof}
Consider two cells $$\alpha, \beta:(F,\alpha_F) \rightarrow (G,\alpha_G): \srarrow{(A,S)}{(A',S')}$$ in $\mathcal{H(\mathbf{Mnd(\mathbb{D})})}$ and their coequalizer
$$\xymatrix{F \ar@<1ex>[r]^{\alpha} \ar@<-1ex>[r]_{\beta} & G \ar[r]^{\gamma} & H}$$
in $\mathcal{H(\mathbb{D})}$. To show that the lemma holds it suffices to show that $H$ is a horizontal comonad map and that $\gamma$ is a cell in $\mathbf{Mnd(\mathbb{D})}$. Since $\odot$ preserves coequalizers, the diagram
$$\xymatrix{F\odot S \ar@<1ex>[r]^{\alpha \odot S} \ar@<-1ex>[r]_{\beta \odot S} & G \odot S \ar[r]^{\gamma \odot S} & H \odot S}$$ is a coequalizer diagram too. We also have
$$\xymatrixcolsep{0.02cm} \xymatrixrowsep{0.2cm} \xymatrix{
A \ar@{=}[dd] \haru{S} & & A \ar@{=}[dd] \haru{F} & & A' \ar@{=}[dd] & & A \ar@{=}[dd] \haru{S} & & A \haru{F} & & A' \ar@{=}[dd] & & A \ar@{=}[dd] \haru{S} & & A \haru{F} & & A' \ar@{=}[dd] & & A \ar@{=}[dd] \haru{S} & & A \ar@{=}[dd] \haru{F} & & A' \ar@{=}[dd] \\
& 1 & & \alpha & & & & & \alpha_F & & & & & & \alpha_F & & & & & 1 & & \beta \\
A \ar@{=}[dd] \haru{S} & & A \haru{G} & & A' \ar@{=}[dd] & & A \ar@{=}[dd] \haru{F} & & A' \ar@{=}[dd] \haru{S'} & & A' \ar@{=}[dd] & & A \ar@{=}[dd] \haru{F} & & A' \ar@{=}[dd] \haru{S'} & & A' \ar@{=}[dd] & & A \ar@{=}[dd] \haru{S} & & A \haru{G} & & A' \ar@{=}[dd] \\
& & \alpha_G & & & = & & \alpha & & 1 & & = & & \beta & & 1 & & = & & & \alpha_G \\
A \ar@{=}[dd] \hard{G} & & A' \ar@{=}[dd] \hard{S'} & & A' \ar@{=}[dd] & & A \ar@{=}[dd] \hard{G} & & A' \ar@{=}[dd] \hard{S'} & & A' \ar@{=}[dd] & & A \ar@{=}[dd] \hard{G} & & A' \ar@{=}[dd] \hard{S'} & & A' \ar@{=}[dd] & & A \ar@{=}[dd] \hard{G} & & A' \ar@{=}[dd] \hard{S'} & & A' \ar@{=}[dd] \\
& \gamma & & 1 & & & & \gamma & & 1 & & & & \gamma & & 1 & & & & \gamma & & 1 \\
A \hard{H} & & A' \hard{S'} & & A' & & A \hard{H} & & A' \hard{S'} & & A' & & A \hard{H} & & A' \hard{S'} & & A' & & A \hard{H} & & A' \hard{S'} & & A'. }$$
So $(S\odot \gamma) \alpha_G (\alpha \odot S) = (S\odot \gamma) \alpha_G (\beta \odot S)$, which means that there is a unique cell $\alpha_H$ in $\mathbb{D}$
$$\xymatrixcolsep{0.2cm} \xymatrixrowsep{0.2cm} \xymatrix{F\odot S \ar@<1ex>[rr]^{\alpha \odot S} \ar@<-1ex>[rr]_{\beta \odot S} & & G \odot S \ar[rr]^{\gamma \odot S} \ar[dr]_-{\alpha_G} & & H \odot S \ar@{-->}[dd]^{\alpha_H} \\
& & & S \odot G \ar[dr]_-{S' \odot \gamma} \\
& & & & S' \odot H }$$
that makes the above diagram commute. We can show that $\alpha_H$ is a horizontal comonad map by using once again the universal property of the coequalizer and the fact that $\alpha_G$ is a horizontal comonad map. Also, the commutativity of the triangle above shows exactly that $\gamma$ is a cell in $\mathbf{Mnd(\mathbb{D})}$ from $(G,\alpha_G)$ to $(H,\alpha_H)$.
\end{proof}

We can now define the double category $\mathbf{Mod(Mnd(\mathbb{D}))}$, and similary we can show that we can define the double category $\mathbf{Mod(Com(\mathbb{D}))}$ as well.

\section{The Double Category of Profunctors}

In this last section we will consider a category $\mathcal{E}$ with finite limits and reflexive coequalizers preserved by pullback functors, and we will define the double category of profunctors internal to $\mathcal{E}$. Its horizontal bicategory is the usual bicategory of internal profunctors as in \cite{JohnstoneT}.

\begin{definition}
The double category of profunctors $\mathbb{P}\mathrm{rof}(\mathcal{E})$ is defined to be the double category $\mathbf{Mod(\mathbb{S}\mathrm{pan}(\mathcal{E})})$.
\end{definition}

The vertical category of $\mathbb{P}\mathrm{rof}(\mathcal{E})$ is exactly the category of internal categories and internal $\mathcal{E}$-valued functors. The horizontal arrows are the internal profunctors in $\mathcal{E}$, and the cells the internal natural transformations, which are defined below. That is, the horizontal bicategory of $\mathbb{P}\mathrm{rof}(\mathcal{E})$ is exactly the bicategory $\mathbf{Prof}(\mathcal{E})$, as in \cite{JohnstoneT}.

\begin{definition}\cite{JohnstoneT}
An \textbf{internal category} $A$ in $\mathcal{E}$ consists of the following data:
\begin{enumerate}[label=\roman*.]
\item an object $A_0$ of $\mathcal{E}$, called the object of objects,
\item an object $A_1$ of $\mathcal{E}$, called the object of arrows,
\item two arrows $s,t:A_1 \rightarrow A_0$ in $\mathcal{E}$, called respectively source and target,
\item an arrow $i:A_0 \rightarrow A_1$ in $\mathcal{E}$ called identity, and
\item an arrow $c : A_1\times_{A_0} A_1 \rightarrow A_1$ in $\mathcal{E}$, called composition, where $A_1\times_{A_0} A_1 $ is the pullback
$$ \xymatrix{ & A_1\times_{A_0} C_1 \ar[r]^{\quad p_2} \ar[d]_{p_1} & A_1 \ar[d]^s \\
& A_1 \ar[r]_t & A_0 &, } $$
\end{enumerate}
subject to the axioms:
\begin{enumerate}
\item $s \circ i = 1_{A_0} = t \circ i$,
\item $s \circ c = s\circ p_1$, $ t\circ c = t\circ p_2$,
\item $ c \circ (c \: \times_{A_0} \: 1_{A_1}) = c \circ (1_{A_1} \: \times_{A_0} \: c)$, and
\item $c \circ (1_{A_1} \: \times_{A_0} \: (i \circ t)) = 1_{A_1} = c \circ ((i \circ s) \: \times_{A_0} \: 1_{A_1})$.
\end{enumerate}

Given an internal category $A$ in $\mathcal{E}$, an \textbf{internal} $\mathcal{E}$\textbf{-valued functor} $P: A \rightarrow \mathcal{E}$ is an object $P_0 \in |\mathcal{E}|$ together with two arrows $p_0:P_0 \rightarrow A_0$ and $p_1:A_1 \: \times_ {A_0} \: P_0 \rightarrow P_0$ of $\mathcal{E}$, where $A_1 \: \times_ {A_0} \: P_0$ is the pullback
$$\xymatrix{ & A_1 \: \times_ {A_0} \: P_0 \ar[d]_{r_1} \ar[r]^{\quad r_2} & P_0 \ar[d]^{p_0} \\
& A_1 \ar[r]_t & A_0 &.}$$
These data are required to satisfy the following axioms:
\begin{enumerate}
\item $p_0 \circ p_1 = s \circ r_1$,
\item $p_1 \circ ( (i\circ p_0) \: \times_{A_0} \: 1_{P_0}) = 1_{P_0}$, and
\item $p_1 \circ (1_{A_1} \: \times_{A_0} \: p_1) = p_1 \circ (c \: \times_{A_0} \: 1_{P_0})$.
\end{enumerate}

Given an internal category $A$ in $\mathcal{E}$ and two internal $\mathcal{E}$-valued functors $P,Q:A \rightarrow \mathcal{E}$, written as $P=(P_0,p_0,p_1)$ and $Q=(Q_0,q_0,q_1)$, an \textbf{internal natural transformation} $\phi :P \Rightarrow Q$ in $\mathcal{E}$ is an arrow $\phi: P_0 \rightarrow Q_0$ in $\mathcal{E}$ which satisfies the following conditions:
\begin{enumerate}
\item $q_0 \circ \phi = p_0$ and
\item $\phi \circ p_1 = q_1 \circ (1_{A_1} \: \times_{A_0} \: \phi)$.
\end{enumerate}

Given two internal categories $A$ and $B$, an \textbf{internal profunctor} $P: A \nrightarrow B$ in $\mathcal{E}$ is an internal $\mathcal{E}$-valued functor $B^{op} \times A \rightarrow \mathcal{E}$. The internal categories in $\mathcal{E}$, together with the internal profunctors and the internal natural transformations between them form a bicategory $\mathbb{P}\mathrm{rof}(\mathcal{E})$.
\end{definition}

\begin{proposition}
The double category $\mathbb{P}\mathrm{rof}(\mathcal{E})$ is Cartesian.
\end{proposition}
\begin{proof}
It follows by \Cref{Mod_Cart}.
\end{proof}

To close this chapter, we would like to give a conjecture for a potential characterization of profunctors. For this we need the Kleisli construction on double categories, which we define below.

\begin{definition}
We say that a double category $\mathbb{D}$ \textbf{admits the construction of Kleisli objects for monads} if the inclusion $J: D_0\rightarrow (\mathbf{Mnd(\mathbb{D})})_0$ has a left adjoint.
\end{definition}

\begin{remark}
To say that a fibrant double category has Kleisli objects, it is to say that for every monad $S:\srarrow{A}{A}$, there is an object $K$ and a cell
$$\xymatrixcolsep{0.2cm} \xymatrixrowsep{0.2cm} \xymatrix{
A \haru{S} \ar[dd]_v & & A \ar[dd]^v \\
\\ K \hard{U} & & K,}$$
universal from $S$ to $U$.
\end{remark}

In \cite{Carboni1987}, an essential part for the characterization of the locally ordered bicategory $\mathbf{Idl}\mathcal{E}$ of ordered objects and ordered ideals is the bicategory of discrete objects in a Cartesian locally ordered bicategory. Since both profunctors and ordered ideals can be seen as modules in suitable bicategories, we believe that there is an analogous characterization for profunctors, that uses the notion of discrete objects. We give the following definition, suitable to double categories:

\begin{definition}
An object in a Cartesian and fibrant double category $\mathbb{D}$ is called \textbf{discrete} if the following two pullbacks satisfy the Beck-Chevalley condition (\Cref{Beck_Chev}):
$$\xymatrix{
A \ar[d]_d \ar[r]^d & A\times A \ar[d]^{d\times A} & & A \ar[d]_{1} \ar[r]^{1} & A \ar[d]^{d} \\
A\times A \ar[r]_-{A\times d} & A \times A \times A & ,& A \ar[r]_-{d} & A\times A.} $$
Define $\mathbb{D}\mathrm{isc(\mathbb{D})}$ to be the double category of discrete objects in $\mathbb{D}$.
\end{definition}

\begin{conjecture}\label{conjecture1}
If $\mathbb{D}$ is a Cartesian and fibrant double category, then the double category $\mathbb{D}\mathrm{isc(\mathbb{D})}$ is unit-pure.
\end{conjecture}

In \cite{CKW}, it was shown that a bicategory $\mathcal{B}$ with local stable coequalizers is equivalent to some bicategory of modules if and only if it has Kleisli objects in the bicategorical sense. The case where every Kleisli object is discrete was also discussed in \cite{Carboni1987}. Inspired by the results in these two papers, we believe that the following statement is true:

\begin{conjecture}\label{conjecture2}
Consider a double category $\mathbb{D}$ in $\mathbf{FbrCat}^{\mathcal{Q}}$ which:
\begin{enumerate}
\item is Cartesian,
\item has Kleisli objects for monads, and
\item the Kleisli object of every monad is discrete.
\end{enumerate}
Then $\mathbb{D} \simeq \mathbf{Mod(\mathbb{D}\mathrm{isc(\mathbb{D})})}$.
\end{conjecture}

\begin{corollary}[of \Cref{conjecture1} and \Cref{conjecture2}]
Consider a double category $\mathbb{D}$ in $\mathbf{FbrCat}^{\mathcal{Q}}$ which:
\begin{enumerate}
\item is Cartesian,
\item has Kleisli objects for monads,
\item the Kleisli object of every monad is discrete, and
\item $\mathbb{D}\mathrm{isc(\mathbb{D})}$ has strong Eilenberg-Moore objects for co-pointed endomorphisms.
\end{enumerate}
Then $\mathbb{D} \simeq \mathbb{P}\mathrm{rof}(\mathbb{D}\mathrm{isc(\mathbb{D})}_0)$.
\end{corollary}
\begin{proof}
From the first three conditions and \Cref{conjecture2} we take that $$\mathbb{D} \simeq \mathbf{Mod(\mathbb{D}\mathrm{isc(\mathbb{D})})}.$$
Also, \Cref{conjecture1} and \emph{4.} show that the double category $\mathbb{D}\mathrm{isc(\mathbb{D})}$ satisfies the conditions of the characterization \Cref{characterization}, and then $$\mathbb{D}\mathrm{isc(\mathbb{D})} \simeq \mathbb{S}\mathrm{pan}(\mathbb{D}\mathrm{isc(\mathbb{D})}_0).$$
It follows that
$$\mathbb{D} \simeq \mathbf{Mod(\mathbb{D}\mathrm{isc(\mathbb{D})})} \simeq \mathbf{Mod(\mathbb{S}\mathrm{pan}(\mathbb{D}\mathrm{isc(\mathbb{D})}_\mathrm{0}))} \simeq \mathbb{P}\mathrm{rof}(\mathbb{D}\mathrm{isc(\mathbb{D})}_0).$$
\end{proof}

\chapter{Conclusion}\label{Conclusion}

\section{A Comparison Between Cartesian Double Categories and Bicategories}

We can now make a comparison between the two different settings that have been used for the generalization of the classic notion of Cartesian categories. The first one was that of Cartesian bicategories and the second is the one that we used and developed in this thesis, i.e. Cartesian double categories. We repeat the two definitions here for reference:

\vspace{0.3cm}
$\clubsuit$ \emph{A bicategory $\mathcal{B}$ is said to be Cartesian if:
\begin{enumerate}
\item The bicategory $\mathbf{Map}(\mathcal{B})$ has finite bicategorical products.
\item Each category $\mathcal{B}(A,B)$ has finite products.
\item Certain derived lax functors $\otimes: \mathcal{B} \times \mathcal{B} \rightarrow \mathcal{B}$ and $I:\mathbbm{1} \rightarrow \mathcal{B}$, extending the product structure of $\mathbf{Map}(\mathcal{B})$, are pseudo.
\end{enumerate}
}

\vspace{0.3cm}
$\clubsuit$ \emph{A double category $\mathbb{D}$ is said to be Cartesian if there are adjunctions
$$\xymatrixcolsep{0.2cm} \xymatrixrowsep{0.2cm} \xymatrix{
\quad \mathbb{D} \ar@/^1pc/[rr]^-{\Delta} & \perp & \mathbb{D} \times \mathbb{D} \ar@/^1pc/[ll]^-{\times} & \text{and} & \mathbb{D} \ar@/^1pc/[rr]^{!} & \perp & \mathbbm{1} \ar@/^1pc/[ll]^-I }$$
in the 2-category $\mathbf{DblCat}$.}

\vspace{0.3cm}
The first and most evident difference is that for double categories the definition is closer to the definition for Cartesian categories, which asks for the respective right adjoints to the diagonal and unique functors. As we have mentioned before, the reason for that is that double categories, double functors, and vertical natural transformations form a strict 2-category. This is an important property which we believe will play a central role in the further development of the subject.

The second difference is related to the symmetric monoidal structure that arises from the above definitions. A monoidal structure on double categories is a quite simple idea, given the fact that it is just a pair of monoidal structures on mere categories, together with an extra compatibility property (\cite{Shulman2010}). A monoidal bicategory on the other hand is not a very simple idea. The traditional definition is that of an one-object tricategory. The original reference for tricategories is Gordon, Power, and Street's \cite{tricategories}, where one can see the complexity of the definition.

Moreover, in the paper \cite{Carboni2008}, the authors had to use the Grothendieck construction to shift from precartesian to Cartesian bicategories. For double categories, to make the same shift, we just replaced the lax right adjoints to the diagonal and the unique double functors in the definition for pseudo. Also, if we were to consider the Grothendieck construction for the horizontal bicategory of a double category $\mathbb{D}$, then this would correspond to the category $D_1$ (when we don't consider the 2-cells of the Grothendieck construction). $D_1$ though was already given to us when we asked for a double category $\mathbb{D}$.

Another issue one encounters when working with a Cartesian bicategory is its very first condition, the one that asks for finite bicategorical products in $\mathbf{Map}(\mathcal{B})$. For the bicategory of spans this was an easy question since $\mathbf{Map}(\mathbf{Span(\mathcal{E})})$ is just the category $\mathcal{E}$ (\cite{SpansLWW}). However, we saw in Chapter 2 that checking whether $\mathbf{Map}(\mathbf{Prof(\mathcal{E})})$ or $\mathbf{Map}(V\mbox{-}\mathbf{Mat})$ has finite products is not a simple question, let alone that products in the bicategorical sense are not trivial. How we overcome this issue with double categories is that we just ask for products in the categories $\mathbf{Cat(\mathcal{E})}$ or $\mathbf{Set}$, i.e. we ask for products only for the ``nice'' maps in $\mathbf{Prof(\mathcal{E})}$ or $V\mbox{-}\mathbf{Mat}$ respectively. It has been a very interesting question throughout the years to characterize those ``nice'' maps in general and see how they can be picked out of $\mathbf{Map}(\mathcal{B})$ for a general $\mathcal{B}$. As far as we know, this still remains unknown.

Lastly, we would like to emphasize that many times, calculations in this thesis appeared to be simpler than the calculations for similar results for Cartesian bicategories. For most of the properties that were shown for Cartesian bicategories, a suitable version of them was also shown for double categories. Moreover, we can use both setting as a base for a characterization of spans. The two characterization theorems look quite different, which stresses the distinction between Cartesian bicategories and Cartesian double categories. However, the point here is that we didn't lose anything when we moved from bicategories to double categories, but we were also able to develop it even more, by proving for examble that $\mathbb{M}\mathbf{od(\mathbb{D})}$ is Cartesian if $\mathcal{B}$ is Cartesian, and as a result, $\mathbf{Prof(\mathcal{E})}$ is Cartesian.

\vspace{0.3cm}
One of our main goals for the future is to show that we can also characterize the double category of profunctors as a Cartesian double category. We gave a guideline for that in our last chapter. We also believe that there is a lot to be said and studied for Cartesian double categories in general, as well as investigating possible generalizations of usual properties of Cartesian categories. For example, Heunen and Vicary showed in their lecture notes \cite{Vicary} that if $\mathcal{C}$ is a symmetric monoidal category equipped with monoidal natural transformations with components $\delta_A: A\rightarrow A\otimes A$ and $\epsilon_A: A \rightarrow I$, subject to specific properties, then $\mathcal{C}$ is a Cartesian category. Would a generalization of that be possible for Cartesian double categories?

\vspace{0.3cm}
Throughout the years, category theorists seemed to favor 2-categories or bicategories for the development of Category Theory in higher dimensions. Recently though, double categories have gained significant ground in the discussion. We hope that what we presented in this thesis will be a helpful tool in this regard.

\bibliographystyle{alpha}
\bibliography{Thesis}

\end{document}